\documentclass[12pt]{amsart}
\usepackage{amsfonts,amssymb,eucal}
\newcommand{\GM}{\ensuremath\tilde{{\Gamma}}}
\usepackage{amsthm} 
 \usepackage{amsmath} 
 \usepackage{latexsym}
\usepackage{bbold,mathbbol,bbm}
\numberwithin{equation}{section}
\newcommand{\ep}{\ensuremath{\epsilon\,}}

\newcommand{\e}{{\mbox{\rm e}}}

\newcommand{\mb}[1]{{\mbox{\boldmath{$#1$}}}}
\newcommand{\mc}[1]{{\mathcal{#1}}}
\newcommand{\got}[1]{{\mathfrak{#1}}}
\newcommand{\db}[1]{{\mathbb{#1}}}
\newcommand{\pa}{\partial}

\newcommand{\R}{\ensuremath{\mathbb{R}}}\newcommand{\C}{\ensuremath{\mathbb{C}}}
 \newcommand{\N}{\ensuremath{\mathbb{N}}}

\newcommand{\Z}{\ensuremath{\mathbb{Z}}}
\newcommand{\g}{\ensuremath{\got{g}}}
\newtheorem{Remark}{Remark}

\newtheorem{Theorem}{Theorem}
\newtheorem{Proposition}{Proposition}
\newtheorem{lemma}{Lemma}

\newtheorem{Comment}{Comment}
\theoremstyle{definition} 
\def\ii{\operatorname{i}}
\newcommand{\Ka}{K\"ahler}
\newcommand{\mr}[1]{{\mathrm{#1}}}
\newcommand{\dd}{\operatorname{d}}

\newcommand{\un}{\ensuremath{\mathbb{1}_n}}

\newtheorem{deff}{Definition}
\newcommand{\h}{\ensuremath{\got{h}}}
\renewcommand{\Re}{\operatorname{Re}}
\renewcommand{\Im}{\operatorname{Im}}
\newcommand{\tr}{\operatorname{tr}}
\allowdisplaybreaks 
\newcommand{\Sp}{\ensuremath{{\mbox{\rm{Sp}}(n,\R)}}}
\begin{document}
\title[Berry phase]{Berry phases and connection matrices defined  on
 homogeneous spaces attached to Jacobi groups}
\enlargethispage{1cm}
\author{Stefan  Berceanu}
\address[Stefan  Berceanu]{`Horia Hulubei'' `Horia Hulubei'' National
 Institute for Physics and Nuclear Engineering\\
Department of Theoretical Physics\\
PO BOX MG-6, Bucharest-Magurele, Romania}
 \email{Berceanu@theory.nipne.ro}
\begin{abstract}
The relation between the Berry phase and connection matrix  on the Siegel-Jacobi
disk $\mc{D}^J_1$  and Siegel-Jacobi upper half-plane $\mc{X}^J_1$ are analyzed.
The connection matrix and the covariant derivative of one-forms on the
extended Siegel-Jacobi upper  half-plane  $\tilde{\mc{X}}^J_1$ are calculated.
\end{abstract}
\subjclass{81Q70,53C05,37J55,53D15}      
\keywords{Berry phase, Berry connection,  Connection matrix,  Jacobi group, invariant metric,
Siegel--Jacobi disk, Siegel--Jacobi upper half-plane,  extended Siegel--Jacobi upper half-plane,
almost cosymplectic manifold}
\maketitle
\tableofcontents
\newpage

\section{Introduction} The complex  Jacobi group \cite{ez,bs} of index $n$ is defined as the
semi-direct product $G^J_n=\mr{H}_n\rtimes\text{Sp}(n,\R)_{\C}$,
where $\mr{H}_n$ denotes the 
$(2n+1)$-dimensional Heisenberg group \cite{Y02,sbj,nou}. To the Jacobi
group $G^J_n $ it is associated  a homogeneous manifold, called the
Siegel-Jacobi  ball  $\mc{D}^J_n$ \cite{sbj},  whose points are in
$\C^n\times\mc{D}_n$, i.e. a partially-bounded space \cite{Y08,Y10}.   $\mc{D}_n$ denotes the
Siegel (open)  ball  of index $n$. The non-compact Hermitian
symmetric space $ \operatorname{Sp}(n, \R
)_{\C}/\operatorname{U}(n)$ admits a matrix realization  as a  homogeneous bounded
domain \cite{helg}:
\[
  \mc{D}_n:=\{W\in  MS (n, \C ): \un-W\bar{W}>0\}.
  \]
The real Jacobi group of degree $n$ is defined as $G^J_n(\R):={\rm Sp}(n,\R)\ltimes \mr{H}_n$,
where $\mr{H}_n=\mr{H}_n(\R)$ is the real $(2n+1)$-dimensional Heisenberg group. ${\rm Sp}(n,\R)_{\C}$~and~$G^J_n$ are isomorphic to~${\rm Sp}(n,\R)$
and~$G^J_n(\R)$ respectively as real Lie groups, see \cite[Proposition~2]{nou}.

The invariant metric on the Siegel-Jacobi upper half-plane  on 
$\mc{X}^J_1=\frac{G^J_1(\R)}{{\rm SO}(2)\times\R}\approx
\mc{X}_1\times\R^2$ \cite{jac1,BER77,FC,SB14} was obtained previously by
Berndt~\cite{bern84,bern} and K\"ahler~\cite{cal3,cal}.

 We determined the invariant metric on a five
dimensional homogeneous manifold  ${\tilde{\mc{X}}^J_1}=\frac{G^J_1(\R)}{{\rm
    SO}(2)}\approx\mc{X}_1\times\R^3$ \cite{SB19},  called the extended Siegel--Jacobi
upper half-plane. The results of \cite{SB19}
concerning $\tilde{\mc{X}}^J_1$ have been generalized  in \cite{SB20N} to  the extended Siegel-Jacobi
upper half space
$\tilde{\mc{X}}^J_n=\frac{G^J_n(\R)}{\text{U}(n)}\approx
\mc{X}^J_n\times \R$, $\mc{X}^J_n\approx\C^n\times \mc{X}_n  $,
$\mc{X}_n=\frac{\mr{Sp}(n,\R)}{\mr{U}(n)}, $ $\N\ni n>
1$.

 We recall that on homogenous \Ka~ 
 manifolds the Hamilton  equations  of motion and the Berry phase
 were  simultaneously  investigated 
  in  \cite{sbcag,sbl,FC}, see also \cite{GO}. In the present paper
 we are interested in
 the same problem  of studying the  Berry phase on odd-dimensional manifolds, where several
 geometric structures can be introduced
 \cite{BL,BL2,boy,boga,can,Capp,MLEO,LI,sas}, see also a brief review 
 in \cite[Appendix]{SB22}. In our paper
 \cite{SB22} we have investigated Hamiltonian systems  on manifolds with
 almost cosymplectic structure in the sense of \cite{paul}. In the
 present paper we investigate 
the connection matrix on odd dimensional manifolds endowed
 with an almost complex structure.

We recall here our 
 interrest   to find a geometric significance to  the phase of  the scalar product of
 coherent states \cite{perG,lis1,neeb}. The answer to this question was given 
 by Pancharatnam for 
 the Poincar\'e sphere
 \cite{Pan,swA}, see also \cite{aa}  and \cite[Proposition 5.1]{ma}
 in the language of holonomy (see \S~\ref{HOLL}) of a loop in the projective
 Hilbert space,  and  by 
 Perelomov \cite[page 63]{perG} for the  sphere
 $\text{S}^2=\frac{\text{SU}(2)}{\text{U}(1)}$. 
    A general answer  to
this question using the coherent state embedding and the so called
'' Cauchy formulas'' was given in \cite{SB2000} and
\cite{SBS}. We also  studied this problem in \cite{SB99}--\cite{SB2003}. Explicit calculation was presented for the compact
Grassmann manifold
$G_n(\C^{m+n})=\frac{\text{SU}(n+m)}{\text{S}(\text{U}(n)\times
  \text{U}(m))}$  in \cite{SB2000}, where it was proved that {\it  the
phase of the scalar product of two coherent states is twice the
symplectic area of a geodesic triangle determined by the corresponding
points on the manifold and the origin of the system of
coordinates},  see
 also \cite[Theorem 2.1]{clerc}. The same result is also true for the
 noncompact dual
 $\frac{\text{SU}(n,m)}{\text{S}(\text{U}(n)\times
  \text{U}(m))}$ of
the compact Grassmann manifold
\cite{SB99,SB99b}. In \cite{FC}   the change
of coordinates $x\rightarrow z  $ in \eqref{2.1}  below was called 
$FC$-transform  (fundamental conjecture \cite{GV,pia,DN}). We
observed   that \cite[Remark 3]{sbl} 
\[
{\text{   For~symmetric manifolds the  FC-transform   gives
  geodesics\qquad (A)}}
\]
In \cite[Remark 1]{SB97}
we underlined that assertion (A) is true for class of manifolds which
includes the naturally reductive spaces \cite{nomizu,atri},
\cite[page 202]{kn}. We have considered the
sequence of manifolds
\begin{center}
  Hermitian symmetric spaces $\subset$ symmetric $\subset $ naturally
  reductive $\subset$ g. o.
\end{center}
We have shown in \cite[Proposition 5.8]{SB19} that $\mc{X}^J_1$ is not
naturally reductive with respect to the balanced metric \cite{don,arr,alo}.
In \cite[Theorem 1]{SB21} we have proved that the extended Siegel-Jacobi
upper half-plane, realized as homogenous Riemannian  manifold
($\tilde{\mc{X}}^J_1=\frac{G^J_1(\R)}{\text{SO}(2)},g_{\tilde{\mc{X}}^J_1}$) is a
reductive, non-symmetric manifold, non-naturally reductive with
respect with the metric \eqref{linvG}, not a g.o. space \cite{kwv}
 with respect to the
invariant metric $g_{\tilde{\mc{X}}^J_1}$.

We recall that the Berry phase is an important object in the study of
geometric phase physics \cite{BS,swA, GO,DP}.
We have studied the Berry phase on homogenous \Ka~ manifolds  in
\cite{sbcag, sbl, FC}. 

The paper is laid out as follows.   In Section  \ref{PR}  we recall the \Ka~
two-form on $\mc{D}^J_1$ and its two-parameter balanced metric image
on $\mc{X}^J_1$ obtained by  
the partial Cayley transform in Proposition \ref{PRFC}, the three
parameter invariant metric on $\tilde{\mc{X}}^J_1$ in the
S-coordinates \cite{BS} in Proposition \ref{PROP5},  while Proposition
\ref{PRR3} recalls  the invariant metric on $\mc{D}^J_n$, $\mc{X}^J_n$
and  $\tilde{\mc{X}}^J_n$.
Section \ref{BEP} recalls our investigation on  Berry phase on \Ka~
manifolds. In particular, Proposition \ref{PR44} recalls the Berry
phase on $\mc{D}^J_1$ and $\mc{X}^J_1$.
Section   \ref{ACS} summarise  the notion of  almost cosymplectic
manifold \cite{paul}. In particular, the manifold   $\tilde{\mc{X}}^J_1$ is endowed with a
generalized transitive almost cosymplectic structure \cite{SB22}. 
In \S \ref{CMSJ} is calculated the connection matrix \cite{cf}  on $\tilde{\mc{X}}^J_1$ and
in \S \ref{CD} are presented the covariant derivative (as in
\cite{YY1,YY2}, see also \cite[\S 3.2]{arr})
on ${\mc{X}}^J_1$
and $\tilde{\mc{Xs}}^J_1$.  The last Section - Appendix - collects
several mathematical concepts used in the paper:
 connections on real manifolds in Subsection \ref{REALC},
connections on complex manifolds,   Chern connections \cite{CH67}  and quantizable \Ka~ manifolds in
Subsection \ref{CCM}, 
the notion of holonomy \cite{BK} is recalled in  Subsection \ref{HOLL}. Some example
are contained   in Subsection \ref{544} devoted to  coherent states: Berry
connection and \Ka~ two-form for
the Heisenberg-Weyl  group, sphere $S^2$,  $\mc{D}_1$,
complex Grassmann manifold $G_n(\C^{n+m})$ and its non-compact dual,
$\db{CP}^n$ and $\db{CP}^{n,1}$.

Proposition
 \ref{PRFC} and  Comment \ref{CM1}
 are  improved   versions of older results, while      Remark
 \ref{RRR}  compare  our approach to Berry phase on \Ka~ manifolds
 \cite{sbcag,sbl,FC} 
 with the geometric phase in \cite{swA,BS,CH}.    The new relevant  results presented in this paper are \S ~\ref{33},
 formula \eqref{thetam} of $\theta_{\mc{X}^J_1}(x,y,q,p)$, 
   formula
\eqref{ABXY} of Berry phase on $\mc{X}^J_1$ in $(u,v)=(m+\ii n, x+\ii
y)$, formulae \eqref{ABXYY}, \eqref{AA1}, \eqref{AA2} of the Berry
phase in $(w,v)=(\alpha+\ii \beta, x+\ii y)$, \eqref{EQ11} for the
Berry phase in $(x,y,q,p)$,  Lemma
  \ref{NNOOU} which gives $\omega_{\mc{X}^J_n}(x,y,p,q)$, formula \eqref{thetap} 
  $\theta'_{\tilde{\mc{X}}^J_1}(x,y,q,p,\kappa)$ of the connection
  matrix on $\tilde{\mc{X}^J_1}$, 
 the covariant
derivatives $Dx$,(~$Dy,~Dq,~Dp$) \eqref{DX}, ( \eqref{DY}, \eqref{DQ},
respectively 
\eqref{DP})  on $\mc{X}^J_1$,  formulae of $Dx$, ($Dy,~Dq,~Dp, ~D\kappa
$)  \eqref{DDX}  (\eqref{DDY}, \eqref{DDQ}, \eqref{DDP},
respectively \eqref{DDK}) on $\tilde{\mc{X}}^J_1$.

\textbf{Notation}
We denote by $\mathbb{R}$, $\mathbb{C}$, $\mathbb{Z}$ and $\mathbb{N}$ 
 the field of real numbers, the field of complex numbers,
the ring of integers,   and the set of non-negative integers, respectively. We denote the imaginary unit
$\sqrt{-1}$ by~$\ii$, the real and imaginary parts of a complex
number $z\in\C$ by $\Re z$ and $\Im z$ respectively, and the complex 
conjugate of $z$ by $\bar{z}$.
 We denote by ${\dd }$ the differential. 
We use Einstein's summation convention, i.e.  repeated indices are
implicitly summed over.  The set of vector fields (1-forms) on real
manifolds  is denoted
by $\got{D}^1$ (respectively $\got{D}_1$). We denote a mixed tensor
contravariant of degree $r$ and covariant of degree $s$ by $\got{D}^r_s=\got{D}^r\times\got{D}_s$,
where $\got{D}^r=\underbrace{\got{D}^1\times \dots \times
  \got{D}^1}_{r}$ and $\got{D}_s=\underbrace{\got{D}_1\times \dots \times
  \got{D}_1}_{s}$ \cite[pages 13-17]{helg}. If $M$ is a complex manifold we
denote by $\got{A}^{r,s}$ the tensor fields of type $(r,s)$. 
  If we
denote with Roman  capital letteres the Lie  groups, then their
associated Lie algebras are denoted with the corresponding lower-case
letteres.   If $\got{H}$ is a Hilbert space, than we adopt the
  convention  that the scalar product $(\cdot,\cdot)$ on
  $\got{H}\times\got{H}$ is antilinear in the first
factor 
$(\lambda a,b)=\bar{\lambda}(a,b),\quad \lambda \in \C\setminus \{0\}
$. If $\pi$ is a representation of a Lie grup $G$ 
on the Hilbert $\got{H}$  and $X\in\got{g}$,  then we denote
$\bf{X}:=\dd \pi (X)$   \cite{SB03,SB14,perG}. The interior  product $i_X\omega$ (interior multiplication or contraction)
of the differential form $\omega$
with $X\in\got{D}^1$ is denoted $X\lrcorner \omega$. We
denote by $M(n,m,\db{F})$ the set of $n\times m$ matrices with elements
in the field $\db{F}$ and  $M(n,\db{F})$ denotes $M(n,n,\db{F})$. If $X\in M(n,m,\db{F})$, then $X^t$ denotes the
transpose of $X$. We denote by $MS(n,\db{F})=\{X\in M(n,\db{F})|X=X^t\}$. The
 conjugate transpose (or hermitian transpose)  of $ A\in
 M(q,\C)$  is  $A^H:=\bar{A}^t$,  also denoted $A^*$, $A^{\dagger}$, $A^+$.
 If $f$ is a function on $\C^n$, we write for the
total differential of $f$  $\dd f= \pa f+\bar{\pa} f$, $\pa f
=\sum_1^n{\pa_{\alpha}f}\dd z_{\alpha}$,
  where $\pa_{\alpha}f=\frac{\pa f}{\pa z_{\alpha}} $ \cite[page
  6]{GH}. If $f$ is a complex function, then by $f-{cc}$ we mean $f-\bar{f}$.

  \section{Preparation}\label{PR}

We adopt  the notation from  \cite{bs,ez} for the real  Jacobi group   $G^J_1(\R)$,  
realized as  submatrices of $\text{Sp}(2,\R)$ of the form
\begin{equation}\label{SP2R}
g=\left(\begin{array}{cccc} a& 0&b &   q\\
\lambda &1& \mu & \kappa\\c & 0& d &  -p\\
         0& 0& 0& 1\end{array}\right),~ M=
    \left(\begin{array}{cc}a&b\\c&d\end{array}\right),~ \det M
   =1, \end{equation}
where
\begin{equation}\label{DEFY}Y:=(p,q)=XM^{-1}=(\lambda,\mu) \left(\begin{array}{cc}a&b\\c&d\end{array}\right)^{-1}=(\lambda d-\mu
  c,-\lambda b+\mu a)\end{equation} 
is related to the Heisenberg group $\rm{H}_1$ described by
$(\lambda,\mu,\kappa)$.  For  coordinatization of  the  real Jacobi
group  we adopt  the so called  $S$-coordinates
$(x,y,\theta,p,q,\kappa)$  \cite{bs}.

Simultaneously with
the Jacobi group $G^J_1(\R)$ consisting of elements $(M,X,\kappa)$, we
considered the restricted real Jacobi  group $G^J(\R)_0$  of elements $(M,X)$ \cite{jac1,SB19}.

  The action  $G^J(\R)_0\times \mc{X}^J_1\rightarrow
  \mc{X}^J_1$    (respectively  $G^J_1(\R)\times \tilde{\mc{X}}^J_1\rightarrow
  \tilde{\mc{X}}^J_1$) in Lemma \ref{LEMN} below  is extracted from \cite[Lemma 5.1]{SB19}, \cite[Lemma  1]{SB20}.

Let \begin{equation}\label{TAUZ}
\C\ni v:=x+\ii y,~~~
\C\ni  u:=pv+q=\xi+\ii \rho,\quad ~x,y,p,q,
\xi,\rho\in\R.
\end{equation}
\Ka~calls $\tilde{\mc{X}}^J_1$ {\it{Phasenraum der Materie}}, $v$ is 
{\it{Pneuma}}, $u$ is {\it {Soma}} \cite[Sec. 35]{cal3}.

Let $\mc{X}^J_1\approx \mc{X}_1\times\R^2$ be the Siegel--Jacobi upper half-plane,
where $\mc{X}_1=\{v \in\C| ~y:=\Im v>0\}$ is the Siegel upper half-plane,  and
$\tilde{\mc{X}}^J_1\approx\mc{X}^J_1\times\R$ denotes  the extended Siegel--Jacobi upper half-plane. 
Then: 
\begin{lemma}\label{LEMN} a) The action  $G^J(\R)_0\times \mc{X}^J_1\rightarrow
  \mc{X}^J_1$
is given by \begin{equation}\label{AC1}
(M,X)\times
(v',u')=(v_1,u_1),\emph{\text{~where~}}v_1=\frac{a v'+b}{c
  v'+d},~u_1=\frac{u'+\lambda u'+\mu}{c u'+d}.\end{equation}

b) If  $u'=p'v'+q'$, $v'=x'+\ii y'$ as in \eqref{TAUZ}, then the
action 
\begin{equation}\label{AC11}
(M,X)\times (x',y',p',q')=(x_1,y_1,p_1,q_1)
\end{equation}
is given by the formula
\begin{equation} x_1+\ii y_1=\frac{(ax'+b)(cx'+d)+ac y'^2+\ii
  y'}{(cx'+d)^2+(cy')^2}\label{ALIGNN1},
                                  \end{equation}
and
\begin{equation}\label{AC12}
(p_1,q_1)=(p,q)+(p',q')
\left(\begin{array}{cc}a & b\\c & d\end{array}\right)^{-1}=(p+dp'-cq',q-bp'+aq').
\end{equation}

c)  The action $G^J_1(\R)\times \tilde{\mc{X}}^J_1\rightarrow
  \tilde{\mc{X}}^J_1$ is given by
\begin{equation}\label{AC2}
\begin{split}
& (M,X,\kappa)\times
 (v',z',\kappa')  =(v_1,z_1,\kappa_1),\\
& (M,X,\kappa)\times (x',y',p',q',\kappa')  =(x_1,y_1,p_1,q_1,\kappa_1),\\
 & \kappa_1  =\kappa +\kappa' +\lambda
q'-\mu p',~
(p',q')= (\frac{\rho'}{y'},\xi'-\frac{x'}{y'}\rho'),~
(\lambda,\mu)=(p,q)M
\end{split}
\end{equation}
and \eqref{ALIGNN1}, \eqref{AC12}.
\end{lemma}

Proposition \ref{PRFC} is an improved version of
\cite[(4.38), (5.8)]{SB14}, \cite[(28),~ (29)]{GAB}, \cite[Proposition 2.1]{SB19},
 \cite[Proposition 2]{SB21}, \cite[Proposition 2]{SB22}, \cite[(18),(19)]{BER7}.

Below $(w,z)\in  ( \mc{D}_1,\C)$,
$(v,u)\in (\mc{X}_1,\C)$, and 
the parameters $k$ and $\nu$ come from representation theory of the
Jacobi group: $k$ indexes the positive discrete series of ${\rm
  SU}(1,1)$, $2k\in\N$, while $\nu>0$ indexes the representations of
the Heisenberg group \cite{jac1}.   $\rm{FC}$ is an abbreviation for
the {\it fundamental conjecture} for homogeneous \Ka~ manifolds
\cite{GV}, see also  \cite{pia}, \cite{DN}.

\begin{Proposition} \label{PRFC}.

Perelomov's coherent state vectors   associated to the group $G^J_1$
are defined as 
\begin{equation}\label{csu}
e_{z,w}:=e^{\sqrt{\mu}z{\mb{a}}^{\dagger}+w{\mb{K}}_+}e_0, ~z\in\C,~ |w|<1 ,
\end{equation}
and the reproducing kernel $K = K(\bar{z},\bar{w},z,w)$ is 
\begin{equation}\label{hot}
K =\!(e_{z,w},e_{z,w})\!=\!
(1\!-\!w\bar{w})^{-2k}\exp{\nu\frac{2z\bar{z}\!+\!z^2\bar{w}\!+\!\bar{z}^2w}{2(1-w\bar{w})}},
z ,w\in\C,|w|<1 . 
\end{equation}

  a) The \Ka~two-form on  $\mc{D}^J_1$, invariant to the action of
  $G^J_1=\rm{SU}(1,1)\ltimes\C$, is 
 \begin{equation}\label{kk1}
  -\ii \omega_{\mc{D}^J_1}
(w,z)\!=\!\frac{2k}{P^2}\dd w\wedge\dd
  \bar{w}\!+\!\nu \frac{\mc{A}\wedge\bar{\mc{A}}}{P},P:=1-|w|^2,\mc{A}=\mc{A}(w,z):=\dd
  z\!+\!\bar{\eta}\dd w.
\end{equation}
We have the change of variables $FC: (w,z)\rightarrow (w,\eta, \bar{\eta})$
\begin{gather}\label{E32}
{\rm FC}\colon \
 z=\eta-w\bar{\eta},\qquad {\rm FC}^{-1}\colon \
 \eta=\frac{z+\bar{z}w}{P},
\end{gather}
and
\begin{equation}\label{E32a}
{\rm FC}\colon \ \mc{A}(w,z)\rightarrow\mc{A}(w,\eta,\bar{\eta}):= \dd \eta -w\dd
\bar{\eta},
\end{equation}
\begin{subequations}
  \begin{align}
  -\ii \omega_{\mc{D}^J_1}(w,\eta)& = 
  -\ii {\rm{FC}}^*(\omega_{\mc{D}^J_1}(w,z))=\frac{2k}{P^2}\dd w\wedge\dd
                                 \bar{w}+\nu\dd\eta\wedge\dd
                                    \bar{\eta}, \label{E32b}\\
    \omega_{\mc{D}^J_1}(\alpha,\beta,q,p) & =4k\frac{\dd \alpha
                                            \wedge\dd \beta}
   {(1-\alpha^2-\beta^2)^2}+2\nu \dd q\wedge \dd p,\label{E32bb}
   \end{align}\end{subequations}
where
\begin{equation}\label{WVAB}
  w=\alpha+\ii \beta, \alpha,~\beta \in \R, \quad \eta=q+\ii p,~ p,q\in \R.
\end{equation}

Also with \eqref{hot} and \eqref{E32} we  have
\begin{equation}\label{GGG}
  B(w,\eta-w\bar{\eta})=(1-w\bar{w})^{-2k}\exp\nu \left[\eta\bar{\eta}-\frac{\bar{w}\eta^2+w\bar{\eta}^2}{2}\right].
  \end{equation}

  With a formula similar to \cite[(7.18)]{jac1} applied to
  \eqref{GGG},  we get
  \begin{subequations}
    \begin{align}
    -\ii \omega(w,\eta)& =\frac{2k}{(1-w\bar{w})^2}\dd w\wedge\dd
                         \bar{w}+\nu [\dd\eta\wedge \dd \bar{\eta}-\bar{\eta}\dd w\wedge
                         \dd \bar{\eta} +\eta \dd \bar{w}\wedge\dd \eta],\label{F1}\\
      \omega(\alpha,\beta,q,p) &=4k\frac{\dd \alpha
                                            \wedge\dd \beta}
   {(1-\alpha^2-\beta^2)^2}+2\nu \dd q\wedge \dd p \label{F2}\\& +2\nu [\dd q \wedge(p\dd \alpha-q\dd
                                 \beta)+\dd p\wedge(p\dd\beta+q\dd \alpha) ]\nonumber,
      \end{align}
    \end{subequations}
    and equation \eqref{F2}  is different of \eqref{E32bb}.

    In \eqref{GGG} we make the change of coordinates $w\rightarrow v$ \eqref{210b} and
$\eta=q+\ii p$, we get
\[
  \bar{\eta}^2w=\frac{(q^2-p^2-2\ii qp)(x^2+y^2-1-2\ii x)}{N},
  \]
  and finally we get
  \begin{equation}\label{FXY}
    f(x,y,q,p)=-2k \log \frac{4y}{N}+\nu F,\quad  F=
    \frac{2}{N}[(y+1)q^2+(x^2+y^2+y)p^2+2qpx].
    \end{equation}

  If in \eqref{F1}
  we make the change of variables $w\rightarrow
    v$ \eqref{ULTRAN1}, we get the \Ka~ two-form
    \begin{equation}
      -\ii \omega(v,\eta)=\frac{k}{2y^2}\dd v\wedge \dd\bar{v}+\nu \{-
      2\ii [
      \frac{\bar{\eta}}{(v+\ii)^2}\dd v\wedge \dd\bar{\eta}+      \frac{\eta}{(\bar{v}-\ii)^2} \dd \bar{v}\wedge \dd \eta] +\dd \eta\wedge \dd\bar{\eta}\},
    \end{equation}
    or the symplectic two-form
    \begin{subequations}\label{altaOM}
      \begin{align}
      \omega_{\mc{X}^J_1}(x,y,q,p)& =\frac{k}{y^2}\dd x\wedge \dd y \!+\!\frac{4\nu}{N^2}\{[q(x^2\!-\!(y+1)^2)\!-\!2px(y+1)](\dd
      x\wedge \dd q+\dd y\wedge \dd p)\nonumber\\
      & +[2qx(y+1)+p(x^2-(y+1)^2)](-\dd x \wedge \dd p+\dd
      y\wedge \dd q)\}
      +2\nu \dd q\wedge \dd p.\nonumber
      \end{align}
    \end{subequations}
    Note that \eqref{altaOM} is different from \eqref{214b} and \eqref{omM}. 
  The matrix of the balanced metric $h=h(\varsigma)$, 
$\varsigma:=(z,w)\in\C\times\mc{D}_1$ associated to the \Ka~two-form
\eqref{kk1} is 
\begin{equation}\label{metrica}
  h(\varsigma) =\left( \begin{array}{cc}h_{z\bar{z}}& h_{z\bar{w}}\\
                         h_{w\bar{z}}& h_{w\bar{w}}\end{array} \right)=
  \left(\begin{array}{cc} \frac{\mu}{P} & \mu \frac{\eta}{P} \\
\mu\frac{\bar{\eta}}{P} &
\frac{2k}{P^2}+\mu\frac{|\eta|^2}{P}\end{array}\right).  
\end{equation}
The inverse of the matrix \eqref{metrica} reads
\begin{equation}\label{hinv}
h^{-1}(\varsigma)= \left(\begin{array}{cc}h^{z\bar{z}}&
      h^{z\bar{w}}\\h^{w\bar{z}}&h^{w\bar{w}}\end{array}\right)  =  \left(\begin{array}{cc}
    \frac{P}{\mu}+\frac{P^2|\eta|^2}{2k} & -\frac{P^2\eta}{2k} \\
-\frac{P^2\bar{\eta}}{2k} & \frac{P^2}{2k}\end{array}\right).
\end{equation}

b) The second partial Cayley transform $\Phi_1: \mc{D}^J_1\rightarrow
\mc{X}^J_1$ and 
\begin{equation}\label{PHH1}
  \Phi_1:={\rm{FC}}_1\circ\Phi:  (w,z)\rightarrow
(v=x+\ii y,\eta=q+\ii p)\end{equation}
 and its inverse $\Phi_1 ^{-1}: (v,\eta)\rightarrow (w,z)$ are  given  by  
\begin{subequations}\label{ULTRAN}
\begin{align}
 \Phi_1: & ~w=\frac{v-\ii}{v+\ii}, \quad  z=\eta-\bar{\eta}\frac{v-i}{v+\ii}=2\ii \frac{pv+q}{v+\ii},\label{ULTRAN1}\\
\Phi_1 ^{-1}: &~ v=\ii \frac{1+w}{1-w}, \quad \eta =
\frac{(1+\ii \bar{v})(z-\bar{z})+v(\bar{v}-\ii)(z+\bar{z})}{2\ii
  (\bar{v}-v)}=\frac{z+\bar{z}w}{P}.\label{ULTRAN2}
\end{align}
\end{subequations}
Introducing the second partial Cayley transform \eqref{ULTRAN1} into
the \Ka ~two-form \eqref{kk1} on $\mc{D}^J_1$, we get the symplectic
two-form \eqref{omSYUHP}  on the
Siegel-Jacobi upper half-plane $(v,\eta)$, $\Im v>0$
\begin{subequations}\label{omSYUHP}
  \begin{align}
  -\ii \omega_{\mc{X}^J_1}(v,\bar{v},\eta,\bar{\eta}
    )&=\frac{8k}{P^2}\frac{\dd v\wedge \dd \bar{v}}{N^2}+\nu \dd
       \eta\wedge\dd \bar{\eta}=-\frac{2k}{(\bar{v}-v)^2}\dd v\wedge
       \dd \bar{v} + \eta\wedge\dd \bar{\eta},\label{214a}\\
     \omega_{\mc{X}^J_1}(x,y,q,p )    & =\frac{k}{y^2}\dd x\wedge \dd
                                         y+2\nu \dd q\wedge \dd p,\label{214b}
 \end{align}
\end{subequations}
\begin{equation}\label{NNN}
  N:=|v+\ii|^2=x^2+(y+1)^2,
\end{equation}
\begin{equation}\label{PPP}
  P=4\frac{y}{N}.
\end{equation}

 c) Using the partial Cayley transform
 $\Phi^{-1}:\mc{D}^J_1\rightarrow \mc{X}^J_1, ~(w,z)\rightarrow (v,u)$ and its
inverse 
\begin{subequations}\label{210}
\begin{align}
\Phi^{-1}: v & =\ii \frac{1+w}{1-w},~~u=\frac{z}{1-w}, ~~w,z\in\C,~
               |w|<1,\label{210a}\\
  \Phi: w & =\frac{v-\ii}{v+\ii},~~z=2\ii
        \frac{u}{v+\ii},~~v,u\in\C,~\Im v>0,\label{210b} 
\end{align}
\end{subequations}
we obtain
\begin{equation}\label{ALEFT}
\mc{A}\left(\frac{v - \ii}{v+ \ii},\frac{2\ii
      u}{v + \ii}\right)=\frac{2\ii}{v+\ii}\mc{B}(v,u),
\end{equation}
   where
  \begin{equation}\label{BFR2}
  \mc{B}(v,u) := {\rm d} u - r{\rm d} v, ~
  r:=\frac{u-\bar{u}}{v-\bar{v}}.
\end{equation}
  The Berndt--\Ka 's two-form (symplectic two-form)  invariant to the action of
$G^J(\R)_0$ $= \rm{SL}(2,\R)\ltimes\C$, is \eqref{BFR} \emph{(}\eqref{BRF}\emph{)}
\begin{subequations}
  \begin{align}
- \ii \omega_{\mc{X}^J_1}(v,u) & = -\frac{2k}{(\bar{v} - v)^2} \dd
v\wedge \dd\bar{v}+ \frac{2\nu}{\ii(\bar{v} -
                                 v)}\mc{B}\wedge\bar{\mc{B}} \label{BFR}\\
     & =\frac{1}{y}\{(\frac{k}{2y}+\nu r^2)\dd v\wedge \dd
       \bar{v}+\nu[\dd u\wedge \dd \bar{u}-r(\dd v\wedge
       \dd\bar{u}-cc)]\}, \label{BFFR}\\
    \omega(x,y,m,n) &= \frac{k}{y^2}\dd x\wedge \dd
                      y+\frac{2\nu}{y}(\dd m-r\dd x)\wedge (\dd n-r\dd
                      y)\label{BRF} \\
   & = (\frac{k}{y^2}\!+\!2\nu \frac{r^2}{y})\dd x\wedge \dd
     y\!+\!2\frac{\nu}{y}[\dd m\wedge\dd n\!+\!r(\dd y\wedge \dd m\!-\!\dd
     x\wedge \dd n)],\label{BRF2}\end{align}
\end{subequations}
\begin{equation}\label{umn}
  u = m+\ii n, \quad m,n \in \R, \quad r= \frac{n}{y} 
\end{equation}
With \eqref{hot}  and \eqref{210b} we get
\begin{equation}\label{DOIIi}
  K(\frac{v-\ii}{v+\ii},\frac{2\ii u}{v+\ii})=\left[\frac{|v+\ii|^2}{2\ii(\bar{v}-v)}\right]^{2k}\exp\frac{2\nu}{|v+\ii|^2}[|u|^2-\frac{(u\bar{v}-\bar{u}v)^2+(\bar{u}-u)^2}{2\ii(\bar{v}-v)}].
\end{equation}

With the    change of variables
${\rm FC}_1\colon (v,u)\rightarrow (v,\eta)$
 \begin{equation}\label{FC1MIN}
 {\rm FC}_1\colon \ 2\ii u=(v+\ii)\eta-(v-\ii)\bar{\eta}, \qquad
   {\rm FC}^{-1}_1 \colon  \eta=\frac{u\bar{v}-\bar{u}v}{\bar{v}-v} +
   \ii r,  \end{equation}
 in \eqref{BFR} we get \[r=p,\quad  m=px+q,\quad n=py, \] 
 \[
   B(v,\bar{v},\eta,\bar{\eta})=\frac{1}{2\ii}[(v+\ii )\dd
     \eta-(v-\ii)\dd \bar{\eta}], \] 
   and finally we regain \eqref{omSYUHP}.

   The matrix corresponding  to the balanced  metric \eqref{NEWMM}
    associated with
     the   \Ka~ two  -form \eqref{BFR}
reads 
     \begin{equation}\label{kmb}
       h(v,u)
       =\left(\begin{array}{cc}h_{v\bar{v}}&h_{v\bar{u}}\\\bar{h}_{v\bar{u}}&h_{u\bar{u}}\end{array}\right)
       =\left(\begin{array}{cc}   \frac{k}{2y^2}+\nu\frac{r^2}{y}
                &-\nu\frac{r}{y}\\-\nu\frac{r}{y} & \frac{\nu}{y}
     \end{array}\right),~y:=\frac{v-\bar{v}}{2\ii},
\end{equation}
and we also have
\begin{equation}\label{kmbINV}
       h^{-1}(v,u)= 
      \left(\begin{array}{cc}h^{v\bar{v}}&h^{v\bar{u}}\\\bar{h}^{u\bar{v}}&h^{u\bar{u}}\end{array}\right)
       =\left(\begin{array}{cc}\frac{2y^2}{k}&\frac{2y^2r}{k}\\\frac{2y^2r}{k} &
     \frac{y}{\nu}+2\frac{r^2y^2}{k}\end{array}\right).
\end{equation}

 d) If we apply the change of coordinates $\mc{D}^J_1\ni
 (v,u)\rightarrow 
 (x,y,p,q)\in\mc{X}^J_1$ \eqref{TAUZ}, 
 
then

\begin{equation}\label{BUVpq}
\mc{B}(v,u)=\dd u -p \dd v,
\end{equation}
\begin{equation}\label{BUVpq1}
  \mc{B}(v,u)=\mc{B}(x,y,p,q):=F \dd t =\mc{F}= v\dd p+\dd q=(x+\ii y)\dd p
  +\dd q,   F:= \dot{p}v+\dot{q},
\end{equation}
and we regain  \eqref{214b}.

e) The two-parameter   balanced  metric  on the
Siegel--Jacobi upper half-plane $\mc{X}^J_1$  associated to the \Ka~
two-form \eqref{BFR} is 
\begin{subequations}\label{METRS2}
  \begin{align}
  \!\!\dd s^2_{\mc{X}^J_1}(x,y,p,q)  \!&=\!
\alpha\frac{\dd x^2\!+\!\dd   y^2}{y^2}\!+\!\frac{\gamma}{y}(S\dd p^2\!+\!\dd q^2+2x\dd
                                     p\dd q)\\
    \!&=\!\alpha\frac{\dd x^2\!+\!\dd   y^2}{y^2} \!+\!
        \frac{\gamma}{y}(A^2\!+\!\!B^2),
 \end{align}  
\end{subequations}
where
\begin{equation}\label{AK}
  \alpha:=k/2,\quad\gamma:=\nu, \quad  S:=x^2+y^2, A\!:=\!\Re\mc{F}\!=\!x\dd p\!+\!\dd q,B\!:=\!\Im
        \mc{F}\!=\!p\dd y.
\end{equation}

  The metric matrix associated  with \eqref{METRS2} is
\[
g_{{\mc{X}}^J_1}\! = \!\left(\begin{array}{cccc}g_{xx} &0 &0 &0\\
0& g_{yy}& 0& 0 \\
0& 0& g_{pp} & g_{pq} \\0 & 0& g_{qp}& g_{qq}
 \end{array}\right),\!
 \begin{array}{cc}\qquad g_{xx}\!=\frac{\alpha}{y^2}, &
 \!g_{yy}\!=\!\frac{\alpha}{y^2};\\
g_{pq}\!=\!\gamma\frac{x}{y} , &
g_{pp} \!=\!\gamma\frac{S}{y},\quad g_{qq}\!=\!\frac{\gamma}{y}.
\end{array}
\]
\end{Proposition}

Below we reproduce  the Comment 5.5 in the first reference \cite{SB19} with some
completions: 
\begin{Comment}\label{CM1}
Berndt  \cite[p 8]{bern84}  considered the closed two-form
$\Omega=\dd \bar{\dd} f'$
of Siegel--Jacobi upper half-plane $\mc{X}^J_1$,  $G^J(\R)_0$-invariant to the action \eqref{AC1},
 obtained from the K\"ahler potential 
\begin{equation}\label{POT}
f'(\tau,z)= c_1\log (\tau-\bar{\tau}) -\ii
c_2\frac{(z-\bar{z})^2}{\tau-\bar{\tau}}, ~c_1,c_2>0.\end{equation}
Formula \eqref{POT} is  presented by Berndt as 
``communicated to the author  by \Ka''. Also in \cite[p 8]{bern84} is 
given  our  equation {\emph{{(5.21a)}}} in first reference \cite{SB19}, while our
present equation \eqref{METRS2} corrects two
printing errors in Berndt's paper.

Later, in \cite[\S~ 36]{cal3}, reproduced also
in \cite{cal},
\Ka~    argues how to choose the  potential as in
\eqref{POT},   
see  also  \cite[(9)  \S ~ 37]{cal3}, 
where  $c_1=-\frac{k}{2}$, $c_2=\ii\nu\pi$, i. e.
\begin{equation}\label{POT1}
  f'(\tau,z)= -\frac{k}{2}\log\frac{ \tau-\bar{\tau}}{2\ii}
  -\ii\pi\nu \frac{(z-\bar{z})^2}{\tau-\bar{\tau}}.\end{equation}

Once the \Ka~ potential \eqref{POT1} is known, we apply the recipe
\eqref{KALP2}
$$-\ii
\omega_{\mc{X}^J_1}(\tau,z)=f'_{\tau\bar{\tau}}\dd
\tau\wedge\dd \bar{\tau}+f'_{\tau\bar{z}}\dd \tau\wedge\dd \bar{z}
-\bar{f}'_{\tau\bar{z}}\dd \bar{\tau}\wedge \dd z+f'_{z\bar{z}}\dd z\wedge
\dd \bar{z}.$$

 The metric {\emph{(8)}} in \cite{cal3}  differs from  the metric
 \eqref{METRS2} by a factor of two,
 since  the Hermitian metric used by  \Ka~  is
 $\dd s^2=2g_{i\bar{j}}\dd z_i\otimes \dd\bar{z}_j$. If in \eqref{POT1} we
 take $k/2\rightarrow k$, we have 
\begin{subequations}
  \begin{align*}
 f'_{\tau}&=-k\frac{1}{\tau-\bar{\tau}}+\ii \pi \nu
    \frac{(z-\bar{z})^2}{(\tau-\bar{\tau})^2},~
    f'_{\tau\bar{\tau}}=-k\frac{1}{(\tau-\bar{\tau})^2}+2\ii
            \pi\nu\frac{(z-\bar{z})^2}{(\tau-\bar{\tau})^3}, \\
 f'_{\tau\bar{z}}& =-2\ii
    \pi\nu\frac{z-\bar{z}}{(\tau-\bar{\tau})^2},~
    f'_{z}  =-2\ii\pi\nu\frac{z-\bar{z}}{\tau-\bar{\tau}}, ~f'_{z\bar{z}}=2\ii\pi\nu\frac{1}{\tau-\bar{\tau}},
 \end{align*}
\end{subequations}
and  we get  \eqref{BFR}. 
 Relation \eqref{BFR}
 has been  obtained by Berndt  \cite[p 30]{bern}, where the
 denominator of the first term is misprinted  as $v-\bar{v}$ 
(or $\tau-\bar{\tau}$ in our notations). Equation
 \eqref{METRS2} appears also in  \cite[p 30]{bern} and \cite[p 62]{bs}.

 We denote in \eqref{POT1} $(\tau,z)$  with $(v,u)$ as in \cite[(9.16)]{jac1}.   
 Indeed, we make successively the transformations: the partial Cayley
 transform, $\Phi:~(w, z) \rightarrow 
 (v,u)$ \eqref{210}, a holomorphic transform, and we get \eqref{222a},
 then we apply  the   $FC_1$ transform $(v,u) \rightarrow  (v,\eta)$
 \eqref{FC1MIN}, 
 a non-holomorphic transform, to obtain \eqref{222b}, and finally we
 make the symplectic transform 
 $(w,z)\rightarrow (x,y,q,p)$ with the result \eqref{222c}
 \begin{subequations}\label{Ktau2}
\begin{align}
   \log K(v,u) & =-\frac{k}{2}\log
   \frac{v-\bar{v}}{2\ii}-\ii \nu
                              \frac{(u-\bar{u})^2}{v-\bar{v}}, 
                 \label{222a}\\
  \log K(v,\eta) & =
              -\frac{k}{2}\log
   \frac{v-\bar{v}}{2\ii}
              +\frac{\ii \nu}{4} (\eta-\bar{\eta})^2(v-\bar{v})\label{222b},\\
  \log K(x,y,q,p) & =- \frac{k}{2}\log y
  +2\nu yp^2.\label{222c}
\end{align}
\end{subequations}
Note that \eqref{222a} is different of \eqref{DOIIi}.

The metric  associated to the \Ka~two-form \eqref{214b} is
\begin{equation}\label{newM}
 \dd s^2 (x,y,q,p)=\frac{k}{2y^2}(\dd x^2 +\dd y^2) +\nu (\dd q^2
  +\dd p^2).
\end{equation}

The metric corresponding to the \Ka~two-form \eqref{222a} is
\begin{subequations}\label{NEWMM}
  \begin{align}
  \dd  s^2(x,y,n,m) & \!=\!(\frac{k}{2}+\nu \frac{n^2}{y})\frac{\dd x^2\!+\!\dd y^2}{y^2}
                      \!+\!\frac{\nu}{y}[\dd n^2 \!+\!\dd m^2 -2r(\dd m \dd x\!+\! \dd n \dd  y)]\\
                    &=\frac{k}{2}\frac{\dd x^2+\dd y^2}{y^2}+\frac{\nu}{y}[
                      (r\dd x -\dd m)^2+(r\dd
                      y-\dd n)^2], ~r=\frac{n}{y}.
    \end{align}
     \end{subequations}

The  metric \eqref{NEWMM} corresponds to the \Ka~potential \eqref{222a}
\begin{equation}
  f"(v,u)=-2k\log \frac{v-\bar{v}}{2\ii}-\ii\nu\frac{(u-\bar{u})^2}{v-\bar{v}}
  \end{equation}
  instead of \eqref{POT1}.

Equation \eqref{222c} was presented in \cite[(9.20)]{jac1}.

 In \cite[(4.3)]{gem}, see also \cite[Proposition
 4.1]{gem},  we have presented a
 generalization of \eqref{222c} for $\mc{X}^J_n$, obtained by Takase in \cite[\S 9]{tak}.

 Yang
 calculated  in \cite{Y07}  the 
  metric on $\mc{X}^J_n$, invariant to the action of $G^J_n(\R)_0$.  The equivalence of the metric of Yang with the metric
  obtained via CS  on $\mc{D}^J_n$ and then transported  to
  $\mc{X}^J_n$ via the partial Cayley transform $(v,u)\rightarrow
  (v,\eta)$ \eqref{210} is underlined in
  \cite{nou}. In particular, the metric {\emph{(5.21b)}}
  in  the first reference \cite{SB19}
 appears in \cite[p 99]{Y07} for the particular values $c_1=1$,
 $c_2=4$. See also \cite{Yan,Y08,Y10}. 
\end{Comment}

\begin{Remark}\label{REM1}
 In formula \eqref{222a} we make  the change of variables $FC_1$
  \eqref{FC1MIN} and we get \eqref{222b}.  
We apply to \eqref{222b} \cite[(7.18)]{jac1} to calculate
\[
  -\ii \omega(v,\eta)=h_{v\bar{v}}\dd v\wedge
  \dd \bar{v}+h_{v\bar{\eta}}\dd v\wedge \dd
  \bar{\eta}-\bar{h}_{v\bar{\eta}}\dd \bar{v}\wedge \dd
  \eta+h_{\eta\bar{\eta}}\dd \eta\wedge \dd\bar{\eta}.
\]
The associated matrix 
\begin{equation}\label{hs}
  h=\left(\begin{array}{cc}h_{v\bar{v}} & h_{v\bar{\eta}}\\
            h_{\eta
            \bar{v}}& h_{\eta\bar{\eta}}\end{array}\right)=
        \left(\begin{array}{cc}\frac{k}{8y^2}&\nu p\\
                \nu p & \nu  y \end{array}\right)
       \end{equation}
is hermitian and we have 
\begin{equation}\label{omM}
  \omega_{\mc{X}^J_1}(x,y,q,p)=\frac{k}{4}\frac{\dd x\wedge\dd
    y}{y^2}+2\nu[p(\dd x\wedge \dd p+\dd q\wedge \dd y)+y\dd q\wedge
  \dd p].
\end{equation}
\eqref{omM} is different of \eqref{214b} obtained introducing
\eqref{ULTRAN1} into \eqref{kk1}.
\end{Remark}

\begin{proof}
  We get from \eqref{222b}
  \begin{subequations}
    \begin{align}
      h_v
      &=-\frac{k}{2}\frac{1}{v-\bar{v}}+\frac{\ii\nu}{4}(\eta-\bar{\eta})^2,\\
       h_{v\bar{v}} & =
                      -\frac{k}{2}(v-\bar{v})^{-2}=\frac{k}{8}\frac{1}{y^2},
                      \\
      h_{v\bar{\eta}} & =-\ii \frac{\nu}{2}(\eta-\bar{\eta})=\nu p,\\
      h_{\eta} & =\frac{\ii \nu}{2} (\eta-\bar{\eta})(v-\bar{v}),\\
      h_{\eta\bar{\eta}} & =\nu y,
             \end{align}
    \end{subequations}
  and we get \eqref{hs} which is Hermitian. The conditions \eqref{EQK}  that the metric
    associated to \eqref{omM} be \Ka~ are met.
   \end{proof}

We have also obtained invariant metric to the action of the Jacobi
group $G^J_1(\R)$ on the extended Siegel-Jacobi upper half-plane $\tilde{\mc{X}}^J_1$  \cite[Proposition 5.6,  (5.25), (5.26)]{SB19}, see also
\cite[Proposition 4,  (69) ]{SB20}
\begin{Proposition}\label{PROP5}
The  three-parameter   metric of  the extended  Siegel-Jacobi upper
  half-plane  
  $\tilde{\mc{X}}^J_1$ expressed  in the  S-coordinates
  $(x,y,p,q,\kappa)$,  left-invariant with  respect  to  the action  of the Jacobi group
 $G^J_1(\R)$,  is given as 
  \begin{equation}\label{linvG}
    \begin{split}
 {\rm d} s^2_{\tilde{\mc{X}}^J_1}(x,y,p,q,\kappa) &
 \!=\!{\rm d} s^2_{\mc{X}^J_1}(x,y,p,q)+\lambda^2_6(p,q,\kappa)\\
 &\!=\!\frac{\alpha}{y^2}\big({\rm d} x^2+{\rm d}
 y^2\big)+\left(\frac{\gamma}{y}S+\delta q^2\right){\rm d} p^2+
 \left(\frac{\gamma}{y}+\delta p^2\right){\rm d} q^2 +\delta {\rm d} \kappa^2\\
& \!+ 2\left(\gamma\frac{x}{y}-\delta pq\right){\rm d} p{\rm d} q +2\delta (q{\rm d} p{\rm d}
\kappa-p{\rm d} q {\rm d} \kappa),
\end{split}
\end{equation}
where \cite[(5.15f), (5.17)]{SB19}
\[
\lambda_6=\sqrt{\delta}( \dd \kappa -p\dd q +q\dd p),
\]
and $S$ was defined in \eqref{AK}.

   The metric matrix associated  to the  metric \eqref{linvG} is
\begin{equation}\label{begGG}
g_{\tilde{\mc{X}}^J_1}\! = \!\left(\begin{array}{ccccc}g_{xx} &0 &0 &0&0\\
0& g_{yy}& 0& 0 & 0\\
0& 0& g_{pp} & g_{pq} & g_{p\kappa}\\0 & 0& g_{qp}& g_{qq}
 &g_{q\kappa}\\
0& 0& g_{\kappa p}& g_{\kappa q} & g_{\kappa\kappa}
 \end{array}\right),\!
 \begin{array}{cc}g_{xx}\!=\frac{\alpha}{y^2}, &
  \!g_{yy}\!=\!\frac{\alpha}{y^2},\\
 g_{pq}\!=\!\gamma\frac{x}{y}-\delta p q , &
~g_{p\kappa}\!=\!\delta q, g_{q\kappa}\!=\!-\delta p, \\
   g_{pp} \!=\!\gamma\frac{S}{y}+\delta q^2,&
 g_{qq}\!=\!\frac{\gamma}{y}+\delta p^2,   g_{\kappa\kappa}\!=\!\delta.
\end{array}
\end{equation}
The extended Siegel--Jacobi upper half-plane $\tilde{\mc{X}}^J_1$ does
not admit an almost contact structure $(\Phi,\xi,\eta)$ with a contact form $\eta=\lambda_6$ and Reeb vector $\xi= \operatorname{Ker}(\eta)$.
\end{Proposition}
Now some results of Proposition \ref{PRFC} are extended from
$\mc{D}^J_1$ and  $\mc{X}^J_1$
to  $\mc{D}^J_n$, respectively $\mc{X}^J_n$. Below  $k$, $2k\in \N$ indexes  the holomorphic discrete series of $\Sp$
  and $\nu>0$ indexes 
  the representations of the Heisenberg group.
Parts of the following Proposition are taken from \cite[Proposition
3, Theorem 1]{SB22}, see also 
\cite[Proposition 3]{nou},  \cite[Theorem 3.2]{SB15}: 
\begin{Proposition}\label{PRR3}
  a) The \Ka~two-form on $n(n+3)$-dimensional $\mc{D}^J_n$,  invariant to the action of
  $(G^J_n)_0$,  is 
 \begin{subequations}\label{NOUW1}
   \begin{align} \!-\!\ii\!\omega_{\mc{D}^J_n} (W,z) & \!=\!\frac{k}{2}\tr(B\wedge\bar{B})\!+\!\nu\tr(A^t\bar{M}\wedge\bar{A}),
 \!A (W,z)\!:=\!\dd z^t\!+\!\dd W\bar{\eta}, W\in\mc{D}_n,\\
B(W)  & := M\dd W, ~M  \!:=\!(\un
            -W\bar{W})^{-1},~z\in
            M(1,n,\C), ~\eta\in M(n,1,\C),
\end{align}
\end{subequations}

b) Using the partial Cayley transform
\begin{subequations}\label{PCT}
\begin{align}
\Phi^{-1} &: v=\ii (\un-W)^{-1}(\un+W); ~ u^t=(\un-W)^{-1}z^t, ~
              W\in\mc{D}_n,~ v\in \mc{X}_n;\\
\Phi &: W=(v-\ii \un)^{-1}(v+\ii \un), ~  z^t=2\ii(v+\ii \un)^{-1}u^t,  ~z,u\in M(1,n,\C),
\end{align} 
\end{subequations}
we get from the \Ka~ two-form on $\mc{X}^J_n$ depending on two parameters,
invariant to the action  of $G^J_n(\R)_0$: 
\begin{equation}\label{NOUW}
 -\ii \omega_{\mc{X}^J_n} (v,u) =
                         \frac{k}{2}\tr(H\wedge\bar{H})+\frac{2\nu}{\ii}\tr(G^tD\wedge\bar{G}),\quad
                         D :=(\bar{v}-v)^{-1},~H:=D\dd v.
                         \end{equation}
                       where\begin{equation}\label{XXV}
  G^t(v,u)=\dd u-p\dd v,
 \end{equation}
  and
  \begin{equation}\label{214}
    G^t(v,u)=G^t(x,y,p,q)=\dd p v+\dd q= \dd p(x+\ii y)+\dd q.
    \end{equation}

     c) Let $M(n,\C)\ni v=x+\ii y$ be a symmetric  positive definite
     matrix and $p,q\in M(n,1,\C)$.  The  three parameter metric on $\tilde{\mc{X}}^J_n$, invariant
      to the  $G^J_n(\R)$ action  
      is
      \begin{equation}\label{BIGM}\begin{split}
        \dd s_{\tilde{\mc{X}}^J_n}^2(x,y,p,q,\kappa)&=
        \dd s_{\mc{X}^J_n}^2(x,y,p,q)+ \lambda_6^2\\
         & =\alpha\tr[(y^{-1}\dd
                                      x)^2+(y^{-1}\dd y)^2]\\  
        & + \gamma[\dd p
        (xy^{-1}x+yy^{-1}y)\dd p^t+\dd q y^{-1}\dd q^t +2\dd p
        xy^{-1} \dd q^t]\\
        & +\delta (\dd \kappa -p \dd q^t+q \dd p^t)^2.
      \end{split}
      \end{equation}
   \end{Proposition}

 \section{Berry phase on \Ka~ manifolds}\label{BEP}

 \subsection{Balanced metric}

The starting point in  Perelomov's approach to coherent states (CS) is  the triplet
$(G,\pi,\got{H})$, where $\pi$ is a unitary, irreducible representation
of the Lie
group $G$ on a separable complex  Hilbert space $\got{H}$  \cite{perG}. 

Two types of CS-vectors belonging to  $\got{H}$ are locally defined on
$M=G/H$:  the normalized (un-normalized) CS-vector
 $\underline{e}_x$ (respectively, $e_z$) \cite[\S 6, Remark 4, (6.25)]{SB95}
 \begin{equation}\label{2.1}
\underline{e}_x=\exp(\sum_{\phi\in\Delta^+}x_{\phi}{\mb{X}}^+_{\phi}-{\bar{x}}_{\phi}{\mb{X}}^-_{\phi})e_0,
\quad e_z=\exp(\sum_{\phi\in\Delta^+}z_{\phi}{\mb{X}}^+_{\phi})e_0,
\end{equation}
where $e_0$ is the extremal weight vector of the representation $\pi$,
$\Delta^+$ is the set of positive roots
of the Lie algebra $\got{g}$, and   $X_{\phi}$, $\phi\in\Delta$
$X^+_{\phi}$   ($X^-_{\phi}$) 
 are the positive (respectively, negative) generators.

In the standard procedure of CS,
the  $G$-invariant \Ka~ two-form  on a $2n$-dimensio-\newline nal homogenous 
manifold $M=G/H$ is obtained from the \Ka~ potential $f$ via the recipe
\begin{subequations}\label{KALP}
  \begin{align}-\ii\omega_M & =\pa\bar{\pa}f, ~f(z,\bar{z})=\log
  K(z,\bar{z}), ~K(z,\bar{z}):=(e_{{z}},e_{{z}}),\label{KALP1}\\
\omega_M(z,\bar{z}) & =\ii \sum_{\alpha,\beta}h_{\alpha\bar{\beta}}\dd
  z_{\alpha}\wedge \dd \bar{z}_{\beta},~
  h_{\alpha\bar{\beta}}=\frac{\pa^2 f}{\pa z_{\alpha}\pa
      \bar{z}_{\beta}},~
                      h_{\alpha\bar{\beta}}=\bar{h}_{\beta\bar{\alpha}},~\alpha,\beta=1,\dots,n,\label{KALP2}
  \end{align}
  \end{subequations}
where   $K(z,\bar{z})$
is  the scalar product of two  un-normalized Perelomov's  CS-vectors $e_{{z}}$ at
$z\in M$  \cite{sbj,SB15, perG}.

It is well known, see \cite[Theorem
4.17]{ball}, \cite[Proposition 20]{SB19}, \cite[(6), p 156]{kn}, 
that the condition \begin{equation}\label{condH}\dd \omega=0\end{equation} for a Hermitian
manifold to have a \Ka~ structure is equivalent with the conditions
\begin{equation}\label{EQK}
  \frac{\pa h_{\alpha\bar{\beta}}}{\pa z_{\gamma}}= \frac{\pa
      h_{\gamma\bar{\beta}}}{\pa z_{\alpha}}, \quad\text{or}\quad \frac{\pa h_{\alpha\bar{\beta}}}{\pa z_{\gamma}}= \frac{\pa
      h_{\alpha\bar{\gamma}}}{\pa z_{\bar{\beta}}},\quad\alpha,\beta,\gamma =1,\dots,n.
 \end{equation}

   In accord with  \cite[p 42 ]{ball}, \cite[p 28]{green},
   \cite[Appendix B]{SB19},     the Riemannian metric associated with the Hermitian  metric
    on the manifold $M$ in local coordinates is 
    \begin{equation}\label{asm}\dd
    s^2_{M}(z,\bar{z})=\sum_{\alpha,\beta}h_{\alpha\bar{\beta}}\dd
    z_{\alpha}\otimes\dd \bar{z}_{\beta}.\end{equation}
  Sometimes \cite[(7.4)]{CH67}, if the  metric is taken as in
  \eqref{asm}, then 
  the \Ka-two form is taken instead of \eqref{KALP2} as 
\begin{equation}\label{KALP3}
    -\ii\omega_M=\frac{\ii}{2} \sum_{\alpha,\beta}h_{\alpha\bar{\beta}}\dd
  z_{\alpha}\wedge \dd \bar{z}_{\beta}.
 \end{equation}

This choice of $f$ in \eqref{KALP3}  corresponds to the situation where the so called
$\epsilon$-function  \cite{cahII, raw,Cah},
\begin{gather*}
\epsilon(z) := \e^{-f(z)}K_M(z,\bar{z}),
\end{gather*}
is constant. The corresponding $G$-invariant metric is called {\it balanced metric}. This denomination was firstly used in~\cite{don} for compact manifolds, then it was used in \cite{arr} for noncompact manifolds, also in~\cite{alo} in the context of Berezin quantization on homogeneous bounded domain, and we have used it in the case of the partially bounded domain $\mc{D}^J_n$ -- the Siegel--Jacobi ball~\cite{SB15}.

\begin{Remark}\label{CER}
  The \Ka~two-form $\omega_{\mc{D}^J_1}(w,\eta)$ given by \eqref{E32b}
  ($\omega_{}(v,\eta)$, \eqref{omSYUHP}) can be
  obtained from the \Ka~ potentials 
  \eqref{V1} (respectively \eqref{V2}) using a formula of the type \eqref{KALP2}
  \begin{subequations}
    \begin{align}
    f(w,\bar{w},\eta,\bar{\eta})
      &=-2k\log\frac{P}{w}+f(w)+g(\bar{w})+\nu\eta\bar{\eta}+f'(\eta)+g'(\bar{\eta}), \label{V1}\\
  f(v,\bar{v},\eta,\bar{\eta}) & =-2k\log\frac{v-\bar{v}}{2\ii} + f_1(v)+g_1(\bar{v})+\nu\eta\bar{\eta}+ f'(\eta)+g'(\bar{\eta}).\label{V2}
    \end{align}
    \end{subequations}
   \end{Remark}

\subsection{Berry phase on homogenous \Ka~ manifolds}

\begin{Proposition}\label{PR44}
  Let $H$ be the Hamiltonian of a quantum system
  $(\Psi, \got{H}, (,))$ on the homogeneous 
manifold $M=G/H$ governed by the
Schr\"odinger equation
\[
  H\Psi= \ii \dot{\Psi}.
\]
Let us introduce the notation
\begin{equation}\label{PPSI}
\Psi =e^{\ii \varphi}\tilde{e}_z, \quad \varphi\in [0,2\pi).
\end{equation}
Then the phase   $\varphi$ is the sum \cite{swA}
\[
  \varphi= \varphi_D+\varphi_B\]
    of the dynamical $\varphi_D$ and the non-adiabatic Berry phase 
    $\varphi_B$, where
    \[
      \varphi_D=-\int  {\mathcal  H}(t) \dd t,
    \]
     and ${\mathcal  H} $ is the energy function    attached to the Hamiltonian $H$
     \begin{equation}\label{ENN}
       {\mathcal H} =(\tilde{e}_z
        |H| \tilde{e}_z).\end{equation}

          The Berry phase is       the integral of the  one-form $A_B$,
          called  {\bf  Berry connection} 
    \begin{equation}\label{BF}
      \varphi_B = \oint A_B,
    \end{equation}
    where
    \begin{equation}\label{BCON} \begin{split}A_B& \!=\!\frac{\ii}{2}\sum_{\alpha\in
        \Delta_{+n}}
      (\dd z_{\alpha}\pa_{\alpha}\!-\!\dd {\bar{z}_{\alpha}}\bar{\pa}_{\alpha})\log
      (e_z,e_z)
      \!=\!-\Im \theta_L, \\  \theta_L&:\!=\! \sum_{\alpha\in
        \Delta_{+n}}\pa_\alpha f(z,\bar{z}) \dd {z}_{\alpha}\!=\!\sum_{\alpha\in
        \Delta_{+n}}\pa_{\alpha}\log(e_z,e_z) \dd z_{\alpha} \!=\!
      \sum_{\alpha\in
        \Delta_{+n}}\frac{\pa_{\alpha}(e_z,e_z)}{(e_z,e_z)}\dd z_{\alpha},
      \end{split}
      \end{equation}
     and $f$ is the \Ka~
     potential defined in \eqref{KALP1}.
     The Berry phase depend on the path  and not on the Hamiltonian.
     Closed paths in $M$ imply line integral
over connection on the closed paths and are obtained
through horizontal lift. If the motion is done on a
closed path in M, it generates in the fiber in M the holonomy
\begin{equation}\label{HOL}
\beta= \oint A_B=\int_S \dd A_B,
\end{equation}
where $\dd A_B$ is the curvature of the fiber bundle, a realisation of
the  two-form $V$ of Simon \cite{BS},
\begin{equation}\label{DDP3}
  \begin{split}
  \dd A_B& =\frac{\ii}{2}\sum_{\alpha,\beta} (-\frac{\pa^2f}{\pa {z}_{\beta}\pa
    \bar{z}_{\alpha}}\dd z_{\beta}\wedge\dd \bar{z}_{\alpha}
  +\frac{\pa^2 f}{\pa\bar{z}_{\beta}\pa z_{\alpha}}\dd
  \bar{z}_{\beta}\wedge \dd z_{\alpha})\\ &=
  -\ii \sum_{\alpha,\beta}\frac{\pa^2\log
    (e_z,e_z)}{\pa z_{\alpha}\pa \bar{z}_{\beta}}\dd z_{\alpha}\wedge
  \dd \bar{z}_{\beta}=-\omega_M(z,\bar{z}).
\end{split}
\end{equation}

  \begin{proof}
    The Proposition is taken from \cite[(4.17)]{sbcag},
        \cite[Proposition]{sbl}, \cite[corrected Proposition 4.1]{FC}, see also \cite[(15)]{GO}.  The
        expresion \eqref{BF}  of the Berry phase corresponds to the
        parallel transport, i.e. the vector 
        \begin{equation}\label{PSIBAR}
          |\underline{Z})=e^{\ii \varphi_B}\tilde{e}_z, \quad
\tilde{e}_z:=(e_z,e_z)^{-\frac{1}{2}}e_z
        \end{equation}
in \eqref{PPSI} has the property  that
 $(\underline{Z},\dot{\underline{Z}}) =0$
\cite[page 2365]{sbl} and in  the proof it is used  the relation
\begin{equation}\label{SUMP}
  \dot e_z=\sum_{\alpha}\frac{\pa e_z}{\pa
    z_{\alpha}}\dot  z_{\alpha}\quad \text{or}~~\dd e_{z}=\pa \dd e_z=
  \sum_{\alpha}\frac{\pa e_z}{\pa z_{\alpha}}\dd z_{\alpha}.
\end{equation}
The proof of the expression \eqref{DDP3}
is a consequence of the relation:
\[
  f= \sum_{j=1}^nf_j\dd x_j\Longrightarrow \dd f= \sum_{j=1}^n\sum_{i=1}^n
  \frac{\pa f_j}{\pa x_i}\dd x_i\wedge\dd x_j,
\]
where $f$ is a smooth  function $x_1,\dots,x_n$. The last expression
of $\dd A_B$ in \eqref{DDP3} is in the convention \eqref{KALP2}.

\end{proof}
\begin{Remark}\label{RRR}
  Equation \eqref{BCON} of $A_B$  can be written with formula \eqref{SUMP}
  as
  \begin{equation}\label{ruc}
    A_B=-\Im\frac{(e_z|\pa|e_z)}{(e_z,e_z)}, ~~\pa f=\frac{\pa f}{\pa
        z_{\alpha}}\dd z_{\alpha}.
    \end{equation}

    Equation  \eqref{ruc} is exactly \cite[(16) page 10]{swA} or \cite[(2.56)]{CH}
    \begin{equation}
A^{(n)}=- \Im\frac{<n|\pa|n>}{<n|n>}, \quad <n|n>=1.
   \end{equation} 

Equation \eqref{DDP3} of $\dd A_B$ can be written as
    \[
      \dd A^{(n)} =-\Im \frac{(\pa e_z|\wedge|\pa e_z)}{(e_z,e_z)}.
      \]
Equation \eqref{DDP3} of $\dd A^{(n)}$ can be written as  \cite[(13) page
  10]{swA} or  \cite[(2.62), (2.63)]{CH}
\begin{equation}
  \begin{split}
    F^{(n)}& =\dd A^{(n)}=-\Im\frac{<\dd n|\wedge|\dd n>}{<n|n>}
      =\frac{1}{2}\frac{F^{(n)}_{ij}}{<n|n>}\dd
x_i\wedge \dd x_j\\ &=
-\Im (\frac{<\pa_i n|\pa_j n>-<\pa _jn|\pa_i n>}{<n|n>}) \dd x_i\wedge \dd x_j, <n|n>=1.
\end{split}
\end{equation}
\end{Remark}
\end{Proposition}

\subsection{Linear Hamiltonian in the generators of the Jacobi  group
  $G^J_1(\R)$}\label{43} The content of the following Remark is mostly
extracted from \cite[\S~4, Lemma 4.1]{FC}, \cite[\S~4.3]{SB22}

\begin{Remark}
  Let us consider a linear Hermitian Hamiltonian $\mb{H}$
in the generators  of the Jacobi group $G^J_1$
\begin{equation}\label{LH}
  \mb{H}=\epsilon_a\mb{a}+\bar{\epsilon}_a\mb{a}^{\dagger}+\epsilon_0\mb{K}_0+\epsilon_+\mb{K}_+
  +\epsilon_-\mb{K}_-,
\end{equation}
where
\[
  \bar{\epsilon}_+=\epsilon_-,~~ \epsilon_a:=a+\ii b,~~\epsilon_+:=m-\ii n, ~~\epsilon_0:=2c,~~
  a,b,c,m,n \in\R.
  \]
The energy function $\mc{H}$ \eqref{ENN} associated to the
 Hamiltonian \eqref{LH} expressed in the variables $(\eta, v)$
 splits into  the sum of two independent functions
\begin{equation}\label{hsum}
  \mc{H}(\eta,v)=\mc{H}(\eta)+\mc{H}(v),\quad v=x+\ii y, ~y>0,~\eta=q+\ii p, 
\end{equation}
where
\begin{subequations}\label{sup}
\begin{align}
\mc{H}(q,p) &=\nu[(m+c)q^2+(c-m)p^2-2nqp+2(aq+bp)],\\
\mc{H}(x,y) & =\frac{k}{y}[(m+c)(x^2+y^2)-2(nx+cy)+c-m]+2kc. 
\end{align}
\end{subequations}
We particularize equations \cite[(3.7)]{SB22} to the  linear
Hamiltonian \eqref{hsum}  to which we add a term  ${h}(\kappa)$
\begin{equation}\label{HPARL}
  \mc{H}=\mc{H}(p
 ,q)+\mc{H}(x,y)+ {h}(\kappa),
\end{equation}
and we get the equations of motion  on the extended
   Siegel-Jacobi upper half-plane  organized as generalized transitive
   almost  cosymplectic
   manifold $(\tilde{\mc{X}}^J_1,\theta,\omega)$ corresponding to the energy
   function \eqref{HPARL}  
   \begin{subequations}\label{corH}
  \begin{align}
    \dot{x} & = (c+m)(-x^2+y^2)+2nx-c+m,\quad
    \dot{y} =   -2y[(c+m)x-n],\label{corHxy}\\
    \dot{q} & = (c-m)p  -qn+b-\frac{q}{2\nu}\frac{\pa
              h}{\pa \kappa}, \quad \dot{p} =
              \!-\!(m+c)q+\!np-\!a\
   \!-\!\frac{p}{2\nu}\frac{\pa{h}}{\pa \kappa}, \quad
     &  \label{corHpq}\\
    \dot{\kappa} & =
  (c+m)q^2+(c-m)p^2+aq+bp\!-2npq-\!\frac{1}{\sqrt{\delta}}
                   \mc{H}\\~~&=
    (m+c)[(1-\frac{\nu}{\sqrt{\delta}}) q^2-\frac{k}{\sqrt{\delta}} \frac{x^2+y^2}{y}]
    +(c-m)[(1-\frac{\nu}{\sqrt{\delta}})
                               p^2-\frac{k}{\sqrt{\delta}}\frac{1}{y}]\\ &
    -2(1-\frac{\nu}{\sqrt{\delta}}) npq+(1-\frac{2}{\sqrt{\delta}})(aq+bp)-2 \frac{k}{\sqrt{\delta}}\frac{x}{y}-\frac{1}{\sqrt{\delta}}h.\nonumber
    \end{align}
  \end{subequations}
 \end{Remark}
\begin{proof}
The differential action of the generators 
the Jacobi algebra   $\got{g}^J_1$ 
 is given by the formulas:
  \begin{subequations}\label{summa}
\begin{eqnarray}
& & \mb{a}=\frac{\pa}{\sqrt{\mu}\pa z};~\mb{a}^{\dagger}=\sqrt{\mu} z+w\frac{\pa}{\sqrt{\mu}\pa z} ;\\
 & & \db{K}_-=\frac{\pa}{\pa w};~\db{K}_0=k+\frac{1}{2}z\frac{\pa}{\pa z}+
w\frac{\pa}{\pa w};\\
& & \db{K}_+=\frac{1}{2}\mu z^2+2kw +zw\frac{\pa}{\pa z}+w^2\frac{\pa}{\pa
w} ,
\end{eqnarray}
\end{subequations}
where $z\in\C$, $|w|<1$.

Then  with \eqref{csu}, we get
\begin{equation}
  \frac{\pa K}{\pa z}= \nu\bar{\eta}K, \quad \frac{\pa K}{\pa w}=(2k\frac{\bar{w}}{1-w\bar{w}}+\frac{\nu}{2}\bar{\eta}^2)K,
  \end{equation}
  and $\eta$ was defined in \eqref{ENN}.

       With \eqref{sup} we find
 \begin{subequations}\label{328}
 \begin{align}
 \frac{\pa H (x,y)}{\pa x}=&\frac{2 k}{y}[(m\!+\!c)x\!-\!n];
 \frac{\pa H (x,y)}{\pa y}=\!\frac{k}{y^2}[(m\!+\!c)(y^2-x^2)\!+\!m\!-\!c\!+\!2nx];\\
\frac{\pa H (q,p)}{\pa q}=
        &2\nu[(m\!+\!c)q\!+\!np\!+\!a];\frac{\pa H (q,p)}{\pa p}\!=\!2\nu[(c\!-\!m)p\!+\!nq\!+\!b].                           
 \end{align}
 \end{subequations}

 \cite[(4.5)-(4.7)]{SB22}  and \eqref{328} imply
  \begin{subequations}
   \begin{align}
     \dot{q}_1&=k \dot{x}=A_1=y^2\frac{\pa H}{\pa y}=k[(m+c)(y^2-x^2)+m-c+2nx]\\
     \dot{q}_2&=2\nu\dot{q}=A_2=\frac{\pa H}{\pa
                p}-\frac{q}{2\nu}\frac{\pa H}{\pa \kappa}=
                2\nu[{(c-m)p-nq+b}-q]-\frac{q}{2\nu}\frac{\pa H}{\pa \kappa}\\
     \dot{p}_1&=y^{-2}\dot{y}=B_1=-\frac{1}{k}\frac{\pa H}{\pa x}=-\frac{1}{y}[(c+m)x-n]\\
     \dot{p}_2&=\dot{p}=B_2=-\frac{\nu}{2\nu}[\frac{\pa H}{\pa
                q}+p\frac{\pa H}{\pa
                k}]=-\frac{2}{2\nu}[(c+m)q-np+a]-\frac{1}{2\nu}p\frac{\pa
                H}{\pa k}\\
     \dot{\kappa}&=\frac{1}{2\nu}(p\frac{\pa H}{\pa p}+q\frac{\pa
                   H}{\pa q})-\frac{1}{\sqrt{\delta}}H,
      \end{align}
    \end{subequations}
  and \eqref{corH} are  proved.
  \end{proof}

        \subsection{\Ka~two-forms,  Christofell's symbols  and
          connections matrices on $\mc{D}_1$ and $\mc{X}_1$}\label{33}
       a) The \Ka~two form on the Siegel disk  $\mc{D}_1$ corresponding
        to the \Ka~ potential \[ f(w)=-2k\log P  
        \]
        is \cite[(7.21)]{jac1}
        \begin{equation}\label{jac111}
          -\ii\omega(w,\bar{w})=\frac{2k}{P^2}\dd w\wedge \dd \bar{w}. 
     \end{equation}
      If we make the change of variables \eqref{210b} $w\rightarrow v$
      and apply
\begin{equation}\label{dwv}
   \dd w=2\ii \frac{\dd v }{(v+i)^2}, 
 \end{equation}
      we get
      \begin{subequations}\label{OMMM}
        \begin{align}
          -\ii \omega(v,\bar{v}) & = -\frac{2k}{(v-\bar{v})^2}\dd v\wedge \dd \bar{v},= \frac{k}{2y^2}\dd v\wedge \dd \bar{v},\label{314a}\\
          \omega(x,y) &= \frac{k}{y^2}\dd x\wedge \dd y, ~~v=x=\ii y.
        \end{align}
        \end{subequations}
        Alternatively, if in \eqref{jac111} we introduce \eqref{WVAB} and
        then \eqref{DALPHA}, we  get again \eqref{OMMM}.

      b)   From \eqref{kk1} we get
        \[
          h_{\mc{D}^J_1}(w)=\frac{2k}{P^2}.
        \]
        With formula \eqref{CRISTU} we get
        \begin{equation}\label{WWW}
          \Gamma^w_{ww}=\frac{2\bar{w}}{P}
        \end{equation}
        We consider formula \eqref{ULTRAN} of the transformation
        $w
        \rightarrow v$ and we write  this change of variables as
        \begin{equation}\label{dir}
         \Gamma^v_{vv}=\Gamma^w_{ww}\frac{\pa w}{\pa v}+ \frac{\pa^2
           w}{\pa v^2}\frac{\pa v}{\pa w}.
          \end{equation}
          With \eqref{dwv} we get $\frac{\pa w}{\pa v} $ and then
          \[
            \frac{\pa^2 w}{\pa v^2} =-\frac{4\ii}{(v+\ii)^3} .
          \]
          We get with \eqref{dir}
          \begin{equation}\label{VVV}
            \Gamma^v_{vv}=\frac{2}{\bar {v}-v}=\frac{\ii}{y},
          \end{equation}
          which is correct, because if we apply \eqref{CRISTU} to
          \begin{equation}
            h_{\mc{X}^J_1}(v)=-\frac{2k}{(v-\bar{v})^2}=\frac{k}{2y^2},
            \end{equation}
            we get \eqref{VVV}.

We have \begin{equation}\label{VBARV}
v-\bar{v}=2\ii \frac{P}{(1-w)(1-\bar{w})}.
\end{equation}

            Inverse, if we consider the change of variables
            $v\rightarrow w$ and apply the inverse formula to
            \eqref{dir}
            \[
              \Gamma^w_{ww}=\Gamma^v_{vv} \frac{\pa v}{\pa w}
              +\frac{\pa^2v}{\pa w^2} \frac{\pa w }{\pa v}.
            \]
            starting with \eqref{VVV} we get \eqref{WWW}.

c) With formula \eqref{VVV} we get
\begin{equation}\theta^v_v(v)=\Gamma^v_{vv}\dd
  v=\frac{2}{v-\bar{v}}\dd v.
 \end{equation}
  We apply relation \eqref{DD1} to the change of variables
  $v \rightarrow w$ \eqref{ULTRAN}. We have
 \[
A(w)=
  \frac{2\ii}{(1-w)^2}   ,\quad \dd A A^{-1}=2\frac{\dd w }{1-w}    
\]
and     we should have
\begin{equation}
\omega'=\omega_{J^1}(w)=2\frac{\dd
  w}{1-w}+\frac{2}{v-\bar{v}}\frac{2\ii}{(1-w)^2}\dd w.
\end{equation}
With \eqref{VBARV} we find
\[\theta^w_{ww} =\Gamma^w_{ww}\dd w= 2\frac{\bar{w}}{1-w\bar{w}}\dd w,
\]
in accord with \eqref{WWW}.
\subsection{Connection matrices  on $\mc{D}^J_1$ and $\mc{X}^J_1$ }\label{CON}

Christoffel's symbols $\Gamma$-s have  the expressions \cite[(38)]{GAB}
\begin{equation}\label{GAMM}
\begin{split}
\Gamma^z_{zz}  & =-\lambda\bar{\eta};~\Gamma^w_{zz}=\lambda; ~
\Gamma^z_{zw}=-\lambda\bar{\eta}^2+\frac{\bar{w}}{P};\\
 \Gamma^w_{wz} & =\lambda\bar{\eta};~
 \Gamma^z_{ww}=-\lambda\bar{\eta}^3; ~ \Gamma^w_{ww} =
 \lambda\bar{\eta}^2+2\frac{\bar{w}}{P}, ~ \lambda = \frac{\nu}{2k}.
\end{split}
\end{equation}

The connection matrix  (form) $\theta_{\mc{D}^J_1} $ on $\mc{D}^J_1$
in $(w,z)$ 
is \cite[(40)]{GAB}
\begin{subequations}
  \begin{align}\theta_{\mc{D}^J_1}(w,z):& =\left(\begin{array}{cc} \theta^w_w & \theta^z_w\\ \theta^w_z &
                                                                  \theta^z_z\end{array}\right)
  =\left(\begin{array}{cc} \Gamma^w_{wz}\dd
                                                                          z+\Gamma^w_{ww}\dd w &
                                                                                                 \Gamma^z_{wz}\dd z+\Gamma^z_{ww}\dd w\\
                                                                                   \Gamma^w_{zz}\dd
                                                                       z+\Gamma^w_{zw}\dd
                                                                       w& \Gamma^z_{zz}\dd
                                                                       z
                                                                       +\Gamma^z_{zw}\dd
                                                                       w
                                                                         \end{array}\right)\label{33ab63}\\
~~& =\left(\begin{array}{cc} \lambda\bar{\eta}\mc{A}+2\frac{\bar{w}}{P}\dd
                                               w
                          & -\lambda\bar{\eta}^2\mc{A} +\frac{\bar{w}}{P}\dd
                          z\\ \lambda \mc{A} & -\lambda\bar{\eta}\mc{A}+\frac{\bar{w}}{P}\dd
                          w
                          \end{array}\right). 
  \end{align}\end{subequations}

In the variables $(u,v)\in(\C,\mc{X}_1)$ the  geodesic equations 
\eqref{geoD1}  for the metric \eqref{kmb} read \cite[(57)]{SB21}
\begin{equation}\label{geomic}
 \left\{
 \begin{array}{l}
 \frac{\dd^2 u}{\dd t^2}+\Gamma^u_{uu}\left(\frac{\dd u}{\dd
     t}\right)^2 +
2\Gamma^u_{uv}\frac{\dd u}{\dd t}  \frac{\dd v}{\dd t} +\Gamma
^u_{vv}\left(\frac{\dd v}{\dd t}\right)^2 =0   ;\\
  \frac{\dd^2 v}{\dd t^2}+\Gamma^v_{uu}\left(\frac{\dd u}{\dd
     t}\right)^2 +
2\Gamma^v_{uv}\frac{\dd u}{\dd t}  \frac{\dd v}{\dd t} +\Gamma
^v_{vv}\left(\frac{\dd v}{\dd t}\right)^2=0 .
     \end{array}
 \right.
\end{equation}
      
  Christoffel's symbols $\Gamma$-s  in $(u,v)$ corresponding to the
  Riemannian metric associated to the \Ka~two-form \eqref{BFR} are
  extracted from \cite[(62)]{SB21} with corrections

\begin{equation}\label{XGAMMM}
\begin{split}
\Gamma^u_{uu}  & =\frac{\ii}{\iota}r,~\Gamma^v_{uu}=\frac{\ii}{\iota}, ~
\Gamma^u_{uv}= \frac{\ii}{2\iota}(\frac{\iota}{y}-2r^2);\\
 \Gamma^v_{vu} & =-\frac{\ii}{\iota}r,~
 \Gamma^u_{vv}=\frac{\ii}{\iota}r^3, ~ \Gamma^v_{vv} =
\frac{\ii}{\iota}(\frac{\iota}{y}+r^2),
\end{split}
\end{equation}
where
\begin{equation}\label{CUCUV}
  \iota=\frac{k}{\nu}=\frac{1}{2\lambda},\quad r=\frac{n}{y}. 
\end{equation}
Equations \eqref{geomic} with the $\Gamma$-s \eqref{XGAMMM} lead to the same equations as \cite[(53)]{SB21}
\begin{subequations}\label{ECIV_2}
\begin{align}
& \ddot{v}+\frac{\ii}{\iota}\left[\dot{u}^2-2r\dot{u}\dot{v}+
(\frac{\iota}{y}+r^2)\dot{v}^2\right]=0,\\
& \ddot{u}+\frac{\ii}{\iota}\left[
  r\dot{u}^2+(\frac{\iota}{y}-2r^2)\dot{u}\dot{v}+r^3\dot{v}^2
 \right]=0.
\end{align}                                             
\end{subequations}

We check the value of $\Gamma^u_{uv}$. We have
\[
  \left\{
    \begin{array}{cc}h_{v\bar{u}}\Gamma^v_{vu}+h_{u\bar{u}}\Gamma^u_{vu}&
 =\frac{\pa h_{u\bar{u}}}{\pa v}\\h_{v\bar{v}}\Gamma^v_{vu}+h_{u\bar{v}}\Gamma^u_{vu} & =\frac{\pa
          h_{u\bar{v}}}{\pa v}\end{array} \right. .\]
\[
  \left\{
    \begin{array}{cc}-\frac{\nu r }{y}\Gamma^u_{vu}+\frac{\nu}{y}\Gamma^u_{vu}&
 =\frac{\ii \nu}{2y^2}\\ (\frac{k}{2y^2}+\frac{\nu
      r^2}{y})\Gamma^v_{vu}-\frac{\nu r}{y}\Gamma^u_{vu} & =-\frac{\ii
                                                           \nu r}{y^2}
    \end{array} \right. .\]
With the notation
\[
    \Delta=-\left\|\begin{array}{cc}h_{u\bar{u}}& h_{v\bar{u}}\\
               h_{u\bar{v}}& h_{v\bar{v}}\end{array} \right\|=-\frac{\nu
               k}{2y^3};\] \[\Delta_1= \left\| \begin{array}{cc}\frac{\ii
                                        \nu}{2y^2}& \frac{\nu
                                                    }{y}\\-\frac{\ii
                                        \nu r}{y^2} & -\frac{\nu r}{y}
                                                     \end{array}\right\|=\frac{\ii
                                                  \nu ^2r}{2y^3};\]
   \[\Delta_2=-\left\|\begin{array}{cc}\frac{\ii \nu}{2y^2}
                       &-\frac{\nu r}{y}\\-\frac{\ii \nu
                       r}{y^2}&\frac{k}{2y^2}+\frac{\nu r^2}{y}
                     \end{array}\right\|
                   =-\ii \frac{\nu k}{4y^4} +\frac{\ii\nu^2r^2}{2y^3},
                 \]
                 we get
              \[
      \Gamma^v_{vu}=\frac{\Delta_1}{\Delta}=-\frac{\nu r}{k};
      ~~~\Gamma^u_{vu}=\frac{\ii}{2y}-\frac{\ii}{\iota}r^2
      .    \]

    We also have
    \[
      \Gamma^u_{vu}=h^{v\bar{u}}\frac{\pa h_{u\bar{v}}}{\pa
        v}+h^{u\bar{u}}\frac{\pa h_{u\bar{u}}}{\pa v}=
      \frac{2y^2r}{k}\frac{\pa }{\pa v}(-\frac{\nu
        r}{y})+[\frac{y}{\nu}+\frac{\nu r^2y^2}{k}\frac{\pa}{\pa v}(\frac{\nu}{y})]=\frac{1}{2}\frac{\ii}{\iota}(\frac{\iota}{y}-2r^2).
      \]

  We get for the connection matrix on $\mc{X}^J_1$ in the variables
  $(v,u)$ the expression  
  \begin{subequations}\label{338}
    \begin{align}
    \theta_{\mc{X}^J_1}(v,u) &=\left(\begin{array}{cc}\theta^v_v&
                                                             \theta^u_v\\
                                \theta^v_u&
                                            \theta^u_u\end{array}\right)=\left(
                                            \begin{array}{cc}
                                             \Gamma^v_{vu}\dd
                                                                       u+\Gamma^v_{vv}\dd v &    \Gamma^u_{vu}\dd
                                              u+\Gamma^u_{vv}\dd v   
                                              \\ \Gamma^v_{uu}\dd
                                                                    u+
                                                                    \Gamma
                                                                    ^v_{uv}\dd
                                              v & \Gamma^u_{uu}\dd
                                              u+\Gamma^u_{uv}\dd v 
                                             \end{array}\right) \label{338a} \\
                     ~~ &\! =\!\frac{\ii}{\iota}
                     \left(\begin{array}{cc}
                       -r\mc{B}+\frac{\iota}{y}\dd v    &
                                                          -r^2\mc{B}+\frac{\iota}{2y}
                                                         \\
  \mc{B}  & r\mc{B}+\frac{\iota}{2y}\dd v
                                      \end{array}\right)\! \\ &=\!\frac{\ii}{\iota}\!
                     \left[\left(\begin{array}{cc} -r &-r^2 \\1 &
                                                         r\end{array}\right)\mc{B}\!+\!
  \frac{\iota}{2y}\left(\begin{array}{cc} 2 \dd v& \dd u \\0 & \dd
                                                              v\end{array}\right)\!\right]. \label{338b}
                                                           \end{align}
  \end{subequations}
  Now we calculate the connection matrix on the Siegel-Jacobi disc in
  the variables $(x,y,q,p)$.

The non-zero  Christoffel's symbols corresponding to the Riemannian
metric
$\dd\!s^2_{\mc{X}^J_1}\!(\!x\!,y\!,\!q,\!p\!)$
\eqref{METRS2} on
 the Siegel-Jacobi upper half-plane are \cite[(73)]{SB21}
\begin{equation}
  \begin{array}{lllll}\label{GSC}
    \Gamma^{x}_{xy}=-\frac{1}{y}
  &\Gamma^{x}_{pp}=-\ep xy
  &\Gamma^{x}_{pq}= -\frac{1}{2}\ep y & & \\
   \Gamma^y_{xx}= \frac{1}{y}&
 \Gamma^y_{yy}=-\frac{1}{y}&\Gamma^y_{pp}=\frac{\ep}{2}(x^2\!\!-\!\!y^2)&\Gamma^y_{pq}=\frac{\ep}{2}x
  &\Gamma^y_{qq}=\frac{\ep}{2} \\
  \Gamma^{p}_{xp}=\frac{1}{2}\frac{x}{y^2} &  \Gamma^{p}_{xq}=
 \frac{1}{2}\frac{1}{y^2}&\Gamma^{p}_{yp}= \frac{1}{2y} & & \\
  \Gamma^q_{xp}=\frac{y^2-x^2}{2y^2} &\Gamma^q_{xq}= -\frac{x}{2y^2}&\Gamma^q_{yp}=-\frac{x}{y}&\Gamma^q_{yq}= -\frac{1}{2y} &  \end{array} ,
\end{equation}
where
\begin{equation}\label{Tep}
\epsilon=\frac{\gamma}{\alpha}=2\frac{\nu}{k}
.\end{equation}

Let
\begin{equation}\label{thetam}
    \!\theta_{\mc{X}^J_1} (x,y,q,p)\!=
    \!\left(\begin{array}{cccc}\theta^ x_x &\theta^ x_y &\theta^ x_q &\theta^ x_p \\
   \theta^y_x & \theta^y_y &\theta^y_q& \theta^y_p\\
  \theta^q_ x&         \theta^q_y&  \theta^q_q  &\theta^q_p\\
  \theta^p_x &  \theta^p_ y&  \theta^p_ q&  \theta^p_p 
          \end{array}\right)\! .\end{equation}
We find for the matrix elements of \eqref{thetam} the values
\begin{subequations}\nonumber
  \begin{align}
 \theta_{\mc{X}^J_1}  =&\left(\begin{array}{cccc}\!-\!\frac{\dd y}{y} &
 \!-\!\frac{\dd x}{y}& \!-\frac{\ep}{2}y\dd p& \!-\ep xy \dd p\!-\!\frac{\ep
                                                      y}{2}\dd q \\
                              \frac{\dd x}{y}&\!-\!\frac{\dd y}{y} &
                                                                  \frac{\ep}{2}\dd q    
                                                                 \!+\!\frac{\ep}{2}x\dd p &
                                                                  \!\frac{\ep}{2}x\dd
                                  q\!+\!\frac{\ep}{2}(x^2-y^2)\dd p\\
  \!-\!\frac{x}{2y^2}\dd q\!+\!\frac{y^2\!-\!x^2}{2y^2}\dd p 
                                  &\!-\!\frac{x}{y}\dd y
  &\! -\!\frac{x}{2y^2}\dd x &
  \frac{y^2\!-\!x^2}{2y^2}\dd x\!-\!\frac{x}{y}\dd y\\
                                  \frac{x}{2y^2}\dd p\!+\!\frac{1}{2y^2}
\dd q & \frac{1}{2y}\dd p& \frac{1}{2y^2} \dd x &
                                                     \frac{x}{2y^2}\dd
                                                     x\!+\!\frac{1}{2y}\dd y                               \end{array}\right)\\
                                                    =&
    \left(\begin{array}{cccc} (\!0\!,-\frac{1}{y},\!0\!,\!0)
    & (-\frac{1}{y}, \!0\!,\!0,\!0)
    &(\!0,\!0,\!0,\!-\!\frac{\ep y}{2})
    &(\!0\!,\!0\!,-\!\frac{\ep}{2}y,-\ep xy)\\
     (\frac{1}{y}\!, \!0\!,\!0\!,\!0)
 & (\!0\!,\!-\frac{1}{y}\!,\!0\!,\!0\!)
 &(\!0\!,\!0\!,\!\frac{\ep}{2}\!,\!\frac{\ep x}{2})\!
 &(\!0\!,0\!,\frac{\ep x}{2}\!,\!
   \frac{\ep}{2} (x^2\!-\! y^2))\\
     (\!0\!,\!0\!,-\frac{x}{2y^2}\!,\!\frac{y^2\!-\!x^2}{2y^2})
   &(\!0\!,-\frac{x}{y},\!0\!,\!0\!)
   & (\!-\frac{x}{2y^2},\!0\!,\!0\!,\!0\!)&
                              (\frac{y^2\!-\!x^2}{2y^2},-\frac{x}{y},\!0\!,\!0)\!\\
                                                                       (\!0,\!0,\frac{1}{2y^2}\!,\!\frac{x}{2y^2})&
                                                                                                          (\!0\!,\!0\!,\!0\!,\!\frac{1}{2y})
                                                                       &(\frac{1}{2y^2},\!0\!,\!0\!,\!0)\!&
                                                                                                            (\frac{x}{2y^2},\frac{1}{2y},\!0\!,\!0\!)\end{array}\right)
 \!\otimes\! \left(\begin{array}{c}\dd x\\ \dd y \\ \dd q\\ \dd  p\end{array}\right) \end{align}\end{subequations}

\subsection{Check of formulae \eqref{522a}}.

  We write \eqref{522a} as
          \[
            \omega'(v,u)=\dd J_{WU} J^{-1}_{WU}+J_{WU}\omega(w,z)J^{-1}_{WU},
        \]
        where $(w, z)$ ($(v,u)$)  are the ``old'' coordinates
        (respectively ``new'' coordinates).
        If denote $J_{WU}$ with $A=A((w,z)\rightarrow (v,u))$, we have
        \[
          \left(\begin{array}{c}\frac{\pa}{\pa v}\\ \frac{\pa}{\pa
                  u}\end{array}\right)=
            A \left(\begin{array}{c}\frac{\pa}{\pa w}\\ \frac{\pa}{\pa
                      z}\end{array}\right)
              =            \left(\begin{array}{cc}
                                                \frac{\pa w}{\pa v} &
                                                                      \frac{\pa
                                                                      z}{\pa
                                                                      v}\\\frac{\pa
                                                w}{\pa u} & \frac{\pa
                                                            z}{\pa
                                                            u}\end{array}\right)
                                                      \left(\begin{array}{c}\frac{\pa}{\pa
                                                              w}\\
                                                              \frac{\pa}{\pa
                                                              z}\end{array}\right)  \]
 For the partial  Cayley  transform \eqref{210a} we find
 \begin{equation}\label{341B}
   A=\left(\begin{array}{cc} \frac{2\ii}{(v+\ii)^2} &-\frac{2\ii
                                                      u}{(v+\ii)^2}
           \\0  & \frac{2\ii}{v+\ii}\end{array}\right); ~~~
      A^{-1}=-\frac{\ii}{2}(v+\ii)
      \left(\begin{array}{cc}v+\ii & u\\0 &1\end{array}\right).
    \end{equation}
    
    \begin{equation}\label{341bb}
      \dd A=
       \frac{2\ii }{(v+\ii)^2}\left(\begin{array}{cc}-\frac{2\dd
                                      v}{v+\ii}&\frac{-(v+\ii)\dd
                                                u+2u\dd
                                                v}{v+\ii}\\ 0
                                              &-\dd
                                                v\end{array}\right).\end{equation}
  \begin{equation}\label{AAM}
\dd A\cdot  A^{-1}= -\frac{1}{v+\ii}\left(\begin{array}{cc}2 \dd v &
                                                                     \dd u
                                            \\ 0 & \dd v\end{array}\right)
                                      \end{equation}
\begin{equation}\label{342}
  J_{UW}\omega(w,z)J^{-1}_{UW}\!=\!\left(\begin{array}{cc}
\lambda\mc{A}(\bar{\eta}-u)\!+\!2\frac{\bar{w}}{P}\dd w &
 \frac{-\lambda\mc{A}(u-\bar{\eta})^2\!+\!\frac{\bar{w}}{P}(u\dd w+\dd z)}{v+\ii} \\
  \lambda\mc{A}(v+\ii) &
                     \lambda\mc{A}(u-\bar{\eta})\!+\!\frac{\bar{w}}{P}\dd w  \end{array}   \right) .
                        \end{equation}
                        \begin{equation}\label{344}
                          \omega'(v,u)
                          =               \!\left(\begin{array}{cc}
\lambda\mc{A}(\bar{\eta}-u)\!+\!2\frac{\bar{w}}{P}\dd w -\frac{2\dd
             v}{v+\ii}&
 \frac{-\lambda\mc{A}(u-\bar{\eta})^2\!+\!\frac{\bar{w}}{P}(u\dd w+\dd z)}{v+\ii} -\frac{\dd u}{v+\ii}\\
  \lambda\mc{A}(v+\ii) &
                     \lambda\mc{A}(u-\bar{\eta})\!+\!\frac{\bar{w}}{P}\dd
                         w-\frac{\dd
                         v}{v+\ii}  \end{array}   \right) .
                         \end{equation}
                        
                                  Now we compare equations
                                  \eqref{338}, \eqref{344}.

                                  We
                                  compare firstly the terms ``21''.
                  Because of \eqref{BFR2}, \eqref{CUCUV}, it is verified that \[\frac{\ii}{\iota}\mc{B}=\lambda\mc{A}(v+\ii),
                                  \]
                                  We compare the terms ``11''.
  We should have
     \[
       2\frac{\bar{w}}{P}\dd w -2\frac{\dd v}{v+\ii}=\ii \frac{\dd v}{y}
       \]
       which is true because of \eqref{210a} and \eqref{VBARV}.

       We compare the terms ``12''.  We should have
       
  \begin{equation}\label{1212}
\frac{\ii}{2y}\dd u= - \frac{\dd u}{v+\ii}+
\frac{\bar{w}}{P}\frac{u\dd w+\dd z}{v+\ii},
\end{equation}
which  is true.

Indeed, from \eqref{210a} we have 
  \[
\dd u=\frac{(1-w)\dd z+z\dd w}{(1-w)^2}, \quad v+\ii=\frac{2\ii}{1-w}, 
    \]
    
    and also we use \[ y=\frac{1-w\bar{w}}{(1-w)(1-bar{w})}.\]
    We compare   the terms ``22'' . We should
                                  have
\[
 \frac{\bar{w}}{P}\dd w-\frac{dv}{v+\ii}=\ii \frac{\dd v}{2y},
\]
\[
  \frac{\bar{w}}{P}\dd w =
  (\frac{\ii}{2y}+\frac{1}{v+\ii}) \dd v,
\]
which is true because of
 of \eqref{210}, \eqref{VBARV}.

\subsection{Berry phase on $\mc{D}^J_1$ and $\mc{X}^J_1$ }

a) Firstly we calculate the Berry phase on $\mc{X}_1$ from the Berry phase on
$\mc{D}_1$.

The relations\eqref{210b},  \eqref{WVAB}, \eqref{TAUZ} and \eqref{NNN}
implies
\begin{equation}\label{ALFAB}
  \alpha= \frac{x^2+y^2-1}{N},\quad \beta
  =-2\frac{x}{N},\quad
 \end{equation}
With 
\[
  f=-2k \log P=-2k\log (1-|w|^2),
\]
we get with \eqref{BCON} the Berry phase on the Sigel disk $\mc{D}_1$
\begin{equation}\label{ThEETA}
A_B (w,\bar{w})=\ii k\frac{\bar{w}\dd w-w\dd \bar{w}}{P}=
2k \frac{\beta\dd \alpha-\alpha \dd \beta}{P}.  
\end{equation}
But from \eqref{ALFAB} 
\begin{subequations}\label{DALPHA}
  \begin{align}
  \dd \alpha & = 2\frac{2x(y+1)\dd x+[-x^2+(y+1)^2]\dd y}{N^2},\\ \dd
    \beta & =-2\frac{[-x^2+(y+1)^2]\dd x-2x(y+1)\dd y}{N^2},
            \end{align}
\end{subequations}
and
\begin{subequations}\label{312}
  \begin{align}
  -\alpha \dd \beta +\beta \dd \alpha 
    = & \frac{2}{N^3}
       \{ [(x^2+y^2-1)(-x^2+(y+1)^2)-4x^2(y+1)]\dd x \nonumber\\
      & + [-2x(x^2+y^2-1)(y+1)-2x(-x^2+(y+1)^2]\dd y\}.\nonumber
  \end{align}
\end{subequations}
\eqref{E32} implies for the Berry phase $ \phi_B(x,y)$
\[
  \phi_B(x,y)=\frac{k}{2}\frac{2}{N^2}\frac{-N(x^2-y^2+1)\dd
    x-2Nxy\dd y}{y},
  \]
and  Berry phase on $\mc{X}_1$ in $(x,y)$ obtained from the Berry phase on
$\mc{D}_1$  in $(\alpha,\beta) $ is
\begin{equation}\label{BNFDX}
   \phi_B(x,y)=k\frac{(-x^2+y^2-1)\dd x -2xy\dd y}{y[x^2+(y+1)^2]}.
\end{equation}

Christofell's symbols in the variables  $(x,y)$ on $\mc{X}_1$ are
extracted from \cite[(69)]{SB21}
\begin{subequations}\label{GM22}
  \begin{align}
    \Gamma^x_{xx} & = 0,\quad  \Gamma^x_{xy}  = -\frac{1}{y},\quad
                    \Gamma^x_{yy}  = 0,\\
     \Gamma^y_{xx} & = \frac{1}{y},\quad  \Gamma^y_{xy}  = 0,\quad  \Gamma^y_{yy}  = -\frac{1}{y}.
 \end{align}
 \end{subequations}

b) {\it The Berry phase on the Siegel-Jacobi disk} $\mc{D}^J_1$ {\it {in}} $(w,z)$, $(\alpha,
\beta, q, p)$.

The
starting point is the scalar product of two CS on
$\mc{D}^J_1$ \cite[(7.13b)]{jac1}
\begin{equation}\label{SCWZ}
  f(z,w)=(e_{z,w},e_{z,w})=-2k\log P+\nu F, \quad F=\frac{2z\bar{z}+\bar{w}z^2+w\bar{z}^2}{2P}.
  \end{equation}

  With \eqref{BCON} we get 
  \begin{equation}\label{AA121}
    A_B(z,w)=\frac{\ii }{2}(A(z,w)-cc),
  \end{equation}
  where \cite[\S 4.2]{FC}
  \begin{subequations}
    \begin{align}
  A(z,w) & =(2k\frac{\bar{w}}{P}+\frac{\nu}{2}\bar{\eta}^2)\dd w
           +\nu\bar{\eta}\dd z\\
      ~& = k(\Gamma^w_{ww}\dd w+2\Gamma^w_{wz}\dd z).\label{331b}
    \end{align}
    \end{subequations}

                                       With \eqref{33ab63} we rewrite \eqref{331b} as
                                       \[
                                         A(z,w)=k(\theta^w_w+\Gamma^w_{wz}\dd z),
                                       \]
                                       which is in fact formula before
                                       (4.27) in \cite{FC}.
  
  With \eqref{E32} we get for $A$ in \eqref{AA121}
  \[
    A(w,\eta)=(2k\frac{\bar{w}}{P}-\frac{\nu}{2}\bar{\eta}^2)\dd
    w+\nu\bar{\eta}(\dd \eta- w\dd \bar{\eta}).
    \]
We also have the following expression for the Berry phase
\begin{equation}
  \begin{split}
      A_B(\alpha,\beta,q,p) & =2k\frac{\beta}{1-\alpha^2-\beta^2}\dd \alpha +
      [-2k\frac{\alpha}{1-\alpha^2-\beta^2}
      +\frac{\nu}{2}(q^2-p^2)]\dd\beta\\ &-\nu[(\alpha-\beta q-1)\dd q+(\alpha q+\beta)\dd p].
    \end{split}
    \end{equation}

    c) {\it{Berry phase in} } $(u,v)$

    We use for $f(u,v)$ \eqref{222a}, $(u,v)=(m+\ii n,x+\ii y)$. We
    have
    \[
 f_v=-\frac{k}{2}\frac{1}{v-\bar{v}}+\ii \nu       r^2=\ii (
              \frac{k}{4y}+\nu\frac{n^2}{y^2}),\quad 
              f_u =-2\ii \nu r=-2\ii \nu  \frac{n}{y}.
              \]

              We get
              \begin{subequations}
                \begin{align}
                A(u,v)& =f_u\dd u + f_v \dd v= -2\ii \nu r \dd u+\ii
                        (\frac{k}{4y}+\nu r^2)\dd v\\
                  & = k[- 2\Gamma^u_{uu}\dd
                    u+(\Gamma^v_{vv}-\frac{3i}{4}\frac{1}{y})\dd v], 
                  \end{align}\end{subequations}
where we have used \eqref{XGAMMM}.
              
      \begin{equation}\label{ABXY}
        A_B(x,y,m,n)= \frac{\ii}{2}
        (A(u,v)-\bar{A}(u,v))= \frac{1}{y}\left[-(\frac{k}{4}+\nu\frac{n^2}{y})\dd
          x+2\nu n\dd m
        \right].        \end{equation}

                               d) {\it Berry phase on} $\mc{X}^J_1$ {\it in} $(v,\eta)=(x+\ii y,q+\ii p)$

      With \eqref{KK1} for $f(v,\eta)$, we have
      \[
          f_v = -2k\frac{1}{v-\bar{v}}= \ii \frac{k}{y};\quad
          f_{\eta}= -\nu (\eta-\bar{\eta})=-2\ii \nu p,
        \]
      and we get
      \[
        A_B(x,y,q,p)=-\frac{k}{y}\dd x+2\nu p\dd q.        \]

    We write
  \begin{equation}\label{ABXYY}
    A_B(\alpha,\beta,x,y)=\dd \phi_B(\alpha,\beta,x,y)=\dd \phi_{B I}+\dd \phi_{B II}+\dd \phi_{B III},
  \end{equation}
  where $\dd \phi_{B I}$, which  appears in \eqref{ThEETA}, was
  calculated as \eqref{BNFDX}, and
  \begin{subequations}\label{AA1}
    \begin{align}\dd \phi_{B II}& = \frac{\ii \nu}{2}(-\bar{\eta}^2\dd w +cc),\\
\dd \phi_{B III}&= \frac{\ii \nu}{2}[(\bar{\eta}+w\eta)\dd \eta -cc].
    \end{align}
  \end{subequations}
  We find
  \begin{subequations}\label{AA2}
    \begin{align}
       \dd \phi_{B II}& =\frac{\ii \nu}{2}[-(q-\ii p)^2(\dd \alpha
                        +\ii \dd \beta)+cc]=\nu [(q^2-p^2)\dd \beta
                        -2qp\dd \alpha],\\
      \dd \phi_{B III}&= \nu\{-[(\alpha+1)q+\beta p]\dd p+[(1-\alpha
                        )p+\beta q ]\dd q
                       \}.
      \end{align}
    \end{subequations}
    With \eqref{ALFAB}, \eqref{DALPHA} we get
\begin{equation}\label{EQ11}
      \begin{split}
      A_B&(x,y,p,q)= \frac{k}{Ny}[(-x^2+y^2-1) \dd x-2xy\dd y]+\\
      & \frac{2\nu}{N^2}[-4x(y+1)pq+(x^2-(y+1)^2) (q^2-p^2)]\dd x +\\
        & \frac{4\nu}{N^2} [(x^2-(y+1)^2) qp
        +x(y+1)(q^2-p^2)]\dd y+\\
        & \frac{2\nu}{N}\{-[(x^2+y(y+1))q-xp]\dd p +[(y+1)p-xq] \dd
        q]\}= \\
        & \frac{1}{N}\{ \frac{k}{y}(-x^2+y^2-1)+\frac{2\nu}{N}[-4x(y+1)pq+(x^2-(y+1)^2)(q^2-p^2)]\}\dd x+\\
        & \frac{2}{N}\{  -kx+\frac{2\nu}{N}[x^2-(y+1)^2]pq+x(y+1)(q^2-p^2)\}\dd y+ \\
          & \frac{2\nu}{N}\{ [-[x^2+y(y+1)]q+xp]\dd p +[(y+1)p-xq]\dd q\}.
        \end{split}
      \end{equation}

    \section{Almost cosymplectic manifolds}\label{ACS}
    
\subsection{Definitions}
Following  \cite{paul}, an      {\it  almost cosymplectic manifold}  (ACOS) is the
triplet $(M,\theta,\Omega) $,
where $M$ is a $(2\mr{n}+1)$-dimensional manifold, $\theta\in \got{D}^1$
\begin{equation}\label{TH}
  \theta=\sum_{I=1}^{\mr{n}}(a_I\dd Q^I+b_I\dd P^I)+c\dd \kappa ,\
  a_I,~b_I, ~c\in \R, ~c\not=0, \end{equation}
$\Omega$ is a   2-form with \begin{equation}\label{COND}\text{rank}(\Omega)=2\mr{n},\end{equation}   and
\begin{equation}\label{thtOM}
  \theta\wedge\Omega^{\mr{n}}\neq 0.
\end{equation}

We recall, see e.g.
\cite[\S 7]{lutke} that if 
$A=a_{ij}\in M(n,\C)$ then  the vectorisation operator is defined as
\[
  (\mr{vec}(A)^t)^t=[a_{11},a_{12},\dots,a_{1n},a_{21},\dots,a_{2n},a_{31},\dots,a_{nn}]\in M(1,n^2,\C),
\]
while the half-vectorisation operator is 
  \[
    (\mr{vech}(A)^t)^t=[a_{11},a_{12},\dots, a_{1n},
    a_{22},\dots,a_{2n},a_{33},\dots, a_{nn}]\in M(1,N_1), ~N_1= \frac{n(n+1)}{2}.
  \]
  Note also that $\mr{vech}(A)=L_n\mr{vec}(A)$, where $L_n\in M(N_1,n^2)$ is the
  elimination matrix.

  We endow  the $n(n\!+\!3)\!+\!1$-dimensional manifold $\tilde{\mc{X}}^J_n$ with an ACOS 
structure $(\tilde{\mc{X}}^J_n,\theta,\omega) $. For $\theta$ we take
the formula \eqref{TH}, and consider  $\Omega$ written in Darboux
coordinates
\begin{equation}\label{OMM1}
  \Omega =\sum_{I=1}^{\mr{\frac{n(n+3)}{2}}}\dd Q^I\wedge\dd P^I. \end{equation}

\begin{lemma}\label{NNOOU} If in formula \eqref{NOUW1} of the \Ka-two
  form on  $\mc{D}^J_n$ we make the partial the 
  Cayley transform $(v,u)\rightarrow (v,\eta)=(x+\ii y,q+\ii p)$
  \eqref{PCT},  where $v\in MS(n,\C)$, we  get the \Ka~ two-form $\omega_{\mc{X}^J_n}$
  \begin{subequations}
    \begin{align}
  \omega_{\mc{X}^J_n}(x,y,p,q) &=\omega_1+\omega_2, \label{42a}\\
 \omega_1& =\frac{k}{4}\tr(y^{-1}\dd x\wedge y^{-1}\dd
 y)=-\frac{k}{4}\tr (\dd x \wedge\dd y^{-1}) .\label{42b}
  \end{align}
  \end{subequations}
We arrange the elements of the symmetric matrices $x, y^{-1}$ in 
$((\rm{vech}(x))^t)^t, ((\rm{vech}(y^{-1}))^t)^t$ s.t. there are $i$ elements on every 
row, $i=1$ on the last row, and $i=n$ on the first row.      We use  the notation
  \begin{equation}\label{id}
    i_d:= N_1-\frac{i(i+1)}{2},\quad  i=1,\dots, n. 
  \end{equation}
  We vectorise $\omega_1$ \eqref{42b} as \begin{equation}\label{1ome} \omega_1 =\sum_{I=1}^{N_1}\dd Q^I\wedge
  \dd P^I, \end{equation}  where 
\begin{equation}\label{OIPI}
  (Q^I, P^I)  = \left\{\begin{array}{ccc} (\frac{k}{4}
                                 x_{i_d+1,i_d+1} ,
                                 &  - {y}^{-1}_{i_d+1,i_d+1})&, I=i_d+1;\\
( \frac{k}{2} x_{i_d+1,i_d+j}, &    - {y}^{-1}_{i_d+1,i_d+j}) &,
 j=2, \dots, i,  I=i_d+2,\dots ,i_d+i .\end{array}  \right. 
\end{equation}
For $\omega_2$ in \eqref{42a} we have
\begin{equation}\label{22OM}
 \omega_2  = 2\nu\dd  Q\wedge \dd P^t =\sum_{I=N_1+1}^{N_1+n}\dd Q^i\wedge \dd P^i,~
      Q^{N_1+i}=2\nu q_i, P^{N_1+i}= p_i,~ i=1,\dots, n.\end{equation}
\end{lemma}

 \begin{proof}
  With
  \[
    H=-\frac{1}{2\ii}y^{-1}\dd (x+\ii y).
    \]
  we get \[
  \tr [H\wedge\bar{H}]=\frac{1}{4}y^{-1}_{ik}(\dd x_{kj}+\ii \dd
  y_{kj})\wedge y^{-1}_{jl}(\dd x_{li}-\ii \dd y_{li})
  =\frac{\ii}{2}\tr (y^{-1}\dd y\wedge y^{-1}\dd x).\]
Now we calculate $\omega_2$. Using the symmetry of the matrices  $x, y$,
we  get successively \begin{subequations}
    \begin{align*}
\tr(G^tD\wedge\bar{G}) &=\frac{\ii}{2} \tr\{[ \dd p(x+\ii y)+\dd q] y^{-1}\wedge
                    [(x-\ii  y)\dd p^t+\dd q^t] \}\\
   &= \frac{\ii}{2} \tr\{\dd p(x\!+\!\ii y)y^{-1}\wedge
                    [\!(x\!-\!\ii y)\dd p^t\!+\!\dd q^t\!] \!+\!\dd qy^{-1}\wedge[(\!x-\!\ii y\!)\dd
                    p^t\!+\!\dd q^t\!]\}\\
                  & = \frac{\ii}{2} \tr\{   \dd p[
                    (xy^{-1}x+yy^{-1}y)\wedge\dd p^t + (xy^{-1}+\ii
                    )\wedge \dd q^t] \\
                  &+\dd q[ (y^{-1}x -\ii )\wedge \dd p^t +
                    \wedge y^{-1}\dd q^t]   \}\\
                  & =\frac{\ii}{2} \tr\{\dd p(xy^{-1}x+yy^{-1}y)\wedge\dd
                     p^t+\dd p xy^{-1}\wedge \dd q^t \\ 
      & +\dd qy^{-1}\wedge\dd q^t+ \ii (\dd p\wedge \dd q^t-\dd q\wedge \dd p^t)
                    \}\\
      &=  \frac{\ii}{2} \tr\{\dd p(xy^{-1}x\!+\!yy^{-1}y)\wedge\dd p^t 
         +\dd q y^{-1}\wedge \dd q^t
                  \\ &+\dd p xy^{-1}\wedge \dd
         q^t +\dd qy^{-1}x\wedge \dd p^t\!+\!2\ii \dd p\wedge \dd q^t\}\\
         &= -\dd p\wedge \dd q^t,
                    \end{align*}
                  \end{subequations}
                  and we find
                  \begin{equation}\label{prodqp}
                    \omega_2=2\nu \dd q^t\wedge \dd p.
                    \end{equation}
                  We identify $\Omega$ in formula \eqref{OMM1} with
                  $\omega$ from \eqref{42a}, where $\omega_1$
                 $ (\omega_2)$
                  is given by \eqref{1ome}, \eqref{OIPI}
                  (respective\eqref{22OM}).
 We find:
 \begin{subequations}\label{AIBI}
 \begin{align}
  a_I &=b_I=0, I=1,\dots, N_1;\\
   a_{N_1+i}& =-\frac{\sqrt{\delta}}{2\nu}p_i;~~
  b_{N_1+i}=\sqrt{\delta}q_i, i=1,\dots, n;\\
  c& =\sqrt{\delta}, ~~~\mr{n}=\frac{n(n+3)}{2}.
  \end{align}
 \end{subequations}

 Now we identify  $\lambda_6$ in formula
\eqref{BIGM}  with $\theta$ given by \eqref{TH}. We write $\lambda_6$
as
\[
  \lambda_6=\sqrt{\delta}\{[\sum_{I=1}^{N_1}-P_I\dd Q_I+ Q_I\dd P_I]+
  \sum_{i=1}^n[-P_{N_1+i}\dd Q_{N_1+i}+Q_{N_1+i}\dd P_{N_1+i}] +\dd \kappa\},
 \]
 and $\theta$ as
  \[
 \theta =\sum_{I=1}^{N_1}[a_I\dd Q_I+ b_I\dd P_I]+
 \sum_{I=N_1+1}^{\mr{n}}[ a_I\dd Q_I+ b_I\dd P_I]+c\dd \kappa.
\]
Finally, $\omega_{\mc{X}^J_n}(x,y,q,p)$ corresponds to $\Omega$
$(Q^I,P^I),I=1,\dots,N_1$ in  $\eqref{OIPI}$
respectively $i=1,\dots,n$ in \eqref{22OM}.

We have $$ \dd \omega =0, \quad \theta\wedge \omega^{\rm{n}}
=c\frac{\mr{n}}{2}\Pi_{I=1}^{\frac{\mr{n}}{2}}\dd Q^I\wedge \dd
P^I\wedge \dd \kappa, $$ i.e. 
condition \eqref{thtOM} is satisfied because of \eqref{COND}.
\end{proof}

We endow   $\tilde{\mc{X}}^J_1$  with a  {\it  generalized transitive
almost cosymplectic}   (GTACOS)  structure \cite{SB22}, i.e.  an
ACOS  structure  $(M,\theta,\Omega)$ such that
 \[\dd \Omega =0.
 \]
 Lemma \ref{L2},  extracted from \cite[Lemma  1]{SB22}, is a
 particular case of Lemma \ref{NNOOU}.  To
 proof \eqref{214b}  we introduce  the relations \eqref{ULTRAN1}  into
 \eqref{omSYUHP}.

 Alternatively, we introduce in formula \eqref{kk1} of $\mc{A}(w,z)$
 $z=\eta- w\bar{\eta}$ \eqref{E32} and we find
 \begin{equation}\label{DBA}
   \mc{A}\wedge\bar{\mc{A}}=P\dd\eta\wedge\dd\bar{\eta}.
 \end{equation}
 Introducing \eqref{PPP}, \eqref{dwv}
 and  \eqref{DBA}, we get again \eqref{214b}.
 
   \begin{lemma}\label{L2}
  If we introduce into the \Ka~two-form  $\omega_{\mc{D}^J_1} (w,z)$
  \eqref{kk1} the second partial Cayley transform \eqref{ULTRAN1}
$(w,z)\rightarrow (x,y,q,p), y>0$ we get the symplectic two-form \eqref{214b}.

In the notation of \cite{SB19}, we  introduce on the extended
5-dimensional Siegel-Jacobi half-plane  $\tilde{\got{X}}^J_1$
parametrized in $(x,y,p,q,\kappa)$   the almost
cosymplectic structure  
$(\tilde{\mc{X}}^J_1,\theta,\omega)$, where $\theta = \lambda_6$ and
$\omega$ is \eqref{214b}, i.e. 
\begin{subequations}\label{TT1}
  \begin{align}\theta & =\sqrt{\delta}(\dd \kappa
  -p \dd q+q\dd p),\quad \delta>0, \label{TT11}\\
    \omega & =\frac{k}{y^2}\dd x\wedge \dd y +2\nu \dd q\wedge \dd
             p,\quad y>0 \label{TT12}.
             \end{align}
           \end{subequations}
We have
           \begin{equation}\label{cond}
             \dd \omega=0,\quad \theta \wedge
             \omega^2=4\frac{k\nu\sqrt{\delta}}{y^2}\dd x\wedge\dd
             y\wedge\dd q\wedge\dd p\wedge \dd \kappa, 
             \end{equation}
             and   $(\tilde{\got{X}}^J_1,\theta,\omega)$ verifies the condition
             \cite[(5.5)]{SB22} of an almost cosymplectic manifold.

         With a formula of the type \eqref{KALP2},     the \Ka~two-form \eqref{TT12} on $\mc{X}^J_1$ corresponds
         to the \Ka~ potential
         \begin{equation}\label{KK33}
           f(v,\bar{v},\eta,\bar{\eta})=-2k\log \frac{v-\bar{v}}{2\ii}+\nu[\eta\bar{\eta}+g(\eta)+h(\bar{\eta})],
           \end{equation}
           in particular we get 
             \begin{equation}\label{KK1}
               f(v,\bar{v},\eta,\bar{\eta})=-2k \log \frac{v-\bar{v}}{2\ii}-\frac{\nu}{2}(\eta-\bar{\eta})^2,
             \end{equation}
             which correspond  to particular values of functions in \eqref{V1}.

If in \eqref{KK1} we make the change of coordinates \eqref{ULTRAN2} in
\eqref{KK1}, 
we get
\begin{equation}\label{KK2}
               f(w,\bar{w},\eta,\bar{\eta})=-2k \log \frac{1-w\bar{w}}{(1-w)(1-\bar{w})}-\frac{\nu}{2}(\eta-\bar{\eta})^2,
             \end{equation}
             which corespond to the particular values
             $f(w)=2k\log\frac{1-w}{w}$, $ g(\bar{w})= 2k\log(1-\bar{w})$,  $f'(\eta)={g'(\bar{\eta})}=-\frac{\nu}{2}\eta^2$
             in \eqref{V2}.
             
     If we apply to the reproducing Kernel \eqref{KK2} a formula of
     the type   \eqref{KALP}, we get again \eqref{E32b}.
     
             In Darboux coordinates we have a particular almost
             cosymplectic manifold
             $(\tilde{\got{X}}^J_1,\theta,\omega)$ verifying
             \cite[(5.5)]{SB22}
             and in  addition the  condition 
            
             \[
               \dd \omega=0. 
             \]
          
  $(\tilde{\got{X}}^J_1,\theta,\omega)$         was called   {\bf  generalized transitive
 almost  cosymplectic  manifold} \cite[GTACOS]{SB22}.
\end{lemma}

\subsection{Connection matrix on  $\tilde{\mc{X}}^J_1$}\label{CMSJ}
We determined the Christofell's symbols corresponding to
 the Riemannian metric \eqref{linvG} of the extended Siegel-Jacobi
 upper half-plane $\tilde{\mc{X}}^J_1$ \cite[page 22]{SB21}. In formulas below we have  included only  the
 $\Gamma$-s which   are not  given in \eqref{GSC}
\begin{equation}\label{MNM}
\begin{array}{llll}
 \!\Gamma^{p}_{pp} \!=\! 2\tau\frac{xq}{y}\!&\!
 \Gamma^{p}_{pq} \!=\!\tau\frac{q\!-\!px}{y}
  \!&\!\Gamma^{p}_{p\kappa} \!=\!\tau\frac{x}{y} \!&\!\Gamma^{p}_{qq}
                                               \!\!=\!\!-\!2\tau\frac{p}{y} \!~\!
  \Gamma^{p}_{q\kappa}\!=\!\tau\frac{1}{y}\! \\
 \!\Gamma^{q}_{pp} \!=\!-2\tau\frac{qS}{y}
  \!&\!\Gamma^{q}_{pq}\!\!=\!\!\tau\frac{-xq\!\!+\!\!pS}{y}
  \!&\!\Gamma^{q}_{p\kappa}\!=\!-\!\tau\frac{S}{y}
 \!&\!\Gamma^{q}_{qq}\!=\! 2\tau
 \!\frac{xp}{y} \!~\!\Gamma^{q}_{q\kappa}\!=\!-\!\tau\frac{x}{y}\! \\
\!\Gamma^{\kappa}_{xp}\!=\!\frac{py^2\!-\!x\xi}{2y^2}& \!\Gamma^{\kappa}_{xq}\!=\!-\!\frac{\xi}{2y^2}\!&  \Gamma^{\kappa}_{yp}\!\!=\!\!-\!\frac{2px+q}{2y}
  \!&\!\Gamma^\kappa_{yq}\!=\!-\frac{p}{2y}~
\Gamma^{\kappa}_{pp} \!\!=\!\!-\!2\!\tau\!\frac{q}{y}(\!p\!S\!+\!q\!x\!)\! \\
  \!\Gamma^{\kappa}_{pq} \!=\! \tau\frac{p^2S-q^2}{y}& 
\!\Gamma^{\kappa}_{p\kappa} \!\!=\!-\!\tau\frac{pS+qx}{y}\!
&
 \Gamma^{\kappa}_{qq} \!\!=\!\!2\tau\frac{p\xi}{
                                                              y}\!\!&\!\Gamma^{\kappa}_{q\kappa}
                                                                      \!\!=-\!\tau\frac{\xi x}{y}.
                                              \end{array}
\end{equation}
where \[
  \tau :=\frac{\delta}{\gamma}, \quad \xi :=px+q .\]
We determine the  connection matrix on the extended Siegel-Jacobi
upper half-plane in the S-coordinates $(x,y,q,p,\kappa)$ 
\begin{subequations}\label{thetap}
  \begin{align}\nonumber
    \!\theta'_{\tilde{\mc{X}}^J_1}(x,y,q,p,\kappa)\!=&
    \left(\begin{array}{ccccc}\theta'^ x_x &\theta'^ x_y &\theta'^ x_q
            &\theta'^ x_p & \theta'^ x_{\kappa}\\
   \theta'^y_x & \theta'^y_y &\theta'^y_q& \theta'^y_p& \theta'^y_{\kappa}\\
  \theta'^q_ x&         \theta'^q_y&  \theta'^q_q  &\theta'^q_p &\theta'^q_{\kappa}  \\
  \theta'^p_x &  \theta'^p_ y&  \theta'^p_ q&  \theta'^p_p &
                                                             \theta'^p_{\kappa}
            \\
              \theta'^{\kappa}_x &  \theta'^{\kappa}_ y&  \theta'^{\kappa}_ q&  \theta'^{\kappa}_p &
                                                                                                     \theta'^{\kappa}_{\kappa}\end{array}\right)
            ~~ &= \left(\begin{array}{ccccc}\theta^ x_x &\theta^ x_y &\theta^ x_q
            &\theta^ x_p & 0\\
  \theta^y_x & \theta^y_y &\theta^y_q& \theta^y_p& 0\\
  \theta^q_ x&         \theta^q_y&  \theta'^q_q  &\theta'^q_p &\theta'^q_{\kappa}  \\
  \theta^p_x &  \theta^p_ y&  \theta'^p_ q&  \theta'^p_p &
                                                             \theta'^p_{\kappa}\\
              \theta'^{\kappa}_x &  \theta'^{\kappa}_ y&  \theta'^{\kappa}_ q&  \theta'^{\kappa}_p &
                                                             \theta'^{\kappa}_{\kappa}
 \end{array}\right). 
  \end{align}\end{subequations}
With \eqref{GSC} and \eqref{MNM}, we find for the matrix elements of
\eqref{thetap} the values
\begin{subequations}\nonumber
  \begin{align}\nonumber
\begin{array}{ll}
  \theta'^q_q &= \Gamma^q_{qx}\dd x+\Gamma^q_{qy}\dd y +\Gamma^q_{qq}\dd q+\Gamma^q_{qp}\dd
                 p+\Gamma^q_{q\kappa}\dd \kappa \\&=-\frac{x}{2y^2}\dd
                 x -\frac{1}{2y}\dd y+\frac{\tau}{y}\left( 2xp\dd q+(-xq+pS)\dd
                 p-x\dd \kappa \right)\\
  \theta'^q_p &=  \theta^q_p +\Gamma^q_{qp}\dd p+ \Gamma^q_{q\kappa}\dd
               \kappa\\ &= \frac{y^2-x^2}{2y^2}\dd x-\frac{x}{y}\dd y
                 +\frac{\tau}{y}\left( (-xq+pS)\dd q-2qS\dd p -x\dd \kappa )\right)\\
  \theta'^q_{\kappa} &= \Gamma'^q_{\kappa q}\dd q +\Gamma^q_{\kappa
                       p}\dd p=-\tau\frac{x}{y}\dd q -\tau\frac{x}{y}\dd
                       p\\
  \theta'^p_q &= \Gamma^p_{qx}\dd x+\Gamma^p_{qq}\dd q+\Gamma^p_{qp}\dd p
                +\Gamma^p_{q\kappa}\dd\kappa\\ & =\frac{x}{2y^2}\dd x
                +\frac{\dd y}{2y}+\frac{\tau}{y}\left(-2p\dd
                q +(q-px)\dd p +\dd \kappa\right)\\
  \theta'^p_p &= \Gamma^p_{px}\dd x\!+\!\Gamma^p_{py}\dd y\!+\!\Gamma^p_{pq}\dd q \!+\!\Gamma^p_{pp}\dd
                p\!+\!\Gamma^p_{p\kappa}\dd \kappa\\
               &=\frac{x}{2y^2}\dd x\!+\!\frac{x}{2y}\dd y\!+\! \frac{\tau}{y}\left(
                (q-px)\dd q+2xq\dd p+x\dd \kappa\right)\\
  \theta'^p_{\kappa } &= \Gamma^{p}_{\kappa q}\dd q +\Gamma^p_{\kappa
                        p}\dd p= \frac{\tau}{y}(\dd q+ x\dd p)\\
   \theta'^{\kappa}_{x} &= \Gamma^{\kappa}_{x q}\dd q
                          +\Gamma^{\kappa}_{x\kappa}\dd p= -
                          \frac{\xi}{2y^2}\dd
                          q+\frac{py^2-x\xi }{2y^2}\dd p\\
     \theta'^{\kappa}_{y} &= \Gamma^{\kappa}_{y q}\dd q +
                            \Gamma^{\kappa}_{yp}\dd p=-\frac{p}{2y}\dd
                            q -\frac{2px+q}{2y}\dd p\\
     \theta'^{\kappa}_{q} &= \Gamma^{\kappa}_{qx}\dd x
                            +\Gamma^{\kappa}_{qy}\dd y
                            +\Gamma^{\kappa}_{qq}\dd
                            q+\Gamma^{\kappa}_{qp}\dd
                            p+\Gamma^{\kappa}_{q\kappa}\dd
                            \kappa\\ &=-\frac{\xi}{2y^2}\dd
                            x-\frac{p}{2y}\dd y+\frac{\tau}{y}(
                            2p\xi\dd q+ (p^2S-q^2)\dd p-\xi x\dd
                                       \kappa )\\
  \theta'^{\kappa}_{p} &= \Gamma^{\kappa}_{px}\dd x
                         +\Gamma^{\kappa}_{py}\dd y
                         +\Gamma^{\kappa}_{pq}\dd q
                         +\Gamma^{\kappa}_{pp}\dd p +
                         \Gamma^{\kappa}_{p\kappa}\dd \kappa\\
  &= \frac{py^2-x\xi }{2y^2}\dd x-\frac{2px+q}{2y}\dd
    y+\frac{\tau}{y}((p^2S-q^2)\dd q-2q(pS+qx)\dd p-(pS+qx)\dd
    \kappa)\\
  \theta'^{\kappa}_{\kappa} &= \Gamma^{\kappa}_{\kappa q}\dd
                              q+\Gamma^{\kappa}_{\kappa p}\dd p=
                              -\frac{\tau}{y}(\xi x\dd q+ (pS+qx)\dd p)                         \end{array}
  \end{align}
  \end{subequations}

  \subsection{Covariant derivative of one-forms on $\mc{X}^J_1$  and
  $\tilde{\mc{X}}^J_1$}\label{CD}
The covariant derivative of a contravariant vector (one-form) is given
by
\begin{equation}\label{CV2}
  D u_i=-\theta^i_j u_j=-u_j\Gamma^i_{jk}u_{k}=-u_j\Gamma^i_{kj}u_k.
\end{equation}
The covariant derivative of $\dd z$ on $\mc{D}^J_1$ has the expression \cite[(41)]{GAB}
\begin{equation}\label{DDZ}
D (\dd z) =\left(\begin{array}{cc}\dd z \dd w\end{array}\right) 
\left(\begin{array}{cc} \lambda\bar{\eta} &
                                            \lambda\bar{\eta}^2-\frac{\bar{w}}{P}\\
        \lambda\bar{\eta}^2-\frac{\bar{w}}{P} & \lambda \bar{\eta}^3\end{array}\right)
\left(\begin{array}{c}\dd z \\\dd w\end{array}\right).
\end{equation}

The covariant derivative of $\dd w$ has the expression \cite[(42)]{GAB}
\begin{equation}\label{DDW}
- D (\dd w)  =\left(\begin{array}{cc} \dd z \dd w\end{array}\right)
\left(\begin{array}{cc} \lambda & \lambda \bar{\eta} \\
        \lambda\bar{\eta} & \lambda\bar{\eta}^2+2\frac{\bar{w}}{P}\end{array}\right)
\left(\begin{array}{c}\dd z \\\dd w\end{array}\right).
\end{equation}

We calculate the covariant derivative of the S-variables $x,y,p,q$ on $\mc{X}^J_1$
\begin{equation}\label{DX}
    D  x=\left(\begin{array}{c}\dd x\\ \dd y \\\dd q \\\dd p \end{array}\right)^t
    \left(\begin{array}{cccc}  0&\frac{1}{y} & 0 &0\\
          \frac{1}{y} &  0& 0& 0\\
             0& 0& 0& \frac{\ep}{2}y\\
                 0& 0& \frac{\ep}{2}y& \ep xy
                \end{array}\right) \left(\begin{array}{c} \dd x \\ \dd y\\ \dd q\\ \dd p\end{array}\right).
 \end{equation}
\begin{equation}\label{DY}
    D  y=\left(\begin{array}{c}\dd x \\\dd y \\\dd q \\\dd p \\\end{array}\right)^t
    \left(\begin{array}{cccc}  -\frac{1}{y}& 0 & 0 &0\\
            0& \frac{1}{y}& 0& 0\\
             0& 0& -\frac{\ep}{2}& -\frac{\ep}{2}x\\
                 0& 0& -\frac{\ep}{2}x& \frac{\ep}{2}(y^2-x^2) 
                \end{array}\right) \left(\begin{array}{c} \dd x \\ \dd y\\ \dd q\\ \dd p\end{array}\right).
  \end{equation}
\begin{equation}\label{DQ}
   D  q=\left(\begin{array}{c}\dd x \\\dd y \\\dd q \\\dd p \end{array}\right)^t
    \left(\begin{array}{cccc}  0& 0 & \frac{x}{2y^2} &\frac{x^2-y^2}{2y^2}\\
            0& \frac{x}{y} & 0& 0\\
             \frac{x}{2y^2}& 0 & 0& 0\\
                 \frac{x^2-y^2}{2y^2}& 0& 0& 0
                \end{array}\right) \left(\begin{array}{c} \dd x \\ \dd y\\ \dd q\\ \dd p\end{array}\right).
  \end{equation}
\begin{equation}\label{DP}
  D  p=\left(\begin{array}{c}\dd x \\\dd y \\\dd q \\\dd p \\\end{array}\right)^t
    \left(\begin{array}{cccc}  0& 0 & -\frac{1}{2y^2} &-\frac{x}{2y^2}\\
            0& -\frac{1}{2y} & 0& 0\\
             -\frac{1}{2y^2}& 0 & 0& 0\\
                 -\frac{x}{2y^2}& 0& 0& 0
                \end{array}\right) \left(\begin{array}{c} \dd x \\ \dd y\\ \dd q\\ \dd p\end{array}\right).
\end{equation}
We calculate the covariant derivative of the S-variables
$x,y,p,q,\kappa$ on $\tilde{\mc{X}}^J_1$
\begin{equation}\label{DDX}
    D  x=\left(\begin{array}{c}\dd x \\\dd y \\\dd q \\\dd p\\ \dd \kappa \end{array}\right)^t
    \left(\begin{array}{ccccc}  0&\frac{1}{y} & 0 &0 &0 \\
           \frac{1}{y} & 0& 0& 0& 0\\
             0& 0& 0& \frac{\ep}{2}y & 0\\
            0& 0& \frac{\ep}{2}{y}& \ep xy & 0\\
                                              0& 0& 0& 0& 0
                \end{array}\right) \left(\begin{array}{c} \dd x \\ \dd
                                           y\\ \dd q\\ \dd p\\ \dd
                                           \kappa\end{array}\right).
                                     \end{equation}
                                   \begin{equation}\label{DDY}
 D  y=\left(\begin{array}{c}\dd x\\ \dd y \\\dd q \\\dd p\\ \dd \kappa\end{array}\right)^t
    \left(\begin{array}{ccccc}  -\frac{1}{y}& 0 & 0 &0 & 0\\
            0& \frac{1}{y}& 0& 0 & 0\\
             0& 0& -\frac{\ep}{2}& -\frac{\ep}{2}x & 0\\
            0& 0& -\frac{\ep}{2}x& \frac{\ep}{2}(y^2-x^2) & 0\\
0& 0& 0& 0& 0
                \end{array}\right) \left(\begin{array}{c} \dd x \\ \dd
                                           y\\ \dd q\\ \dd p \\ \dd \kappa \end{array}\right).
  \end{equation}
\begin{equation}\label{DDQ}
  D  q=\left(\begin{array}{c}\dd x\\ \dd y \\\dd q \\\dd p\\ \dd \kappa \end{array}\right)^t
    \left(\begin{array}{ccccc}  0& 0 & \frac{x}{2y^2}
            &\frac{x^2-y^2}{2y^2}& 0 \\
            0& 0 & \frac{1}{2y}& \frac{x}{y} & 0\\
             \frac{x}{2y^2}& \frac{1}{2y}& -2\frac{\tau}{y}xp& \frac{\tau}{y}(xq-pS)& \frac{\tau}{y}x\\
            \frac{x^2-y^2}{2y^2}& \frac{x}{y}&\frac{\tau}{y}(xq-pS) & \frac{\tau}{y}2qS& \frac{\tau}{y}S\\
                                            0& 0& \frac{\tau}{y}x& \frac{\tau}{y}S & 0                \end{array}\right) \left(\begin{array}{c} \dd x \\ \dd
                                           y\\ \dd q\\ \dd p \\ \dd
                                                                                                     \kappa \end{array}\right).
                                                                                             \end{equation}
  \begin{equation}\label{DDP}
    D  p\!=\!\left(\!\begin{array}{c}\dd x\! \\\dd y\! \\\dd q\! \\\dd p\! \dd \kappa \!\end{array}\right)^t
    \left(\!\begin{array}{ccccc}
            0& 0 & -\frac{1}{2y^2}&-\frac{x}{2y^2} &0 \\
            0& 0 & 0& -\frac{1}{2y}& 0\\
             -\frac{1}{2y^2}& 0 & 2\frac{\tau}{y}p& \frac{\tau}{y}(q-px)& -\frac{\tau}{y}  \\
            -\frac{x}{2y^2}& -\frac{1}{2y}& \frac{\tau}{y}(q-px) &
            -\tau\frac{xq}{y}& -\frac{\tau}{y}x\\
            0& 0& -\frac{\tau}{y}& -\frac{\tau}{y}x& 0
                \end{array}\right)\! \left(\!\begin{array}{c} \dd x\! \\ \dd
                                           y\!\\ \dd q \!\\ \dd p \!\\ 
                                            \dd \kappa\!
                                         \end{array}\right)\!\end{equation}

\begin{equation}\label{DDK}
 D\kappa\!=\!\left(\!\begin{array}{c}\dd x\\\dd y\\\dd q\\\dd
                            p\\\dd\kappa \!
                          \end{array}\right)^t\!
    \left(\!\begin{array}{ccccc}
              \!0\!& \!0\! &\! \frac{\xi }{2y^2}\!&\!\frac{x\xi-py^2}{2y^2}\! &\!\!0\!\! \\
            \!0\!&\!\! 0\!\! & \!\frac{p}{2y}\!& \!\frac{2px+q}{2y}\!& \!\!0\!\!\\
             \!\frac{\xi}{2y^2}\!& \!\!\frac{p}{2y} \!\!& \!-2\frac{\tau}{y}p\xi\!&\!
                                                                    -\frac{\tau}{y}(p^2S\!-\!q^2)\!& \!\!\frac{\tau}{y}\xi x  \!\!\\
            \!\frac{x \xi - py^2}{2y^2}\!& \!\!\frac{2px+q}{2y}\!\!& \!\frac{\tau}{y}(q^2-p^2S)\! &
            2\frac{\tau q}{y}(pS+qx)&\!\! \frac{\tau}{y}(pS\!+\!qx)\!\!\\
           \!0\!& \!\!0\!\!&\!\frac{\tau}{y}\xi  x \!&\frac{\tau}{y}(pS\!+\!qx) & \!\!0\!\!
                                                   \end{array}\right)\! \left(\!\begin{array}{c} \!\dd x\! \\ \dd
                                           y\!\\ \!\dd q \!\\ \!\dd p \!\\ 
                                            \!\dd \kappa\!
                                         \end{array}\right)\!\end{equation}

                                   \section{Appendix }\label{APP}
\subsection{Theory}
                                   \subsubsection{Connections on real manifolds}\label{REALC}
a) {\it Connections on vector bundles.}
Following \cite[page 101]{ccl}, let $E$ be a $q$-dimensional real vector bundle with
projection $\psi : E\rightarrow M$ on the
$m$-dimensional manifold $M$ and let 
$\Gamma(E)$ be the set of smooth  sections of $E$ on $M$.

\begin{deff}\label{def1}
  {\bf  A connection}  on the vector bundle $E$ is a map
  \[
 D:\Gamma(E)\rightarrow \Gamma(T^*(M)\otimes E)
  \]
  such that 
\begin{subequations}\label{DOI}
  \begin{align}
  1. &~~  D(s_1+s_2)=Ds_1+Ds_2, \quad \forall s_1,s_2\in \Gamma (E),\\
  2. &~~  D(\alpha s)=\dd \alpha\otimes s+\alpha D s,\quad \alpha\in C^{\infty}(M).\label{51b}
  \end{align}
  \end{subequations}  The {\bf absolute differential quotient} or  the {\bf covariant derivative} of the section $s$ along $X\in\got{D}^1$ is defined
 as
\begin{equation}\label{DXS}
  D_Xs:=<X,Ds>,\end{equation}where $<,>$ is the pairing between the
tangent space $T(M)$ and  the cotangent space $T^*(M)$. \end{deff}
Chose a {\bf local field frame} of $E$ on the neighborhood $U\subset
M$, i.e. $q$-linearly independent smooth sections $s_{\alpha}$.  Then
at every point $p\in U$ 
\{$\dd u^i\otimes s_{\alpha},1\leq i  \leq m; 1\leq \alpha\leq q$\} form
a basis of $T^*_p\otimes E$ and
\begin{equation}\label{DSAL}
  Ds_{\alpha}=\sum_{\beta =1}^q\omega^{\beta}_{\alpha}\otimes s_{\beta},\quad
    \omega_{\alpha}^{\beta}=\sum_{1\leq i \leq m}\Gamma^{\beta}_{\alpha
        i}\dd u^i,
    \end{equation}
where $\omega^{\beta}_{\alpha}$ are real valued 1-forms on
$U$ and $\Gamma^{\beta}_{\alpha i}$ are smooth functions on U. Sometimes  instead of  $\omega$ it is used the symbol  $\theta$, see
 \eqref{BCON},  \eqref{ThEtA}, \eqref{nablaK}. 

If we   denote
    \begin{equation}\label{54}
      S=\left(\begin{array}{c}s_1\\\vdots\\s_q\end{array}\right),
      \quad \omega =\left(\begin{array}{ccc}\omega^1_1& \cdots& \omega^q_1\\
      \vdots &\ddots &\vdots\\ \omega^1_q&\cdots&
                                                \omega^q_q\end{array}\right),\end{equation}
    then \eqref{DSAL} can be written as
    \begin{equation}\label{DSom}
      DS=\omega \otimes S. 
    \end{equation}
    The matrix of real  one-forms $\omega$ in \eqref{54} is called the {\bf
      connection matrix} 
    or {\bf connection form} \cite{cf}.
    
  If $S'=(s'_1,\dots,s'_q)^t$ is another local  frame
  field on $U$ and 
  \begin{equation}\label{SSA}
    S'=AS,\quad
    A=\left(\begin{array}{ccc}a^1_1& \cdots& a^q_1\\
                          \vdots &\ddots &\vdots\\
                        a^1_q& \cdots&
                                       a^q_q\end{array}\right), \quad
                                   A\in M(q,\R), \end{equation}then
  \begin{equation}\label{DD1}
  DS' =\omega'\otimes S'.
\end{equation}
With \eqref{DSom}, \eqref{SSA}, we got
\[
  \dd A\otimes S+ADS=\omega'\otimes AS,
\]
and finally 
\begin{equation}\label{DD111}
    \omega'=\dd A\cdot A^{-1}+A\cdot\omega\cdot A^{-1}.
   \end{equation}
   \begin{deff}\label{def2}
     The $q\times q$ matrix of two-forms 
    \begin{equation}\label{OO}\Omega=\dd \omega-\omega\wedge\omega
    \end{equation} is called the {\bf  curvature
    matrix} of the connection $D$ on $U$  \cite[Definition 1.2 page 108]{ccl}. Sometimes the curvature
  matrix $\Omega$ is denoted $\Theta$, see \eqref{ThEtA} below.  
  \end{deff}
  In the new system of coordinate $S'$  \eqref{SSA} the curvature
  matrix $\Omega'$  is
  \begin{equation}
 \Omega'=A\cdot\Omega\cdot A^{-1}.
\end{equation}

The curvature matrix $\Omega$ defines a linear
transformation from $\Gamma(E)$ to $\Gamma(E)$. For 
\begin{equation}\label{sSUM}
\Gamma(E)\ni   s=\sum_{\alpha=1}^{q}\lambda^{\alpha}s_{\alpha}.\end{equation}
let $X,Y\in \got{D}^1$, and define
\begin{equation}
  R(X,Y)s:=     \sum_{\alpha=1}^q\lambda^{\alpha}<X\wedge Y,\Omega^{\beta}_{\alpha}>s_{\beta}|_p.    
\end{equation}
$R(X,Y)$ is called the {\bf curvature operator} of the connection $D$
and \cite[Theorem 1.3 page 109]{ccl}
\[
  R(X,Y)=D_XD_Y-D_YD_X-D_{[X,Y]}.
\]
We have \cite[page 117]{ccl}
\begin{equation}\label{511}  \Omega^j_i=\frac{1}{2}R^j_{ikl}\dd u^k\wedge \dd u^l, \quad
  R^j_{ikl}=\frac{\pa \Gamma^j_{il}}{\pa u^k} -\frac{\pa\Gamma^j_{ik}}{\pa u^l}
  +\Gamma^h_{il} \Gamma^j_{hk}-  \Gamma^h_{ik} \Gamma^j_{hl},
\end{equation}
and the $(1,3)$-tensor \[
  R=R^j_{ikl}\frac{\pa}{\pa u^j}\otimes \dd u^i\otimes \dd u^k\otimes
  \dd u^l
\]
is called the {\bf curvature tensor} of the affine connection $D$.

The section $s$ of a vector bundle $E$ is called {\bf parallel
  section} if \cite[page 116]{ccl}
\begin{equation}\label{PAR}
  Ds=0 .
\end{equation}
For \eqref{sSUM}, equation
\eqref{PAR} with \eqref{DXS}, \eqref{DOI} becomes
\begin{equation}\label{DdD}
  \dd
  \lambda^{\alpha}+\sum_{\beta=1}^q\lambda^{\beta}\omega^{\alpha}_{\beta}=0,
  \quad 1\leq\alpha\leq q.
\end{equation}
\begin{deff}\label{def3}
  Let $C$ be a  parametrized curve and $X$  a tangent
  vector field along $C$. If the section $s$ of the vector bundle $E$
  on $C$ satisfies  \begin{equation}\label{DXX}
    D_Xs=0,\end{equation}
  then $s$ is said {\bf parallel} along the curve $C$.

  If 
  \begin{equation}\label{TXX}
    X(t)=\sum_{i=1}^m\frac{\dd u^i}{\dd t}\left(\frac{\pa}{\pa
        u^i}\right)_{C(t)},\end{equation}
  and $s$ is given by \eqref{sSUM}
  then \eqref{DXX} reads 
  \begin{equation}
    \frac{\dd \lambda^{\alpha}}{\dd
      t}+\sum_{\beta, i}\Gamma^{\alpha}_{\beta i}\frac{\dd u^i}{\dd t}\lambda^{\beta}=0, \quad 1\leq\alpha \leq q.
    \end{equation}
    If any vector $v\in E_p$ is given at a point $p$ on $C$, then it
    determines uniquely  a vector field along
    $C$,  called {\bf parallel displacement of} $v$ along $C$.
  \end{deff}

   b) {\it Affine  connections.} A connection on the $m$-dimensional
  tangent vector bundle
  $T(M)$ is called {\bf affine connection} on $M$      \cite[\S~ 4-2]{ccl}.

  Formulas
  \eqref{DSAL} in the natural basis \{$s_i=\frac{\pa}{\pa u^i}, 1\leq
    i\leq m$\} in the local coordinate system $(U,u^i)$ on $M$
  became
  \begin{equation}\label{514}
    D s_i=\omega^j_i\otimes s_j=\Gamma^{j}_{ik}\dd u^k\otimes s_j,
  \end{equation}
  and the smooth functions $\Gamma^j_{ik}$  on $U$ are called {\bf coefficients} of
  the connection
  $D$ with respect to  the local coordinates $u^i$.

  In \cite[\S 4]{helg} the affine connection $D$
  is denoted $\nabla$,  $D_X$ defined at \eqref{DXS}  is denoted
  $\nabla_X$  and \eqref{DOI} becomes
\begin{subequations}\label{TREI}
  \begin{align}
 1.  & \nabla_{fX+gY}=f\nabla_X+g\nabla_Y, ~X,Y\in\got{D}^1,
            ~f,g\in C^{\infty}(M),\\
 2.  & \nabla_X(fY)=f\nabla_X(Y)+(Xf)Y, \quad Xf=<X,\dd f>.
  \end{align}
  \end{subequations}

  Equation \eqref{DXS} for $X= s_l$, $s=s_i$ becomes
  \cite[(1), page 27]{helg}
  \begin{equation}
    \nabla_{\frac{\pa}{\pa u^l}}\left(\frac{\pa}{\pa u^i}\right)
      =\Gamma^j_{il} \frac{\pa}{\pa u^j}.
    \end{equation}

  If $(W,w^i)$ is another coordinate system of $M$ and
  $s'_i=\frac{\pa}{\pa w^i}$, then on $U \cap W\not=0$  we have 
  \begin{equation}\label{520}
S'=J_{WU}\cdot S,\quad J_{WU}=\left(\begin{array}{ccc}
                                    \frac{\pa u^1}{\pa w^1}&\cdots
                                      &\frac{\pa u^m}{\pa w^1}\\
                                                           \vdots &\ddots &\vdots\\
                                   \frac{\pa u^1}{\pa w^m}& \cdots &
                                                                    \frac{\pa
                                                                    u^m}{\pa
                                                                    w^m} \end{array}\right),
                                                            \end{equation}
and   \eqref{DD1} becomes \cite[pages 113, 114]{ccl}
  \begin{subequations}\label{DD123}
  \begin{align}
    \omega' & = \dd J_{WU}\cdot J^{-1}_{WU}+J_{WU}\cdot\omega\cdot
              J^{-1}_{WU},\label{522a}\\
    \omega'^{j}_i&=\dd \left(\frac{\pa u^p}{\pa w^i}\right)\frac{\pa w^j}{\pa
                   u^p}+\frac{\pa u^p}{\pa w^i}\frac{\pa w^j}{\pa u^q}w^q_p,\label{522b}\\
                   \Gamma'^{j}_{ik}& =\Sigma_{pqr}\Gamma^q_{pr}\frac{\pa w^j}{\pa
                                     u^q}\frac{\pa u^p}{\pa
                                    w^i}\frac{\pa u^r}{\pa w^k}
                                     +\Sigma_p\frac{\pa^2u^p}{\pa w^i\pa
                                     w^k}\cdot \frac{\pa w^j}{\pa u^p},\label{522c}
  \end{align}
\end{subequations}
where $\omega^{\prime j}_i=\Gamma^{\prime j}_{ik}\dd \omega^k$.
Fomula \eqref{522c} appears also as \cite[ (2) page 27]{helg}.
$DX$  is a (1,1)-type tensor field on $M$, called {\bf the absolute differential} of $X$, \cite[page
114]{ccl}.

                               Let   $T^r_s$ be the tensor product of
                               the tangent and cotangent bundles. If
                               $t$  is an $(r,s)$-type tensor field,
                               then the image of $t$ under the
                               induced connection $D$ is an
                               $(r,s+1)$-type tensor field $Dt$.
                               \begin{deff}\label{def4}
                                 Let $C: u^i=u^i(t)$ be a
                                 parametrized curve on $M$ and
                                 \begin{equation}
                                   X(t)=x^i(t)\left(\frac{\pa}{\pa
                                       u^i}\right)_{C(t)}.
                                 \end{equation}
                                 $X(t)$ is {\bf parallel} along $C$
                                 if \begin{equation}\label{DXT}
                                   \frac{DX}{\dd t}=0.
                                   \end{equation}
                                   If the tangent vectors of a curve
                                   $C$ are parallel along $C$, then we
                                   call $C$ a {\bf self-parallel
                                     curve}, or a {\bf geodesic}.
                                   \eqref{DXT} is equivalent with the
                                   system of first order differential
                                   equations 
                                   \begin{equation}
                                     \frac{\dd x^i}{\dd
                                       t}+x^j\Gamma^i_{jk}\frac{\dd
                                     u^k}{\dd t}=0, \quad i=1,\dots, m.
                                 \end{equation}
  The tangent vector $X$ at any point of $C$ give rise to a parallel
  vector field, called the {\bf parallel displacement} of $X$ along
  the curve $C$.     With \eqref{TXX}, a geodesic curve $C$ should
  satisfy the system of second-order differential equations                           
  \begin{equation}\label{GEO}
    \frac{\dd^2u^i}{\dd t^2}+\Gamma^i_{jk}\frac{\dd u^j}{\dd
      t}\frac{\dd u^k}{\dd t}=0, \quad i=1,\dots,m.
    \end{equation}
\end{deff}
 c) {\it Riemannian connections.}
Following  \cite[Chapter 5]{ccl}, let us suppose that  $M$ is an $m$-dimensional
smooth manifold and let $G$ be  a symmetric covariant tensor of rank two on $M$. In a
local coordinate system $(U,u^i)$
\[
  G=g_{ij}\dd u^i\otimes \dd u^j,  \quad g_{ij}= g_{ji}.
\]
We have also
\[
  G(X,Y)=g_{ij}(p)X^iY^i, \quad X,Y\in T_p(M).
  \]
If the local coordinate system $u^i$ is changed to     $(u')^i$, then
$g_{ij}$ became 
\begin{equation}\label{schg}
  g'_{ij}=\frac{\pa u^k}{\pa {u'}^i}g_{kl}\frac{\pa u^l}{\pa {u'}^j}.
  \end{equation}

If $G$ is a smooth, everywhere nondegenerate symmetric tensor field of
rank 2, $M$ is called {\bf generalized Riemannian manifold}. If $G$ is
positive definite, then $M$ is called a {\bf Riemannian manifold}.  
\begin{deff}\label{def5}
  Suppose $(M,G)$ is an $m$-dimensional generalized
  Riemannian manifold and $D$ is an affine connection on $M$. If
  \begin{equation}\label{DG}
    DG=0,
  \end{equation}
  then $D$ is called a { \bf  metric-compatible connection} on
  $(M,G)$. 
\end{deff}
Condition \eqref{DG} is equivalent with
\begin{equation}\label{DDG}
  \dd g_{ij}=\omega^k_ig_{kj}+\omega^k_jg_{ik}, \quad \text{or~~~}\dd
  G=\omega\cdot G+G\cdot \omega^t .
\end{equation}
The geometric meaning of metric-compatible connections is that parallel
translations preserve the metric \cite[page 127]{ccl}.

Let \begin{equation}\label{TR}
  T^{j}_{ik}:=\Gamma^j_{ki}-\Gamma^j_{ik}.
  \end{equation}
  Then \begin{equation}
    T=T^{j}_{ik}\frac{\pa}{\pa u^j}\otimes\dd u^i\otimes \dd
      u^k\end{equation}
    is a $(1,2)$-type tensor, called the {\bf torsion tensor} of the
    affine connection $D$ and
    \[
      T(X,Y)=D_XY-D_YX-[X,Y]. \]
    If the torsion tensor of the affine
      connection $D$ is zero, then the connection is said to be  a {\bf
        torsion-free} connection.
      
   According to \cite[page 138]{ccl}   
\begin{Theorem} {\bf{(Fundamental Theorem of Riemannian Geometry)}} If $M$ is an $m$-dimensional generalized Riemannian manifold, then
  there exists an  unique torsion-free and metric compatible connection
  on $M$, called the {\bf Levi-Civita connection} on $M$, or the
  {\bf Riemannian connection} of $M$. 
\end{Theorem}
Let us denote \cite[page 138]{ccl}:
\[\Gamma_{ijk} :=g_{lj}\Gamma^l_{ik},\quad w_{ik}:=g_{lk}w^l_i.
\]
Also
\[\frac{\pa g_{ij}}{\pa u^k}=\Gamma_{ijk}+\Gamma_{jik}, 
\]or
\[
  g_{ij,k}:=\pa_kg_{ij}-\Gamma^l_{ki}g_{lj} -\Gamma^l_{kj}g_{il} =0,
\]
and
\begin{equation}\label{geot}
\Gamma_{ikj}\!=\!\frac{1}{2}\left( \frac{\pa g_{ik}}{\pa u^j} \!+ \!\frac{\pa g_{jk}}{\pa
    u_i}- \frac{\pa g_{ij}}{\pa u^k}\right), ~\Gamma^k_{ij}=\frac{1}{2}g^{kl}\left(\frac{\pa g_{il}}{\pa u^j}\!+\!\frac{\pa g_{jl}}{\pa
    u^l}\!-\!\frac{\pa g_{ij}}{\pa u^l}\right), ~ g_{ij}g^{jk}=\delta^k_i.
\end{equation}
$\Gamma_{ijk}$ ($\Gamma^i_{jk}$) is called Christofell's symbol of
first (respectively second) kind.
Also  $\Omega G$ is an antisymmetric tensor \cite[(98)]{CH67}:
 \begin{equation}\label{AS}
  \Omega G+G\Omega^t =0.
  \end{equation}
  
  Let us introduce \cite[pages 14, 142]{ccl}
  \[
    \Omega_{ij}:=\Omega^k_ig_{kj},
  \]
  and the skew symmetric tensor $ \Omega_{ij} $ is  given by
  \[
    \Omega_{ij}=\dd \omega_{ij}+\omega^l_i\wedge \omega_{jl}.
  \]
  If
  \[
R_{ijkl}:=\frac{1}{2}R^h_{ikl}g_{hj},
\]
then \[
  \Omega_{ij}=\frac{1}{2}R_{ijkl}\dd u^k\wedge \dd u^l,\quad
  R_{ijkl}=\frac{\pa \Gamma_{ijkl}}{\pa u^k}- \frac{\pa \Gamma_{ijk}}{\pa u^l}
  +\Gamma^h_{ik}\Gamma_{jhl} -\Gamma^h_{il}\Gamma_{jhk}.
\]
The covariant tensor of rank 4 $R_{ijkl}$ is called the
{\bf curvature tensor} of the generalized Riemannian manifold $M$ and
has the properties \cite[page 142]{ccl}:\\
1. $R_{ijkl}=-R_{jikl} =-R_{ijlk},$\\
2. $R_{ijkl}+R_{iklj} +R_{iljk}=0,$\\
3. $R_{ijkl}=R_{klij}.$\\
\subsubsection{Connections on complex manifolds}\label{CCM}
 Several definitions introduced in
 \S\ref{REALC}  for real connections  are adapted  to the  complex
 case  \cite{CH67}.
 Essentially, the transpose of $A^t, ~A\in M(q,\R)$ is replaced by  the
 conjugate transpose (or hermitian transpose)   $A^H$,  $ A\in
 M(q,\C)$, also denoted $A^*$, $A^{\dagger}$, $A^+$.

 In Definition \ref{def1}  $E$ is  taken a  {\bf complex fibre bundle of complex
dimension} $q$ \cite[\S 5]{CH67}, i.e. the fibres of $E$ are $\C^q$
and the structural group is $GL(q,\C)$. Then in  formula \eqref{DSAL}  the 1-forms
$\omega^{\beta}_{\alpha}$ are {\bf complex valued} and the parallel sections
defined in \eqref{PAR} are named {\bf horizontal lifts}.

Using the decomposition $T^*_M=(T_M^*)'+(T_M^*)"$, $D=D'+D"$, where
$D':\got{A}^0(E)\rightarrow \got{A}^{1,0}$ and
$D":\got{A}^0(E)\rightarrow \got{A}^{0,1}$. The connection $D$ is called
{\bf compatible with the complex structure} if  $D"=\bar{\pa}$
\cite[page 73]{GH}. 

Instead of
 Riemannian metric  in the real case,  a hermitian structure on
the  complex bundle $E$ is  introduced. A {\bf hermitian structure} on a complex vector
space $V$ is a complex-valued function $H(\xi,\eta)$, $\xi,\eta\in V$
such that \cite[page 9]{CH67}
\begin{subequations}\label{patru}
  \begin{align}
    1. &~~H(\lambda_1\xi_1+\lambda_2\xi_2,\eta)= \lambda_1H(\xi_1,\eta)+\lambda_2H(\xi_2,\eta),
         \quad \lambda_1,\lambda_2\in \C,~~\xi_1,\xi_2,\eta\in V;\\
  2. &~~  \overline{H(\xi,\eta)}=H(\eta,\xi),\quad \xi,\eta\in V.
  \end{align}
\end{subequations}
$H$  is called {\bf positive definite} if
\[
  H(\xi,\xi)>0, \quad \xi\not=0.
\]
A {\bf hermitian structure} on a complex bundle $\psi:~ E\rightarrow M$ is a $C^{\infty}$
field of positive definite hermitian structure in the fibers of $E$. A
complex vector bundle $\psi: E\rightarrow M$ with a hermitian structure is called {\bf
  hermitian vector bundle}. For 
every frame field $s$ the hermitian structures defines an hermitian
matrix
\[
  H_s=(H(s_i,s_j)) =H_s^H=\bar{H}_s^t,\quad 1\leq i,j\leq q.
\]
Under of change of coordinates \eqref{SSA}, we
have instead of \eqref{schg}
\[
  H_{s'}=A H_s\bar{A}^t, \quad A\in M(q,\C).
  \]
If condition \eqref{PAR} is fulfilled for the hermitian vector bundle,
 $s$ is called {\bf horizontal} and condition \eqref{DdD}
is also  fulfilled. The mapping $C\rightarrow \psi^{-1}(C)$ 
which assign to a point $t\in C$ $s=\lambda^{\alpha}(t)s_{\alpha}$ is
called {\bf horizontal lifting} if \eqref{DdD} is satisfied (parallel
vector field in Definition \ref{def3}). Equation \eqref{DD1} are
satisfied for $A\in M(q,\C)$  in \eqref{SSA}.

Instead of metric compatible connection on generalized Riemannian
manifolds, for hermitian vector bundle the connection is called {\bf admissible} if
$H(\xi,\eta)$
remains constant when $\xi,\eta$ are horizontal sections along
arbitrary curves. Instead of \eqref{DDG}, we have
\begin{equation}\label{DDG1}
  \dd G=\omega\cdot H+H\cdot {\omega}^H , 
\end{equation}
where
\[
  H(\xi,\eta)=\sum_{i,k=1}^nh_{ik}\xi^i\bar{\eta}^k,\quad \xi=\xi^is_i,~ \eta=\eta^is_i,~~
  h_{ik}=H(s_i,s_k). 
\]

The curvature matrix $\Omega$ is introduced as in Definition~\ref{OO}
and $\Omega H$ is skew-hermitian, while in the real case is
skew-symmetric as in \eqref{AS}. A frame field $s$ of a hermitian
vector bundle is called {\bf unitary} if $H_s=\mathbb{1}_q$ and the connection
and curvature matrix are both skew-hermitian.

Now let $M$ be a $m$-dimensional complex manifold and let 
$\psi:~ E\rightarrow M$ be a complex vector bundle over $M$ with fiber
dimension $q$.  If the transition functions $E$ are holomorphic, then
$E$ is a {\bf  holomorphic bundle}.
If $q=1$ we have a {\bf line
  bundle}.

Let $E\rightarrow M$ be a holomorphic vector bundle on the complex
manifold $M$ with hermitian  metric $h$ defined by the holomorphic
reper $f$. The dual bundle $E^*\rightarrow M$ has the metric
\[
  h_{E^*}(f^*)=h^{-1}_E(f) \] in the dual reper $f^*$, and
\[
  \theta_{E^*}=-\theta_E, \quad \Theta_{E^*}=-\Theta_E.
  \]
Above $\theta_E~ (\Theta_E) $ are the connection matrix (respectively
curvature matrix, denoted in \eqref{OO} with  $\Omega$))  of the holomorphic vector bundle $E$ and
\begin{equation}\label{ThEtA}
  \theta_E(f)=\pa \log H_E(f),\quad
  \Theta_E(f)=\bar{\pa}\theta(f)=\bar{\pa}\pa \log H_E(f).
  \end{equation}
\eqref{ThEtA} in the case of the projective space give the hermitian
metric
\[
  h_{[-1]}(f,f)=(f,f),
\]
and we find for the tautological line  bundle $[-1]$  on $\db{CP}^n$
\[
  \Theta _{[-1]}(f)=-\frac{(f,f)(\dd f,\wedge \dd f)-(\dd f,f)\wedge
    (f,\dd f)}{(f,f)^2}, 
\]
We have also the relations 
  \[
    H_{[-1]}=1+||w||^2, \quad
    H_{[1]}=(1+||w||^2)^{-1}.
    \]

    The curvature matrix of the hyperplane line  bundle on
    $\db{CP}^n$ is
\[
\Theta_{[1]}=\frac{(1+||w||^2)\sum \dd w_k\wedge \dd \bar{w}_k-
      \sum\bar{w}_k \dd w_k\wedge \sum w_k\dd \bar{w}_k}{(1+||w||^2)^2}.
\]

Let $(z^1_U\dots,z^m_U)$ (respectively $(z^1_V\dots,z^m_V)$) be local
coordinates in $U$ (resp. in $V$). The tangent bundle has as
transition functions
the Jacobian matrices similar with \eqref{520}
\begin{equation}
  J_{UV}=\frac{\pa (z^1_U,\dots,z^m_U)}{\pa (z^1_V,\dots,z^m_V)}.
\end{equation}

A section $s\in E$ is {\bf holomorphic} if its components relative to
a chart are holomorphic. A connection such that the connection matrix
is a matrix of 1-forms of type (1,0) relative to  holomorphic frame
field is called  a {\bf connection of type} (1,0). Formulae similar
with \eqref{520}, \eqref{SSA}, \eqref{522b} hold for 1-forms of type (1,0).

We extract from \cite[pages 5, 6]{SB21} Remark \ref{RM10} and some considerations.

In the convention  $\alpha, \beta, \gamma,\dots$
run from 1 to $n$, while A,B,C,$\dots$  run through $1,\dots,n, $  
 $\bar{1},\dots,\bar{n}$, \cite[p 155]{kn}, for an almost complex
 connection without torsion we have
the relations 
$$\Gamma^{\alpha}_{\beta\gamma}=\Gamma^{\alpha}_{\gamma \beta}; \quad
\bar{\Gamma}^{\alpha}_{\beta\gamma}=\Gamma^{\bar{\alpha}}_{\bar{\beta}\bar{\gamma}}$$
and all other $\Gamma^A_{BC}$ are zero. For a complex manifold of 
complex dimension  $n$ there are  $\frac{n^2(n+1)}{2}$     distinct $\Gamma$-s.

If we take into account the hermiticity condition \eqref{EQK}  in \eqref{KALP2}  of
the metric   and the  K\"ahlerian
restrictions \eqref{condH}, 
the non-zero Christoffel's  symbols  $\Gamma$  of the Chern connection (cf. e.g.  \cite[\S 3.2]{ball}, also Levi-Civita connection,  cf. e.g. 
 \cite[Theorem 4.17]{ball})  which appear in \eqref{geot} are determined by the equations,
see also  e.g.   \cite[(12) at p 156]{kn}
\begin{equation}\label{CRISTU}  h_{\alpha\bar{\epsilon}}\Gamma^{\alpha}_{\beta\gamma}=
\frac{\pa  h_{\bar{\epsilon}\beta}}{\pa z_{\gamma}} = \frac{\pa
  h_{\beta\bar{\epsilon}}}{\pa z_{\gamma}}, ~~\alpha, \beta, \gamma,
\epsilon =1,\dots,n, 
 m\end{equation}
and \[
\Gamma^{\gamma}_{\alpha\beta}=\bar{h}^{\gamma\bar{\epsilon}}\frac{\pa
  h_{\beta\bar{\epsilon}}}{\pa z_{\alpha}}=
h^{\epsilon\bar{\gamma}}\frac{\pa h_{\beta\bar{\epsilon}}}{\pa  z_{\alpha}},\quad
\text{where~ ~ ~ }
h_{\alpha\bar{\epsilon}}h^{\epsilon \bar{\beta}}=\delta_{\alpha\beta}.
\]

If a hermitian structure  $H$ is defined on the  holomorphic vector
bundle $\psi:~E\rightarrow M$, then it has an  uniquely
defined admissible connection of type (1,0)
given by
\begin{equation}\label{THeT}
  \omega =\pa H\cdot H^{-1}.
\end{equation}
If $q=1$, then  $H=(h), \Omega =(\Omega), h>0$ and 
\begin{equation}\label{PHO}
  \Omega=-\pa\bar{\pa}\log h. 
\end{equation}
$\Omega$ is closed and globally defined.

Chern \cite[page 45]{CH67} calles
\begin{equation}\label{1CH}
    \frac{1}{2\pi \ii}\Omega\end{equation}  the {\bf curvature form of the
  connection}. The holomorphic line  bundle $E\rightarrow M$ is said
to be {\bf positive} if $E$ can be given a metric
$h\in\C^{\infty}(M,E^*\times\bar{E}^*)$ such the first Chern class
$c_1(E)$  is positive.

\begin{Remark}\label{RM10}Let $M$ be a \Ka~ manifold with local complex  coordinates
   $(z^1,\dots,z^n)$. Let
   $\Gamma^i_{jk}(z)$ be the holomorphic Christofell's symbols in the
   formula  of
   geodesics 
\begin{equation}\label{geoD1} 
\frac{\dd ^2 z^i}{\dd t^2}+\Gamma^i_{jk}\frac{\dd z^j}{\dd t}\frac{\dd
z^k}{\dd t} =0. \quad i=1,\dots,n.
\end{equation}
Let us make in formula
   \eqref{geoD1} the change of variables $z^j=\xi^j+\ii \eta^j$,
   $\xi^i,~\eta^i\in \R$ 
  and let us introduce the notation $\xi^{j'}:=\eta^j$, $j':=j+n$,
  $j=1,\dots,n$.

  Then the geodesic equations \eqref{geoD1} in
  $(z^1,\dots,z^n)\in\C^n$ became geodesic equations in the
  variables $(\xi^1,\dots,\xi^n,\xi^{1'},\dots,\xi^{n'})\in\R^{2n}$

\begin{gather}\label{GR}
 \frac{\dd^2\xi^i}{\dd t^2}+\GM^i_{jk}\frac{\dd\xi^j}{\dd
  t}\frac{\dd\xi^k}{\dd t}+2\GM^i_{jk'}\frac{\dd\xi^j}{\dd
  t}\frac{\dd\xi^{k'}}{\dd t}+\GM^i_{j'k'}\frac{\dd\xi^{j'}}{\dd
  t}\frac{\dd\xi^{k'}}{\dd t} =0,\\
 \frac{\dd^2\xi^{i'}}{\dd t^2}+\GM^{i'}_{jk}\frac{\dd\xi^j}{\dd
  t}\frac{\dd\xi^k}{\dd t}+2\GM^{i'}_{j'k}\frac{\dd\xi^{j'}}{\dd
  t}\frac{\dd\xi^{k}}{\dd t}+\GM^{i'}_{j'k'}\frac{\dd\xi^{j'}}{\dd
  t}\frac{\dd\xi^{k'}}{\dd t} =0,
\end{gather}
where 
\begin{equation}\label{GAMELE}
\GM^{i}_{jk}=\GM^{i'}_{j'k}=-\GM^{i}_{j'k'}=\Re{\Gamma^i_{jk}};
~-\GM^{i}_{jk'}=\GM^{i'}_{jk}=-\GM^{i'}_{j'k'}=\Im{\Gamma^i_{jk}}
\end{equation}
and the real and imaginary parts of $\Gamma^i_{jk}$ are functions of
$(\xi,\xi')\in\R^{2n}$.

We find for the Berry phase \eqref{BCON}  the expression 
\begin{equation}\label{BFF}\varphi_B=  \sum_{i,j}(\varphi_B)^j_i\!\!=\!-\!\!\sum_{ij}\frac{\ii}{2}(\theta^j_i\!\!-\!\!\bar{\theta}^j_i)\!\!=\!\!\sum_{ij}\Im
  \Gamma^j_{ik}\dd\! \xi^k\!\!+\!\!\Re \Gamma^j_{ik}\dd\! \eta^k
 \! \!=\!\!\sum_{ij}\tilde{\Gamma}^{j'}_{ik}\dd\! \xi^k\!\!+\!\!\tilde{\Gamma}^j_{ik}\dd\! \eta^k.
  \end{equation}
\end{Remark}
\begin{proof}
  The first part is extracted from \cite[Remark 1]{SB21}. \eqref{BFF}
  is proved with \eqref{514}, \eqref{THeT}, \eqref{DSAL}.
  \end{proof}

If $M$ is a complex $m$-dimensional manifold, $M$ is called {\bf 
  hermitian} if a hermitian structure $H$ is given in its tangent
bundle $T(M)$. Then in a {\bf natural}  frame
field in local coordinates $z^1,\dots,z^m$
\[s_i=\frac{\pa}{\pa z^i},\quad
   h_{ik}=H(\frac{\pa}{\pa z^i},\frac{\pa}{\pa z^k}),\quad
   H={H}^H=(h_{ik}), \]
 and $H$ is positive definite hermitian.
 The \Ka~ form
 \[
  \hat{H}=\frac{\ii}{2}\sum h_{ik}\dd z^i\wedge\dd \bar{z}^k
 \]
 is a real-valued form of type (1,1) and the hermitian manifold $M$ is
 {\bf \Ka}~ if
 \[\dd \hat{H}=0.\]

 To a \Ka~manifold $(M,\omega_M)$  it is attached the triple
 $({L},h,\nabla_L)$ \cite{Cah,SBS},
  where  ${L}$ is a holomorphic (prequantum) line bundle on $M$,
  $h$ is a hermitian metric on ${L}$ (taken conjugate linear in
  the first argument), and
  $\nabla_L$ is a connection compatible with the metric and the \Ka~
  metric, \begin{equation}\label{nablaK}
    \nabla _L= \pa+ \theta_L + \bar{\pa}, \quad \theta_L= \pa \log \hat{h},\end{equation} where
  $\hat{h}$ is local representative of the hermitian metric $h$;  see
  also \cite[Proposition 3.21]{ball}, where the connection \eqref{nablaK}
  is called {\bf Chern connection}       \cite[page 31]{ball}.   With respect to local
  holomorphic coordinates of the manifold and with respect to a local
  holomorphic
  frame for the bundle the metric $h$ can be given as \cite[\S 2]{SB15a}
  \[
    h(s_1,s_2)(z)=\bar{\hat{h}}(z)\bar{\hat{s}}_1(z)\hat{s}_2(z),\quad \hat{h}(z)=(e_z,e_z)^{-1}.
    \]
  where $\hat{s}_i$ is a local representing function for the
  section $s)i$, $i=1,2$, and $\hat{h}$ is a locally defined   real-valued
  function on $M$.

     A  \Ka~ manifold $(M,\omega_M)$ is {\bf
   quantizable} \cite{Cah}  if in  local coordinates the curvature of
 $L$ \eqref{PHO} and   $\omega_M$ are related by the relation
 \begin{equation}\label{QUANT1}
   \Omega_L =-\ii \omega_M, \quad \text{or}\quad\Omega_L=-\pa\bar{\pa}\log {h}.
 \end{equation}
 Then $\omega_M$ is integral, i.e.  $ \omega_M\in H^2(M,\Z)$ and the first Chen class
 $c_1[{L}]=[\omega_M]$ \cite[page 141]{GH}. $M$ is called a
 {\bf Hodge} manifold for compact   \Ka~ manifolds.     Then $[]$
   \[
[]: \operatorname{Div}(M)\rightarrow (M,\mc{O}^*) 
  \]
  is a functorial homomorphism
  between the group of divisors and the Picard group of equivalence
  classes of $C^{\infty}$ holomorphic line bundles \cite[page 133]{GH}.

  Details for quantizable noncompact manifolds  can be find in
  \cite[\S 2, \S 5]{SB15a}, \cite[\S 2.1]{SB15}

 Since $\Omega H$ is of type (1,1), we get
 \[
   \Omega H=(\Omega_{ik}),\quad \Omega_{ik}=\sum_{j
     l}R_{iikj\bar{l}}\sigma^j\wedge\bar{\sigma}^l, 
 \]
 where $\sigma$ is a base dual with  holomorphic  base $s$.

 Note that in \cite{SB2000}, instead of quantization condition
 \eqref{QUANT1}, we used the condition 
 \begin{equation}\label{QUAN2}
   \Omega_L =-2\ii \omega_M.
 \end{equation}

 The {\bf holomorphic sectional curvature} at $(x,\xi)\in (U,T(M))$ is
 \[
   R(x,\xi)=2\sum
 R_{ijkl}\xi^i\bar{\xi}^k\xi^j\bar{\xi}^l/(\sum
 h_{ik}\xi^i\bar{\xi}^k)^2,
\]
The (1,1) type form 
\[\Phi:=\operatorname{Tr} \Omega.
\]
is called {\bf  Ricci form}.

\subsubsection{Holonomy}\label{HOLL}
Following \cite{BK}, 
let $\pi: L\rightarrow M$  be a line bundle over $M$ and let
$\got{L}=\got{L}(M) $ be  the set  of equivalence classes of line
bundles  over $M$.
$\got{L}$ has a group structure and this group is naturally isomorphic
with $H^2(M,\Z)$.
Then there exists  locally an unique $\alpha\in \got{D}^1$, the 
{\bf connection form}, such that  \cite[(1.4.3)]{BK}
\begin{equation}\label{NAB}
  \nabla_{\xi}s=2\pi \ii<\alpha, \xi>s.
\end{equation}
Comparison of \eqref{DXS}, \eqref{DSom} with \eqref{NAB}  gives the
correspondence $2\pi\ii \alpha\rightarrow \omega$  in the notation of
connection form in \cite{BK}  of the and  respectively  connection
matrix in \cite{ccl}.

The construction in \eqref{NAB} can be globalized \cite[Proposition 1.5.1]{BK}.
There exists an unique $\Omega\in \got{D}^2$ closed such that locally
\cite[Proposition 1.6.1]{BK}
\[
  \dd \alpha |_U=\Omega|_U.
  \]
The closed 2-form $\Omega$ is  called the  {\bf curvature} of
$(L,\alpha)$ \cite[page 104]{BK},
$\Omega= \operatorname {curv} (L,\alpha)$.

Let $\gamma: [a,b]\rightarrow M$ be a pice-wise smooth  curve and there is
a    linear isomorphism
\[ P_{\gamma}: L_{\gamma(a)}\rightarrow L_{\gamma(b)}\]
called {\bf parallel transport along} $\gamma$. Then the function
$Q:\Gamma\rightarrow \C\setminus \{0\}$
called {\bf scalar parallel transport function}, where
$\Gamma=\Gamma(M)$ is the set  of all piece-wise closed
curves on $M$
\[
  P_{\gamma}(s)=Q(\gamma)s, \quad ~~\forall s\in L_p'=L_p\setminus\{0\}.
  \]
   $Q(\gamma)=\exp(\ii \beta)$       is calculated with {\bf Stokes' formula} \cite{SB2000}
   \begin{equation}\label{ST}
     \beta=\oint_{\gamma}A_L=\oint_{\gamma}\ii
     \theta_L=\int_{\sigma}\dd A_L, 
   \end{equation}
where $\sigma$ is a 
the surface deformation of $\gamma$ \cite[page 108]{BK} and $\theta_L$
is defined in \eqref{nablaK}.

We get  \cite[Theorem
1.8.1 page 108]{BK} 
  \begin{equation}\label{Q1}
    Q(\gamma)= \exp (\ii \beta)=\exp(-\oint_{\gamma}\theta_L)=  \exp({ \ii \int_{\sigma}\omega_M}),
  \end{equation}
  In the convention \eqref{QUAN2} in \cite{SB2000},         \eqref{Q1} becomes
  \begin{equation}\label{Q2}
    Q(\gamma)= \exp (\ii
    \beta)=\exp(-\oint_{\gamma}\theta_L)=\exp(-\oint_{\sigma}\dd \theta_L)=
-\oint_\sigma\Theta_L=
    \exp({ 2\ii \int_{\sigma}\omega_M}),
  \end{equation}
  With Stokes' formula \eqref{ST} and \eqref{Q1}, \eqref{Q2} we get in
  the convention \eqref{QUAN2}
  \cite[(5.4)]{SB99}
  \[
    \dd A_L= 2\omega_M,
  \]
  or
  \[
    \dd A_L= \omega_M,
  \]
  in the convention \eqref{QUANT1}.

\subsection{Examples}\label{544}

Below for the Heiwsenberg-Weyl group,
$\db{CP}^1$ and $\mc{D}_1$ we follow \cite[\S~7.2.7]{SB19}.

$\bullet${\it {The  HW}} (Heisenberg-Weyl) group  is the group with the
3-dimensional real  Lie algebra isomorphic to the Heisenberg algebra
 $\got{h}_1\equiv
\g_{HW}=$ $<is1 +z{\mb a}^{\dagger}-\bar{z}{\mb a}>_{s\in \R, z\in\C}$,
 where the bosonic creation (annihilation)
operators ${\mb a}^{\dagger}$ (respectively ${\mb a}$) verify the canonical
commutation  relations 
$ [{\mb a},{\mb a}^{\dagger}]=1,$
and the action of the annihilation operator on the vacuum is
${\mb a}\e_0=0. $

Glauber's coherent states $e_z=e^{z\mb{a}^{\dagger}}e_0$ have the scalar
product
\[
(e_{\bar{z}},e_{\bar{z}'})=e^{z\bar{z}'}, \quad \omega=\ii \dd z\wedge
  \dd \bar{z}.
\]

$\bullet$  $\db{CP}^1=S^2=\text{SU}(2)/\text{U}(1)$. The
generators of $\text{SU}(2)$ verify the commutation relations
\[
  [J_0,J_{\pm}]=\pm J_{\pm};\quad [J_-,J_+]= -2J_0,
  \]
  and the finite dimensional representation of $\text{SU}(2)$ are  defined by the action on the
  extremal weight
  \begin{equation}\label{rep1}\mb{J}_+e_{j,-j}\not=0,\quad \mb{J}_-e_{j,-j}=0,\quad
    \mb{J}_0e_{j,-j}=-j e_{j,-j}, \quad  j=\frac{m}{2},~ m\in \db{N}.
         \end{equation}
If the CS vectors on $S^2$ are introduced as  \[e_z=e^{z\mb{J}_+}e_{j,-j}, \] then the   scalar product, \Ka~
two-form, Berry connection $A_B$ and $\dd A_B$ are 
 respectively 
 \begin{subequations}\label{551}
   \begin{align}
  (e_{\bar{z}},e_{\bar{z}'})& =(1+z\bar{z}')^{2j},\quad
                              \omega_{S^2}=2\ii j\frac{\dd
                              z\wedge \dd \bar{z}}{(1+|z|^2)^2},\label{547a}\\ A_B& =\ii j\frac{\bar{z}\dd
      z-z \dd \bar{z}}{1+|z|^2}, \quad \dd A_B=-2\ii  j\frac{\dd z\wedge
    \dd \bar{z}}{(1+|z|^2)^2}.
   \end{align}
   \end{subequations}

   $\bullet$  
  $\mc{D}_1=\text{SU}(1,1)/\text{U}(1)$. The generators of
  $\text{SU}(1,1)$ verify the commutation relations
  \[
    [K_0,K_{\pm}]=\pm K_{\pm},\quad [K_-,K_+]=2K_0.
  \]
  The positive holomorphic discrete
  series representation corresponds to the action on the extremal
  weight
  \begin{equation}\label{rep2}
    \mb{K}_+e_{k,k}\not=0,\quad \mb{K}_-e_{k,k}=0,\quad
    \mb{K}_0e_{k,k}= k e_{k,k}.
         \end{equation}

  If
  \[
    e_z=e^{z\mb{K}_+}e_{k,k}, \quad |z|<1,\quad
k=1,\frac{3}{2},2,\frac{5}{2},\dots,
\]
then the scalar product, \Ka~two-form, Berry connection $A_B$ and $\dd
A_B$ are
respectively 
\begin{subequations}\label{552}
  \begin{align}
  (e_{\bar{z}},e_{\bar{z'}})& =(1-z\bar{z}')^{-2k},\quad
                              \omega_{\mc{D}_1} = 2\ii k
                              \frac{\dd z\wedge \dd \bar{z}}{(1-|z|^2)^2},\label{549a}\\
    A_B&=\ii  k\frac{\bar{z}\dd
      z-z \dd \bar{z}}{1-|z|^2}, \quad \dd A_B= -2\ii k\frac{\dd z\wedge\dd\bar{z}}{(1-|z|^2)^2}.\label{549b}
  \end{align}
\end{subequations}
We remark that \eqref{549b} has already appeared in \eqref{ThEETA}.

Let us denote
with $J$  the   extremal weight of the representation  \eqref{rep1}
(\eqref{rep2})  i.e. $J=-2j$ $(J=2k)$ of $\text{SU}(2)$ (respectively, $\text{SU}(1,1)$).
Then  we write \eqref{552} as
\begin{subequations}
  \begin{align}
(e_{\bar{z}},e_{\bar{z'}})& =(1\pm z\bar{z}')^{-2J},\quad \omega_{S^2,\mc{D}_1} = \mp
                            2 \ii  J 
                              \frac{\ \dd z\wedge \dd \bar{z}}{(1\pm|z|^2)^2},\\
    A _B&=\mp \ii  J\frac{\bar{z}\dd z-z\dd\bar{z}}
      {1\pm |z|^2}, \quad \dd A_B= \pm 2\ii J\frac{\dd z\wedge\
        \dd\bar{z}}{(1\pm |z|^2)^2},
  \end{align}
  \end{subequations}
where + (-)   corresponds to the compact
(respectively noncompact) manifold $\db{CP}^1$ (respectively
$\mc{D}_1$).

We also have on $S^2$ and $\mc{D}_1$
\[
  \omega_{S^2,\mc{D}_1}=-\dd A_B.
  \]

  $\bullet$  The Complex Grassmann Manifold
  $G_n(\C^{m+n}) =G_c/K$  and its noncompact
  dual $G_n/K$, $G_c=   \text{SU}(n+m)$, $G_n= \text{SU}(n,m)$, $K=
  S(\text{U}(n)\times \text{U}(m))$.  See \cite[page 452, Type  AIII]{helg}.
Let $Z\in M(n,m,\C)$ be Pontrjagin's  coordinates for the compact
(noncompact) Grassmann manifold. Then the scalar
product of two coherent state vectors is \cite[ (6.26)]{SB95} 
\begin{equation}\label{DOII}
 (e_{Z'},e_{Z})=\det(\un+\epsilon ZZ'^+)^{\epsilon}, \text{~~where~~}\epsilon = ~- ~(+) \text{~~for~~}
 X_n~ (X_c), 
  \end{equation}
 for the particular dominant weight \cite[\S 6, Remark 4, (6.25)]{SB95} 
\[
  j=j_0= (\underbrace{1,\dots,1}_n,\underbrace{0,\dots,0}_m).
\]
Formula \eqref{DOII} for $\epsilon=1$ for the complex Grassmann
manifold appears in \cite[(3.20)]{SBCAG89},
\cite[(3.6)]{SBCAG91}, using the technique   
to realize $G_n(\C^{m+n})$ as Slater determinat manifold \cite{MORSE}. 

With formula \eqref{BCON}  applied to \eqref{DOII}  and the relation
\[
  \frac{\dd }{\dd t}\det A=\det A \operatorname{Tr}(A^{-1}\frac {\pa
      A}{\pa t}), \quad A\in M(n,\C),
  \]
  we get
  \begin{subequations}\label{DOII3}
    \begin{align}
  A_B=& \frac{\ii }{2}\operatorname{Tr}[(\dd Z Z^+-Z \dd Z^+)(\un
      +\epsilon ZZ^+)^{-1} ],\label{556a} \\
      \dd A_B=& -\ii   \operatorname{Tr}[\dd
  Z(\mathbb{1}_m+\epsilon Z^+Z)^{-1}\wedge \dd Z^+(\un +\epsilon
              ZZ^+)^{-1}].\label{556b}
    \end{align}
  \end{subequations}
  Formula \eqref{556a} of the Berry connection on the complex Grassmann manifold  $G_n(\C^{n+m})$
corresponding
to the scalar product  \eqref{DOI} was obtained in  \cite[(5.17)]{sbcag}
\begin{equation}
  A_B=\frac{\ii}{2}\operatorname{Tr} [(Z^+\dd Z-\dd Z^+Z)(\mathbb{1}_m +Z^+Z)^{-1}],
\end{equation}
where $Z\in  M(m,n,\C)$, which is identical with \eqref{DOII} if
$m\leftrightarrow n$.

Now we apply formula \eqref{KALP2} for $G_c/K$ and $G_n/K$ for the
scalar product \eqref{DOII} and we get
\begin{equation}\label{OMM} \omega= \ii  \operatorname{Tr}[(\un+\epsilon ZZ^+)^{-1}\dd Z\wedge
(\mathbb{1}_m+\epsilon Z^+Z)^{-1}\dd Z^+].\end{equation}

Comparing  \eqref{556b} and \eqref{OMM}, it follows that 
 on $G_c/K$ and $G_n/K$  we have the relation 
 \begin{equation}\label{DABOM}
  \dd A_B= - \omega. 
  \end{equation}

Formula \eqref{556b}  for $\epsilon=1$ on $G_n(\C^{n+m})$  
appears  in \cite[page 1005]{sbcag}, where it was emphasized that it
is  the explicit realization
of the two-form $V$ of Simon \cite{BS}.  

We recall that the invariant metric on $X_c$ ($X_n$) \cite[(6.10)]{SB95} is
\[\dd s^2=k \operatorname{Tr}[\dd
  Z(\mathbb{1}_m+\epsilon Z^+Z)^{-1}\dd Z^+(\un+\epsilon ZZ^+)^{-1}],
  \quad Z\in M(n,m,\C).
\]
The equation of geodesics on $X_{c,n}$  \cite[(6.13)]{SB95}
\[
  \frac{\dd^2}{\dd t^2}
 -2\epsilon\frac{\dd Z}{\dd t}Z^+(\un +\epsilon ZZ^+)^{-1}\frac{\dd
   Z}{\dd t}=0 \]
has the solution $Z=Z(tB)$
\[
  Z=Z(B)=B\frac{\text{ta}\sqrt{B^+B}}{\sqrt{B^+B}},\quad  B \in M(n,m,\C),
  \]
  with the initial condition $Z(0)=B$, where $\text{ta}= \tan$
  ($\tanh$) for $X_c$ (respectively, $X_n$) .

  Equations \eqref{556a}  (\eqref{556b})
  for $\frac{\text{SU}(m+n)}{\text{SU}(m)\times\text{SU}(n)}$
 (respectively  $\frac{\text{SU}(m,n)}{\text{SU}(m)\times\text{SU}(n)}$) of  $\varphi_B = \oint A_B$, $\dd A_B $
 have been obtained in \cite[ (62), (63)]{sbl}.  

In \cite{SB2000} in the relation $A_L=\ii \theta_L$ the connection
matrix $\theta_L$ corresponds to the Hermitian metric on the dual
of the tautological line bundle on the Grassmann manifold
$\hat{h}_L(Z)=\det(\un+ZZ^+)^{-1}$. The Berry connection which
corresponds to $A_L$ is \cite[(5.2)]{SB2000}
\begin{equation}\label{ABERRY}  A_B=\frac{\ii}{2}\operatorname{Tr}[(\dd ZZ^+-Z\dd Z^+)(\un + ZZ^+)^{-1} ].
\end{equation}
The corresponding two-form on  $G_n(\C^{n+m})$ is  \cite[(5.3)]{SB2000}
\[
  \omega=\frac{\ii}{2}\operatorname{Tr}[\dd
  Z(\mathbb{1}_{m}+Z^+Z)^{-1}\wedge \dd Z^+(\un +ZZ^+)^{-1}].  \]
We reported \cite[(5.5)]{SB99}  the holomorphic   connection on the compact 
 complex Grassmann manifold and its non-compact dual
\[
  A_L=\ii \operatorname{Tr}[\dd  ZZ^+(\mathbb{1}_n+\epsilon ZZ^+)^{-1}]
\]
corresponding to the hyperplan line bundle,  the dual  of the tautological line bundle $L$ on the
Grassmannian and its non-compact dual \cite[(5.6)]{SB99} 
\begin{equation}\label{HL}
  \hat{h}_L=\det(\un+\epsilon ZZ^+)^{-\epsilon},
\end{equation}
the \Ka~ two-form is \cite[(5.7)]{SB99},
where $\mathbb{1}_n$ should be replaced with $\mathbb{1}_m$
\[
  \omega=\frac{\ii}{2}\operatorname{Tr}[\dd
  Z(\mathbb{1}_{m}+\epsilon Z^+Z)^{-1}\wedge \dd Z^+(\un +\epsilon
  ZZ^+)^{-1}],  \]
while the Berry connection is \cite[(5.8)]{SB99}
\begin{equation}\label{ABERRY2}  A_B=\frac{\ii}{2}\operatorname{Tr}[(\dd
ZZ^+-Z\dd Z^+)(\un + \epsilon
ZZ^+)^{-1} ].
\end{equation}

$\bullet$  The complex projective space $\db{CP}^n$ and
  $\db{CP}^{n,1}$.
We recall that the ray space is defined as 
\[
  \db{CP}^n=\db{P}(\C^n)=S^{2n+1}/\approx =
  \frac{\text{SU}(n+1)}{S(\text{U}(n)\times\text{U}(1))}, \quad\text{where~~~}
  x\approx y\leftrightarrow x=\lambda y, \lambda\in\text{U}(1).
\]

Also we have
\[
  \db{CP}^n=\C^{n+1}/\approx \quad\text{where~~~}
  x\approx y\leftrightarrow x=\lambda y, \lambda\in\C^*=\C\setminus 0.
\]

Also we have the dual space
\[
  \db{CP}^{n,1}= \frac{\text{SU}(n,1)}{S(\text{U}(n)\times\text{U}(1))}.
  \]

    In \cite[(24)]{SBCAG91}, \cite[(5.8)]{SBCAG912}    we have proved that for
    $\db{CP}^N$      and $M(1,N) \ni Z  =(Z_1,\dots,Z_N)$ that
    \begin{subequations}\label{563}
      \begin{align}
      (\mathbb{1}_N+Z^+Z)^{-1}_{\alpha,\beta} & =(\mathbb{1}_N+|Z|^2)^{-1}\left\{\begin{array}{l}
                                                               1+|Z|^2-|Z_{\alpha}|^2,\quad \alpha=\beta\\
    -\bar{Z}_{\alpha}Z_{\beta}, \quad\quad\quad\quad\quad
                                                                                   \alpha\not=
                                                                                \beta    \end{array}\right.\\
                                                                           & =\frac{(1+|Z|^2
                                                                             )\delta_{\alpha\beta}-\bar{Z}_{\alpha}Z_{\beta}}{1+|Z|^2} ,
                                                                             \end{align}
                                                                               \end{subequations}
      where \[
        |Z|^2=|Z_1|^2+|Z_2|^2+\cdots|Z_N|^2.\]

      We recall  that    the Fubini-Studdy metric  on
      $\db{CP}^N$  is 
 \begin{equation}\label{FB}
 g_{\alpha\beta}= \frac{(1+|Z|^2)\delta_{\alpha\beta}-\bar{Z}_{\alpha}Z_{\beta}}{(1+|Z|^2)^2}.
  \end{equation}

For  $\db{CP}^{N,1}$ we have \eqref{5631} replacing \eqref{563} for $\db{CP}^N$ and
we can write together  
$\db{CP}^N$ ($\epsilon =1$) and  $\db{CP}^{N,1}$ ($\epsilon =-1$) as
\begin{subequations}\label{5631}
\begin{align}
      (\mathbb{1}_N+\epsilon Z^+Z)^{-1}_{\alpha,\beta} &
                                                         =(\mathbb{1}_N+\epsilon |Z|^2)^{-1}\left\{\begin{array}{l}
                                                              ( 1+\epsilon |Z|^2),\quad \alpha=\beta\\
                                                                                               ( -\epsilon\bar{Z}_{\alpha}Z_{\beta}),
                                                               \quad\quad\quad\quad\quad
                                                               \alpha\not=
                                                                                \beta    \end{array}\right.\\
                                                                           &
                                                                             =\frac{(1+\epsilon
                                                                           |Z|^2)\delta_{\alpha\beta}-\epsilon\bar{Z}_{\alpha}Z_{\beta}}{1+\epsilon
                                                                             |Z|^2}.
                                                                             \end{align}
                                                                               \end{subequations}
                                                                      
 Now we introduce \eqref{5631} into \eqref{OMM} and we get
 \begin{equation}\label{FBD}
  \omega_{\db{CP}^n, \db{CP}^{n,1}}(Z,\bar{Z}) =\ii g_{\alpha\beta}\dd
  Z_{\alpha}\wedge\dd\bar{Z}_{\beta}=\ii 
\frac{(1+\epsilon
  |Z|^2)\delta_{\alpha\beta}-\epsilon\bar{Z}_{\alpha}{Z}_{\beta}}{(1+\epsilon`|Z|^2)^2}
\dd Z_{\alpha}\wedge\dd\bar{Z}_{\beta},
\end{equation}
which is the Fubini-Study \Ka~ two - form \eqref{FB} on $\db{CP}^n$
(respectively, its non-compact dual $\db{CP}^{n,1}$, sometimes called
the hyperbolic space and denoted $\db{H}^n$ \cite[page 67]{JAW} ).
The condition on $Z$ for $\epsilon=-1$,   $(n,m)=(1,n)$, is \cite[(23)]{sbl}
   \[   1-|Z|^2>0.\]
 If equation \eqref{FBD} we put  $n=1$ we regain \eqref{547a} for
 $\epsilon =1$  ( $\epsilon=-1$)
 $S^2$,  \eqref{547a}, $j=\frac{1}{2}$  ($\mc{D}_1$,\eqref{549a}), respectively, $k=\frac{1}{2}$).
                                                                                                                         
     We recall also the definition of the {\bf tautological line bundle} $[-1]$
      \begin{equation}\label{TAUT} [-1]= \{(z,v)\in\db{P}(\C^{n+1})\times \C^{n+1}|~v\in [z],~ [z]
        \text{~is the line bundle defined by z}\},
      \end{equation}
      associated to the transition functions
      \[
        g_{ij}([z])=\{\frac{z_i}{z_j}\}, \quad [z]\in U_i\cap U_j.
      \]
      $[1]$ is the {\bf hyperplane bundle}, the dual of the
      tautological bundle $[-1]$.

        The tautological line bundle does not have global holomorphic
        sections not identically 0.

        Formula \eqref{556a} particularised for $\db{CP}^n$ and 
        $\db{CP}^{n,1}$ reads
        \begin{equation}\label{ADP}
          A_B=\frac{\ii }{2}\frac{\bar{z}_i\dd z_i-z_i \dd
    \bar{z}_i}{1+\epsilon |z|^2},\end{equation}
which for $\epsilon =1$  is \cite[(10)]{DP}.
If we differentiate \eqref{ADP}, we got
\begin{equation}\label{dADP}
       \dd A_B=    \ii
              \frac{\epsilon \bar{z}_iz_j-(1+\epsilon
            |z|^2)\delta_{ij}}{(1+\epsilon |z|^2)^2}\dd z_i \wedge \dd
          \bar{z}_j,\end{equation}
        which is \eqref{556b} particularised for $m=1$ with \eqref{5631}.

        Applying \eqref{PHO} to \eqref{HL} for $\db{CP}^n, \db{CP}^{n,1}$, we get
        \[
          \Omega_{\db{CP}^n, \db{CP}^{n,1}}=
            \frac{-\epsilon
              \bar{Z}_{\alpha}Z_{\beta}+(1+\epsilon|Z|^2)\delta_{\alpha\beta}}{(1+\epsilon\Z|^2)}\dd
            Z_{\alpha}\wedge \dd \bar{Z}_{\beta}= -\ii \omega_{\db{CP}^n, \db{CP}^{n,1}},
      \]
      which is the  quantizablity condition \eqref{QUANT1}
       of $\db{CP}^n$ (respectively $\db{CP}^{n,1}$).

\subsection*{Acknowledgements}
This research was conducted in the framework of the
ANCS project program PN 19 06 01 01/2019.

\end{document}